\documentclass[reqno, 10pt]{amsart}
\usepackage{amssymb,amsmath,amsfonts,bm}
\usepackage{dsfont}
\usepackage{epsfig}
\usepackage{graphicx}
\usepackage{esint}
\usepackage{enumerate}
\usepackage{verbatim}
\usepackage{color}
\usepackage{caption}
\usepackage{geometry}
\usepackage{enumitem}
\usepackage{hyperref}



\def\R{{\mathbb R}}
\def\N{{\mathbb N}}
\def\L{{\mathcal L}}
\def\e{{\varepsilon }}
\def\l{{\langle }}
\def\r{{\rangle }}
\def\A{{\mathcal A}}
\def\1{{\mathds{1}}}

\def\K{{\tild{K}_1}}
\def\2{{\tild{K}_2}}
\def\3{{\tild{K}_3}}
\def\4{{\tild{K}_4}}
\def\5{{\tild{K}_5}}
\def\6{{\tild{K}_6}}
\def\7{{\tild{K}_7}}
\def\8{{\tild{K}_8}}

\def\k{{\bar{K}}}

\date{} 

  \DeclareMathOperator*{\sign}{sgn}
\DeclareMathOperator\dist{dist}

\theoremstyle{plain}
\newtheorem{theorem}{Theorem}[section]
\newtheorem{definition}[theorem]{Definition}
\newtheorem{proposition}[theorem]{Proposition}

\newtheorem{lemma}[theorem]{Lemma}

\newtheorem{remark}[theorem]{Remark}
 \newcommand{\tild}{\widetilde}

\numberwithin{theorem}{section}
\numberwithin{equation}{section}
\numberwithin{figure}{section}

\let\oldtocsection=\tocsection
 
\let\oldtocsubsection=\tocsubsection
 
\let\oldtocsubsubsection=\tocsubsubsection
 
\renewcommand{\tocsection}[2]{\hspace{0em}\oldtocsection{#1}{#2}}
\renewcommand{\tocsubsection}[2]{\hspace{1em}\oldtocsubsection{#1}{#2}}
\renewcommand{\tocsubsubsection}[2]{\hspace{2em}\oldtocsubsubsection{#1}{#2}}

\begin{document}
\bibliographystyle{plain}

\parskip=4pt

\vspace*{1cm}
\title[Moment estimates and well-posedness of the  binary-ternary Boltzmann equation]
{Moment estimates and well-posedness of the  binary-ternary Boltzmann equation}

\author[Ioakeim Ampatzoglou]{Ioakeim Ampatzoglou}
\address{Ioakeim Ampatzoglou,  
Courant Institute of Mathematical Sciences, New York University}
\email{ioakampa@cims.nyu.edu}

\author[Irene M. Gamba]{Irene M. Gamba}
\address{Irene M. Gamba,  
Department of Mathematics, The University of Texas at Austin}
\email{gamba@math.utexas.edu}
\author[Nata\v{s}a Pavlovi\'{c}]{Nata\v{s}a Pavlovi\'{c}}
\address{Nata\v{s}a Pavlovi\'{c},  
Department of Mathematics, The University of Texas at Austin.}
\email{natasa@math.utexas.edu}

\author[Maja Taskovi\'{c}]{Maja Taskovi\'{c}}
\address{Maja Taskovi\'{c},  
Department of Mathematics, Emory University}
\email{maja.taskovic@emory.edu}
\begin{abstract}
In this paper, we show generation and propagation of polynomial and exponential moments, as well as global well-posedness of the homogeneous binary-ternary Boltzmann equation. We also show that the co-existence of binary and ternary collisions yields better generation properties and time decay, than when only binary or ternary collisions are  considered. To address these questions, we develop for the first time angular averaging estimates for ternary interactions. This is the first paper which discusses this type of questions for the binary-ternary Boltzmann equation and opens the door for studying moments properties of gases with higher collisional density. 
\end{abstract}
\maketitle
\tableofcontents

\section{Introduction}\label{intro}

\noindent 
The Boltzmann equation (cf. \cite{ce88, ceilpu94, bo64, ma867}), given by
\begin{equation} \label{eq-bte}
	\partial_t f  + v \cdot \nabla_{x} f \, = \, Q_2(f,f),  
\end{equation} 
for the position $x \in \R^d$, the time $t\in\R^+$ and velocity $v\in\R^d$
describes the evolution of the density $f(t,x,v)$ of gas particles,
where $Q_2(f,f)$ is a quadratic integral operator that expresses the change of $f$ due to instantaneous binary collisions of particles. The exact form of $Q_2$ depends on the type of interaction between particles. However, when the gas is dense enough, higher order interactions become relevant too. For example, in the case of a colloid\footnote{which is a homogeneous non-crystalline substance consisting of 
ultramicroscopic particles of one substance dispersed through a second substance.} 
it was noted by Russ - Von-G\"{u}nberg \cite{ruvo02} that multi interactions among particles significantly contribute to the grand potential of a colloidal gas and are modeled by a sum of higher order interaction terms. Motivated in part by this observation, in a sequence of works \cite{ampa21, ampa20, am20} Ampatzoglou and Pavlovi\'{c} proposed a toy model for a non-ideal gas of the form: 

\begin{equation}\label{intro:generalized Boltzmann}
\begin{cases}
\partial_t f+v\cdot\nabla_xf=\displaystyle\sum_{k=2}^m Q_k(\underbrace{f,f,\cdots,f\,}_\text{$k$-times}),\quad (t,x,v)\in(0,\infty)\times\mathbb{R}^{d}\times\mathbb{R}^d,\\
f(0,x,v)=f_0(x,v),\quad (x,v)\in\mathbb{R}^{d}\times\mathbb{R}^d, 
\end{cases}
\end{equation}
where $Q_k(f,f,\cdots,f)$ denotes the $k$-th order collisional operator, and $m\in\mathbb{N}$ is the parameter 
that reflects the order of the highest collision allowed. 
Also equations similar to \eqref{intro:generalized Boltzmann} were studied for Maxwell molecules in the works of Bobylev, Cercignani and 
Gamba in \cite{bogace09, bogace08}. 

In this paper we continue our analysis of \eqref{intro:generalized Boltzmann} with $m=3$.  We refer to this equation as the binary-ternary Boltzmann equation. While in our previous work \cite{amgapata22} we established global well-posedness near vacuum for the binary-ternary equation, the paper at hand focuses on understanding behavior of moments for the probability density associated with the {\it spatially homogeneous} version of the binary ternary equation which we write as: 
\begin{equation}\label{intro binary-ternary equation}
\begin{cases}
\partial_t f=Q[f]\quad (t,v)\in(0,+\infty)\times\mathbb{R}^d,\\
f(t=0)=f_0,
\end{cases}
\end{equation} 
where the binary-ternary collision operator is given by
\begin{align}\label{binary-ternary operator}
Q[f]:=Q_2(f,f)+Q_3(f,f,f),
\end{align}
with operators $Q_2$ and $Q_3$ given in Section \ref{sec-vocabulary}. Importantly, we show that the homogeneous binary-ternary equation \eqref{intro binary-ternary equation}
is ``better behaved" compared to the homogeneous version of the Boltzmann equation \eqref{eq-bte},
namely
\begin{equation}\label{binary equation}
\partial_t f=Q_2(f,f), 
\end{equation} 
or the homogeneous version of the purely ternary equation \begin{equation}\label{ternary equation}
\partial_t f=Q_3(f,f). 
\end{equation}

In order to describe what we mean by the phrase ``better behaved", we recall definitions of polynomial and exponential moments of a measurable function and present a brief summary of results on generation and propagation in time of moments associated with solutions to \eqref{binary equation}.

Given $k\geq 0$, we define the $k$-th order polynomial moment of a measurable non-negative function $f(t,v)$  as
\begin{equation}\label{definition of k-moment}
m_k[f](t):=\int_{\mathbb{R}^{d}}f(t,v) \, \langle v\rangle^k\,dv,
\end{equation}
where $\l v \r = \sqrt{1+|v|^2}$. \footnote{Note that moments are increasing with respect to their order i.e. $m_{k_1}[f]\leq m_{k_2}[f]$, when $k_1\leq k_2$.}
Analogously, given $s >0$ and $z>0$, we define exponential moment of order $s$ and rate $z$ of a measurable non-negative function $f(t,v)$ as 
\begin{align}\label{definition of exponential moment}
E_s[f](t,z) := \int_{\mathbb{R}^d} f(t,v) \, e^{z \l v \r^s} \, dv.
\end{align}
When it is clear from the context which function we are referring to, we will just write $m_k$ instead of $m_k[f]$, and $E_s$ instead of $E_s[f]$.

For each of these two types of moments, one can consider proving either generation or propagation of moments in time. In particular, by intorducing
$$\phi(v)=
\begin{cases}
\langle v\rangle^k & \mbox{in the case of polynomial moments},\\
e^{z \l v \r^s} & \mbox{ in the case of exponential moments},
\end{cases}$$
 by generation in time of $\phi$-moments one means
\begin{align}\label{def-generation}
	\int_{\mathbb{R}^d} f_0(v) \,\l v \r^2 \, dv < \infty
	\quad \Rightarrow \quad 
	\int_{\mathbb{R}^d} f(t,v) \,\phi(v) \, dv < \infty \,\, \mbox {for all } t >0.
\end{align}
On the other hand,  by propagation in time of $\phi$-moments one means
\begin{align}\label{def-propagation}
	\int_{\mathbb{R}^d} f_0(v) \,\phi(v) \, dv < \infty
		\quad \Rightarrow \quad 
		\sup_{t\geq 0}\,\int_{\mathbb{R}^d} f(t,v) \,\phi(v) \, dv < \infty.
\end{align}

Moment estimates for the spatially homogeneous  Boltzmann equation \eqref{binary equation} have been studied for decades. In the setting that is similar to ours - the case of an integrable angular kernel $b_2$, sometimes referred to as the cutoff case,  with  variable hard potentials, that is when the potential rate $\gamma_2$  of the angular kernel is strictly positive (see \eqref{cross-section 2} for the definition of $b_2$ and $\gamma_2$) - 
 it is known that (cf. \cite{miwe99, alcagamo13}) generation of polynomial moments, \eqref{def-generation} with $\phi(v) = \l v \r^k$, as well as propagation of polynomial moments, \eqref{def-propagation} with $\phi(v) = \l v \r^k$,  hold for any $k>2$, and once moments become finite they remain uniformly bounded in time. Similarly,  (cf. \cite{alcagamo13}) exponential moment of order $\gamma_2$ is generated instantaneously, i.e. \ref{def-generation} holds with $\phi(v) = e^{z \l v \r^{\gamma_2}}$ for some $z>0$. Also, exponential moments of orders $s \in (0, 2]$ propagate in time i.e , \eqref{def-propagation} with $\phi(v) = e^{z \l v \r^{\gamma_2}}$ holds for $s \in (0,2]$. For more results regarding moment estimates for the spatially homogeneous Boltzmann equation and its variations, see \cite{alga07, allo10, ar72, ar72part2, bo97, bogapa04, brei16, de93, el83, fo21, gapavi09, gapa20, gapa20preprint, lumo12, mimo06,mowaya16,  mo06,  pata18, po62, styu14, taalgapa18, we97}.

One of the main results of this paper demonstrates that the spatially homogeneous binary-ternary Boltzmann equation \eqref{intro binary-ternary equation} behaves better that the classical homogeneous Boltzmann equation \eqref{binary equation} in the following sense (see Theorem \ref{exponential moments theorem}): 
\begin{itemize}
\item Adding the ternary operator $Q_3$ to the classical Boltzmann equation \eqref{binary equation} can improve the order of the exponential moment that is generated. In other words, the binary-ternary Boltzmann equation generates higher order exponential than the binary Boltzmann equation \eqref{binary equation} or ternary Boltzmann equation \eqref{ternary equation} alone.
\item Generation of exponential moments happens even if one of $\gamma_2, \gamma_3$ is zero (corresponding to the Maxwell molecules case) as long as the other one is strictly positive. This is in contrast with the binary Boltzmann equation for the Maxwell molecules, in which case generation of exponential moments is not known to happen.
\end{itemize}
Additionally, we show that exponential moments of solutions to \eqref{intro binary-ternary equation} of order $s\in (0,2]$ as well as polynomial moments of order $k>2$ propagate in time (see Theorem \ref{exponential moments theorem} and Theorem \ref{polynomial moments theorem}), as was the case with the Boltzmann equation \eqref{binary equation}. Finally, polynomial moments of any order are generated in time as long as initial mass and energy are finite (see Theorem \ref{polynomial moments theorem}).

The proof of propagation and generation of moments estimates is done in two phases:
\begin{enumerate}
    \item Phase 1: proving that polynomial moments are finite. More precisely, in Theorem \ref{finiteness of moments} we show that any  solution corresponding to initial data with finite mass and energy has finite and differentiable moments of any order $k>2$. This is proven by an inductive argument that relies on the new decomposition of the collision operator and the novel angular averaging estimate (see Lemmata \ref{new decomposition lemma} and \ref{new decomposition lemma ternary}).
    \item Phase 2: proving quantitative moment estimates (see Subsection \ref{sec - phase 2}). Here one uses Povzner-type angular averaging estimates on the binary and ternary gain operators \eqref{standard gain binary}, \eqref{standard gain ternary} (see Lemma \ref{binary povzner} and Lemma \ref{ternary povzner}) to obtain an ordinary differential inequality for polynomial moments, which results in quantitative estimates after comparison to a Bernoulli-type ODE.
\end{enumerate}

Although the moments analysis described above provides explicit bounds on the generation and propagation of polynomial and exponential moments respectively, the results proved are still a-priori, in the sense that they assume existence of a solution to the equation \eqref{intro binary-ternary equation}. However, since it is the first time that the binary-ternary homogeneous equation \eqref{intro binary-ternary equation} is studied in the literature, we need to address its well-posedness as well. 
That is exactly what we do.
Namely we prove that, as long as the initial data have $2+\varepsilon$  finite moments, there exists a unique, global in time, solution to the equation \eqref{intro binary-ternary equation};
 see Theorem \ref{gwp theorem intro} for more details. To prove this existence theorem, we will rely on techniques of the general theory for ODEs in Banach spaces,  namely Theorem \ref{thm - banach}, which was  implemented in the context of the Boltzmann equation for the first time by Bressan in \cite{br05}. Subsequently, versions of this technique have been used in the context of the Boltzmann equation \cite{alga22}, systems of Boltzmann equations for gas mixtures \cite{gapa20} , as well as the quantum Boltzmann equation \cite{algatr16}.

The instruments that are essential for obtaining moments as well as well-posedness results in this paper are the angular averaging estimates.   For both binary and ternary collision operators, we prove two types of such estimates:
\begin{enumerate}
\item The first type provides an upper bound on the angular averaging part of the gain operator \eqref{standard gain binary}, \eqref{standard gain ternary} in terms of the total energy of the interaction (see Lemma \ref{binary povzner} and Lemma \ref{ternary povzner}). These estimates are used to prove the second type of angular averaging described below, and  to establish the propagation and generation of moments.
In the binary case,  this type of estimate was obtained in \cite{bo97, bogapa04} and is typically referred to as Povzner-type estimate. To the best of our knowledge, this is the first paper that establishes such an estimate for the ternary collision operator. 
\item The second type  introduces a new decomposition of the angular averaging part of the collision operator, which we refer to as the modified gain and modified loss terms (see \eqref{new gain poly}, \eqref{new loss poly}, \eqref{new gain poly ternary}, \eqref{new loss poly ternary}) and consequently provides upper bounds on the modified gain terms and lower bounds on the corresponding modified loss terms (see Lemma \ref{new decomposition lemma} and Lemma \ref{new decomposition lemma ternary}). These estimates are then applied in an inductive argument that establishes finiteness of moments (for details see Subsection \ref{sec - phase 1}).

While these estimates are inspired by the work of Mischler and Wennberg \cite{miwe99}, we emphasize that they are novel even in the binary case. More precisely, results in \cite{miwe99} rely on the representation of post-collisional energies of particles as a sum of a convex combination of pre-collisional energies  and a remainder. We, on the other hand, base our estimates on representing post-collisional energies as a fraction of the total energy of the interaction. This representation is especially suitable for higher order interactions, such as ternary, where pre-post collisional laws are more complex.
 \end{enumerate}

\subsection*{Organization of the paper.} In Section \ref{sec - main results} we provide the notation pertaining to the binary-ternary Boltzmann equation and some of its basic properties, and we state our main results. In Section \ref{sec - angular averaging est}, we derive angular averaging estimates, while in Section \ref{sec - collision operator estimates} we establish estimates on the collision operator for a function that is not necessarily solution to the binary-ternary Boltzmann equation.
In Section \ref{sec - polynomial moments}, we focus on propagation and generation of polynomial moments, while in Section \ref{sec - exponential moments}, we prove propagation and generation of exponential moments. In Section \ref{sec - existence} we prove well-posedness of the equation. Finally, in the Appendix we gather properties of the collision operator and its kernel, as well as  provide some general results such as:  estimates for polynomials and convex functions, auxiliary moment estimates, a general well-posedness theorem for ODEs in Banach spaces, and a lower convolution type bound for generic functions of uniformly nonnegative mass and bounded energy.

\subsection*{Acknowledgements.}   I.A. was supported by the NSF grant DMS-2206618 and the Simons Collaborative Grant on Wave Turbulence. I.M.G. was supported by the funding from NSF DMS: 2009736 and  DOE DE-SC0016283.
N.P. was partially supported by the NSF under grants No. DMS-1840314, DMS-2009549, DMS-2052789. M.T. was partially supported by the NSF grant DMS-2206187 and the AMS-Simons Travel Grant.

\section{Notation and main results}
\label{sec - main results}

\subsection{Vocabulary}\label{sec-vocabulary}
We begin this section by introducing notation that will be used throughout the paper. After that, we will state our main results.

\subsection*{The binary-ternary Boltzmann equation}\label{sec:binary-ternary equation}
We study the well-posedness and generation and propagation of polynomial and exponential moments for the homogeneous binary-ternary Boltzmann equation
\begin{equation}\label{binary-ternary equation}
\begin{cases}
\partial_t f=Q[f]:=Q_2(f,f)+Q_3(f,f,f),\quad (t,v)\in(0,+\infty)\times\mathbb{R}^d,\\
f(t=0)=f_0.
\end{cases}
\end{equation} 
\noindent In \eqref{binary-ternary equation}, $Q_2(f,f)$ is the binary collisional operator given by
\begin{equation}\label{binary collisional operator}
Q_2(f,f)=\int_{\mathbb{S}^{d-1}\times\mathbb{R}^d}B_2(u,\omega)\left(f'f_1'-ff_1\right)\,d\omega\,dv_1,
\end{equation}
where $\omega\in\mathbb{S}^{d-1}$ is the relative position and $u:=v_1-v$ is the relative velocity of the colliding particles, $f':=f(t,v')$, $f_1':=f(t,v_1')$, $f:=f(t,v)$, $f_1:=f(t,v_1)$, and the post-collisional velocities $v',v_1'$ are related to the pre-collisional velocities $v,v_1$ via the binary collisional law:
\begin{equation}\label{binary collisional law}
\begin{cases}
v'=v+(\omega\cdot u) \, \omega,\\
v_1'=v_1-(\omega\cdot u)\, \omega.
\end{cases}
\end{equation}
Clearly the binary momentum-energy conservation system is satisfied i.e.
\begin{equation}\label{binary ME conservation}
\begin{aligned}
v'+v_1'&=v+v_1,\\
|v'|^2+|v_1'|^2&=|v|^2+|v_1|^2.
\end{aligned}
\end{equation}
Denoting $u':=v_1'-v'$  the post-collisional relative velocity, one can also easily verify that the binary micro-reversibility condition holds
\begin{equation}
u'\cdot\omega=-u\cdot\omega,\label{binary micro-reversibility text}
\end{equation}
as well as the  conservation of binary relative velocities magnitude
\begin{equation}
|u'|=|u|.\label{conservation of relative velocities binary}
\end{equation}
Moreover,  given $\omega\in\mathbb{S}^{d-1}$, the transformation $(v,v_1)\to (v',v_1')$ is a linear measure-preserving involution of $\mathbb{R}^{2d}$.

\noindent The binary cross-section $B_2(u,\omega)$, which expresses the statistical repartition of binary collisions, is assumed of the form 
\begin{equation}\label{cross-section 2}
B_2(u,\omega)= |u|^{\gamma_2} b_2(\hat{u}\cdot\omega),\quad u\neq 0,\quad \gamma_2\in [0,2],
\end{equation}
where
the binary angular cross-section $b_2:[-1,1]\to [0,+\infty)$ is an even function. Then relations \eqref{binary micro-reversibility text}-\eqref{conservation of relative velocities binary} yield
\begin{equation}\label{binary microreversibility full text}
B_2(u',\omega)=B_2(u,\omega),\quad\forall u\neq 0,\quad\omega\in\mathbb{S}^{d-1}.
\end{equation}

\noindent We assume $b_2$  satisfies the cut-off assumption $b_2(z)(1-z^2)^{\frac{d-3}{2}}\in L^1([-1,1],\,dz)$. We then define
\begin{equation}\label{binary cut-off}
\|b_2\|:=\int_{\mathbb{S}^{d-1}}b_2(\hat{u}\cdot\omega)\,d\omega<\infty.
\end{equation}
Indeed, by a spherical change of coordinates, $\|b_2\|$ is independent of the direction $\hat{u}$ and finite since
\begin{equation*}
\|b_2\|=\int_{\mathbb{S}^{d-1}}  b_2(\hat{u}\cdot\omega)\,d\omega=\omega_{d-2}\int_{-1}^1 b_2(z)(1-z^2)^{\frac{d-3}{2}}\,dz <\infty,
\end{equation*}
where $\omega_{d-2}$ denotes the surface measure of $\mathbb{S}^{d-2}$.

\bigskip

 The ternary collisional operator $Q_3(f,f,f)$, introduced for the first time in \cite{am20, ampa21}, is given by 
\begin{equation}\label{ternary collisional operator}
\begin{aligned}
Q_3(f,f,f)&=\int_{\mathbb{S}^{2d-1}\times\mathbb{R}^{2d}}B_3(\bm{u},\bm{\omega})\left(f^*f_1^*f_2^*-ff_1f_2\right)\,d\bm{\omega}\,dv_{1,2}\\
&+2\int_{\mathbb{S}^{2d-1}\times\mathbb{R}^{2d}}B_3(\bm{u_1},\bm{\omega})\left(f^{1*}f_1^{1*}f_2^{1*}-ff_1f_2\right)\,d\bm{\omega}\,dv_{1,2},
\end{aligned}
\end{equation}
where $\bm{\omega}=\binom{\omega_1}{\omega_2}\in\mathbb{S}^{2d-1}$ is the impact directions vector and $\bm{u}:=\binom{v_1-v}{v_2-v}$, $\bm{u_1}=\binom{v-v_1}{v_2-v_1}$ are the relative velocities vectors of the colliding particles when the tracked particle $(x,v)$ is central or adjacent respectively for the ternary interaction happening.  When the tracked particle is central, the collisional formulas are
\begin{equation}\label{ternary collisional law}
\begin{cases}
v^*=v+\displaystyle\frac{\bm{u\cdot\omega}}{1+\omega_1\cdot\omega_2}(\omega_1+\omega_2),\\
\\
v_1^*=v_1-\displaystyle\frac{\bm{u\cdot\omega}}{1+\omega_1\cdot\omega_2}\omega_1,\\
\\
v_2^*=v_2-\displaystyle\frac{\bm{u\cdot\omega}}{1+\omega_1\cdot\omega_2}\omega_2.
\end{cases}
\end{equation}
When the tracked particle is adjacent, the collisional formulas are
\begin{equation}\label{ternary collisional law adjacent}
\begin{cases}
v^{1*}=v+\displaystyle\frac{\bm{u_1\cdot\omega}}{1+\omega_1\cdot\omega_2}(\omega_1+\omega_2),\\
\\
v_1^{1*}=v_1-\displaystyle\frac{\bm{u_1\cdot\omega}}{1+\omega_1\cdot\omega_2}\omega_1,\\
\\
v_2^{1*}=v_2-\displaystyle\frac{\bm{u_1\cdot\omega}}{1+\omega_1\cdot\omega_2}\omega_2.
\end{cases}
\end{equation}

\noindent One can easily verify that the ternary momentum-energy conservation system is satisfied i.e.
\begin{equation}\label{ternary ME conservation}
\begin{aligned}
v^*+v_1^*+v_2^*&=v^{1*}+v_1^{1*}+v_2^{1*}=v+v_1+v_2,\\
|v^*|^2+|v_1^*|^2+|v_2^*|^2&=|v^{1*}|^2+|v_1^{1*}|^2+|v_2^{1*}|^2=|v|^2+|v_1|^2+|v_2|^2,
\end{aligned}
\end{equation}
as well as the ternary micro-reversibility conditions:
\begin{align}
\bm{u^*}\cdot\bm{\omega}=-\bm{u}\cdot\bm{\omega},\quad \bm{u_1^{1*}}\cdot\bm{\omega}&=-\bm{u_1}\cdot\bm{\omega},\label{ternary micro-reversibility text}
\end{align}
where we denote $\bm{u^*}:=\binom{v_1^*-v^*}{v_2^*-v^*}$ and $\bm{u_1^{1*}}:=\binom{v^{1*}-v_1^{1*}}{v_2^{1*}-v_1^{1*}}$. Moreover, given $\bm{\omega}=(\omega_1,\omega_2)\in\mathbb{S}^{2d-1}$, the transformations $(v,v_1,v_2)\to (v^*,v_1^*,v_2^*)$ and $(v,v_1,v_2)\to (v^{1*},v_1^{1*},v_2^{1*})$ are linear measure-preserving involutions of $\mathbb{R}^{3d}$.

\noindent Let us define the symmetric quantity
\begin{equation}\label{u tilde}
|\bm{\tild{u}}|:=(|v-v_1|^2+|v-v_2|^2+|v_1-v_2|^2)^{1/2}.
\end{equation}
One immediately observes the inequality
\begin{equation}\label{equivalence of relative velocities ternary}
\frac{1}{\sqrt{3}}|\bm{\tild{u}}|\leq |\bm{u}|,|\bm{u_1}|\leq |\bm{\tild{u}}|.
\end{equation}
By conservation of momentum and energy,  one can also easily verify that
\begin{align}
|\bm{\widetilde{u}^*}|= |\bm{\widetilde{u}_1^{1*}}|=|\bm{\widetilde{u}}|,\label{conservation of relative velocities ternary}
\end{align}
where we denote $|\bm{\widetilde{u}^*}|=(|v^*-v_1^*|^2+|v^*-v_2^*|^2+|v_1^*-v_2^*|^2)^{1/2}$ and $|\bm{\widetilde{u}_1^*}|:=(|v^{1*}-v_1^{1*}|^2+|v^{1*}-v_2^{1*}|^2+|v_1^{1*}-v_2^{1*}|^2)^{1/2}.$

\noindent Defining as well the quantities
\begin{align}
\bm{\bar{u}}:=|\bm{\tild{u}}|^{-1}\bm{u},\quad \bm{\bar{u}_1}:=|\bm{\tild{u}}|^{-1}\bm{u_1},\label{velocity notation ellipsoid ternary}
\end{align}
we notice that $\bm{\bar{u}},\bm{\bar{u}_1}\in\mathbb{E}_1^{2d-1}$, where $\mathbb{E}_1^{2d-1}$ is the $(2d-1)$-dimensional ellipsoid
\begin{equation}\label{ellipsoid}
\mathbb{E}_1^{2d-1}=\{(\nu_1,\nu_2):|\nu_1|^2+|\nu_2|^2+|\nu_1-\nu_2|^2=1\}.
\end{equation}

\noindent The ternary cross-section, which expresses the statistical repartition of ternary collisions, is assumed to be of the form 
\begin{equation}\label{cross-section 3}
\begin{aligned}
B_3(\bm{u},\bm{\omega})&= |\bm{\tild{u}}|^{\gamma_3-\theta_3}|\bm{u}|^{\theta_3} b_3(\bm{\hat{u}\cdot\omega},\omega_1\cdot\omega_2),\quad \bm{u}\neq 0,\quad \gamma_3\in [0,2],\quad \theta_3\geq 0,
\end{aligned}
\end{equation}
where $b_3:[-1,1]\times [-1/2,1/2]\to [0,\infty)$ is of the form 
$$b_3(x,y)=|x|^{\theta_3}\phi(y)$$
 where $\phi\in L^\infty([-1/2,1/2])$.
Note that, since $\bm{\omega}=(\omega_1,\omega_2)\in\mathbb{S}^{2d-1}$, Cauchy-Schwarz inequality implies $\omega_1\cdot\omega_2\in [-1/2,1/2]$. Moreover, notice that due to the form of $b_3$, we may as well write
\begin{equation}\label{alternate form of B_3}
B_3(\bm{u},\bm{\omega})=|\bm{\tild{u}}|^{\gamma_3}b_3(\bm{\bar{u}}\cdot\bm{\omega},\omega_1\cdot\omega_2), \quad B_3(\bm{u_1},\bm{\omega})=|\bm{\tild{u}}|^{\gamma_3}b_3(\bm{\bar{u}_1}\cdot\bm{\omega},\omega_1\cdot\omega_2).
\end{equation}
thus by \eqref{ternary micro-reversibility text}, \eqref{conservation of relative velocities ternary}, we obtain 
\begin{equation}\label{ternary micro-reversibility full text}
B_3(\bm{u^*},\bm{\omega})=B_3(\bm{u},\bm{\omega}),\quad B_3(\bm{u_1^{1*}},\bm{\omega})=B_3(\bm{u_1},\bm{\omega}).
\end{equation}
We also define
\begin{equation}\label{ternary cut-off}
\|b_3\|:=\int_{\mathbb{S}^{2d-1}}b_3(\bm{\hat{u}}\cdot\bm{\omega},\omega_1\cdot\omega_2)\,d\bm{\omega}<\infty.
\end{equation}
Indeed, by spherical symmetry, the quantity $\|b_3\|$ is independent of the direction $\bm{\hat{u}}$ and finite since
\begin{align*}
\|b_3\|&=\int_{\mathbb{S}^{2d-1}}|\bm{\hat{u}}\cdot\bm{\omega}|^{\theta_3}\phi(\omega_1\cdot\omega_2)\,d\bm{\omega}\\
&\le \|\phi\|_{L^\infty}\int_{\mathbb{S}^{2d-1}}|\bm{\hat{u}}\cdot\bm{\omega}|^{\theta_3}\,d\bm{\omega}\\
&=\omega_{2d-2}\|\phi\|_{L^\infty}\int_{-1}^1 |z|^{\theta_3}(1-z^2)^{d-\frac{3}{2}}\,dz<\infty,
\end{align*}
where $\omega_{2d-2}$ denotes the surface measure of $\mathbb{S}^{2d-2}$.

 Throughout this paper we assume that the binary collisional kernel $B_2$ satisfies \eqref{cross-section 2} with \eqref{binary cut-off},  ternary collisional kernel $B_3$ satisfies \eqref{cross-section 3} with \eqref{ternary cut-off}, and that
  \begin{align}
  \gamma : = \max\{\gamma_2, \gamma_3\} >0.
  \end{align}
 Equivalently, at least one of $\gamma_2$, $\gamma_3$ is strictly positive, and this will play a crucial role in the paper.

\subsection*{Spaces relevant for this paper}
The relevant spaces for the study of properties of moments that will be used throughout the paper are polynomially weighted $L^1$-spaces, in particular, given $q\ge 0$, we define the Banach spaces  
$$L^1_q:=\left\{f:\mathbb{R}^d\to\mathbb{R}\text{ such that $f$ is measurable and }\int_{\mathbb{R}^d} |f(v)|\l v\r^q\,dv<\infty\right\},$$
with norm $$\|f\|_{L^1_q}=\int_{\mathbb{R}^d}|f(v)|\l v\r^q\,dv.$$
Notice that $L^1_{q_2}\subset L^1_{q_1}$, whenever $q_1<q_2$ and
for a non-negative function $f$, $m_q[f]<\infty\Leftrightarrow f\in L^1_q$. 

\noindent Now, consider a time interval $I\subset\mathbb{R}$ and $(X,\|\cdot\|)$ a Banach space. We will denote
$$C(I,L^1_q):=\left\{f:I\to L^1_q\text{ such that $f$ is continuous}\right\}.$$
In case $I$ is compact, the above linear space becomes a Banach space with norm $$\|f\|_{C(I,L^1_q)}:=\sup_{t\in I}\|f(t)\|_{L^1_q}<\infty.$$
If $S\subset X$ with $S\neq\emptyset$, we will write
$$C(I,S):=\{f\in C(I,X)\,\text{ such that }f(t)\in S,\quad\forall t\in I\}.$$
Finally, we will denote
$$L^1(I,X):=\left\{f:I\to X\text{ measurable, such that }\int_I\|f(\tau)\|_{X}\,d\tau<\infty\right\},$$
which is a Banach space with norm
$$\|f\|_{L^1(I,X)}:=\int_I\|f(\tau)\|_{X}\,d\tau.$$
Clearly, if $I$ is compact, then $C(I,X)\subset L^1(I,X)$.
We also denote
$$L^1_{loc}(I,X):=\left\{f: I \to X \mbox{ with } f \in L^1(K, X) \mbox{ for any } K \subset \subset I \right\}.$$

\subsection{Weak form of the collisional operators}
We will now write the  collisional operators $Q_2(f,f)$, $Q_3(f,f,f)$ and $Q[f]=Q_2(f,f)+Q_3(f,f,f)$  in weak formulation.  Assuming sufficient integrability for $f$ and a test function $\phi$ for all integrals involved to make sense (see Remark \ref{remark on integrability} below),  the binary collisional operator can be written in weak form as (see Appendix \eqref{binary weak form})
\begin{align}
\int_{\mathbb{R}^d}Q_2(f,f)\phi\,dv
&=\frac{1}{2}\int_{\mathbb{S}^{d-1}\times\mathbb{R}^{2d}} B_2(u,\omega)
ff_1(\phi'+\phi'_1-\phi-\phi_1)\,d\omega\,dv_1\,dv\nonumber\\
&= \frac{1}{2}\int_{\mathbb{S}^{d-1}\times\mathbb{R}^{2d}} |u|^{\gamma_2} b_2(\hat{u}\cdot\omega) 
ff_1(\phi'+\phi'_1-\phi-\phi_1)\,d\omega\,dv_1\,dv,\label{binary weak form text}
\end{align}
while  the ternary collision operator can be written in weak form as (see Appendix \eqref{ternary weak form})
\begin{equation}\label{ternary weak form text}
\begin{aligned}
&\int_{\mathbb{R}^d} Q_3(f,f,f)\phi\,dv=\frac{1}{6}\int_{\mathbb{S}_1^{2d-1}\times\mathbb{R}^{3d}}B_3(\bm{u},\bm{\omega})ff_1f_2(\phi^*+\phi_1^*+\phi_2^*-\phi-\phi_1-\phi_2)\,d\bm{\omega}\,dv_{1,2}\,dv\\
&= \frac{1}{6}\int_{\mathbb{S}_1^{2d-1}\times\mathbb{R}^{3d}}|\bm{\tild{u}}|^{\gamma_3-\theta_3}|\bm{u}|^{\theta} b_3(\bm{\hat{u}\cdot\omega},\omega_1\cdot\omega_2)ff_1f_2(\phi^*+\phi_1^*+\phi_2^*-\phi-\phi_1-\phi_2)\,d\bm{\omega}\,dv_{1,2}\,dv.
\end{aligned}
\end{equation}
Combining  \eqref{binary weak form text}-\eqref{ternary weak form text}, we conclude that $Q[f]$ can be written in weak form as:
\begin{equation}\label{weak form full equation}
\begin{aligned}
 &\int_{\mathbb{R}^{d}}Q[f] \phi\,dv=  \frac{1}{2}\int_{\mathbb{S}^{d-1}\times\mathbb{R}^{2d}} |u|^{\gamma_2} b_2(\hat{u}\cdot\sigma) 
ff_1(\phi'+\phi'_1-\phi-\phi_1)\,d\omega\,dv_1\,dv \\
& + \frac{1}{6}\int_{\mathbb{S}_1^{2d-1}\times\mathbb{R}^{3d}}|\bm{\tild{u}}|^{\gamma_3-\theta_3}|\bm{u}|^{\theta_3} b_3(\bm{\hat{u}\cdot\omega},\langle\omega_1,\omega_2\rangle)ff_1f_2(\phi^*+\phi_1^*+\phi_2^*-\phi-\phi_1-\phi_2)\,d\bm{\omega}\,dv_{1,2}\,dv.
\end{aligned}
\end{equation}

\begin{remark}\label{remark on integrability}
When studying moments, the relevant test functions $\phi$ for \eqref{weak form full equation} are of the form $|\phi(v)|\leq\psi(v)\l v\r^q$, where     $q\geq 0$ and $\psi\in L^\infty$. In that case,  a sufficient condition for \eqref{weak form full equation} to hold in the cut-off regime \eqref{binary cut-off}, \eqref{ternary cut-off}, is $f\in L^1_{q+\gamma}$, where $\gamma=\max\{\gamma_2,\gamma_3\}$. This is a consequence of  Lemmata \ref{lemma on regularity of collisional operator binary}-\ref{lemma on regularity of collisional operator ternary} in the Appendix. In particular $Q:L^1_{q+\gamma}\to L^1_q$.   However, the fact  $f\in L^1_{q+\gamma}$ is  typically not a-priori known, therefore one of the main difficulties one has to face when studying moments is to actually prove that the required integrability is guaranteed under time evolution, see Theorem \ref{finiteness of moments} where we resolve this issue for solutions to the equation \eqref{binary-ternary equation}.
\end{remark}
\subsection*{Collisional averaging}
 By Remark \ref{remark on integrability}, the weak formulation \eqref{weak form full equation} combined with the microscopic conservation laws of binary and ternary collisions \eqref{binary ME conservation}, \eqref{ternary ME conservation}, yield the mass, momentum and energy averaging properties of the collisional operator:
\begin{align}
\int_{\mathbb{R}^d}Q[f]\,dv&=0, \quad\text{when }f\in L^1_{\gamma},\label{averaging of mass}\\
\int_{\mathbb{R}^d} Q[f]v\,dv&=0,\quad\text{when }f\in L^1_{1+\gamma},\label{averaging of momentum}\\
\int_{\mathbb{R}^d} Q[f]|v|^2\,dv&=0, \quad\text{when }f\in L^1_{2+\gamma}.\label{averaging of energy}
\end{align}
where we denote $\gamma=\max\{\gamma_2,\gamma_3\}$.

\subsection*{Conservation laws} The averaging properties of the collisional operator \eqref{averaging of mass}-\eqref{averaging of energy} applied in time,  formally yield the conservation of mass, momentum, and energy for a solution $f$ to the binary-ternary Boltzmann equation \eqref{binary-ternary equation} with initial data $f_0\in L_2^1$:

\begin{align}
\int_{\mathbb{R}^d}f(t,v) 
\begin{pmatrix}
1\\
v\\
|v|^2
\end{pmatrix} \,dv
&=\int_{\mathbb{R}^d}f_0(v)
\begin{pmatrix}
1\\
v\\
|v|^2
\end{pmatrix} 
\,dv.\label{conservation laws}
\end{align}

\subsection{Definition of solutions to the binary-ternary Boltzmann equation and conservation laws}\label{sec:definition of solutions}
In this section, we give a precise definition of a solution to the binary-ternary Boltzmann equation \eqref{binary-ternary equation}.

\begin{definition}\label{def - strong solution}
Let $T>0$, $\gamma=\max\{\gamma_2,\gamma_3\}$ and $ f_0 \in  L^1_{2}$ with $f_0\geq 0$. We say that a  nonnegative function $f\in C([0,T],L^1_2)\cap C^1((0,T],L^1_{2})\cap L^1_{loc}((0,T),L^1_{2+\gamma})$  is a  solution of the binary-ternary Boltzmann equation \eqref{binary-ternary equation} with initial data $f_0$ in $[0,T]$ if \eqref{conservation laws} hold, and if
\begin{equation}\label{IVP binary-ternary}
\begin{cases}
\displaystyle\frac{\,d f}{\,d t}=Q[f],& t\in (0,T],\\
f(t=0)=f_0.
\end{cases}
\end{equation}
\end{definition}

\begin{remark}
Although Definition \ref{def - strong solution} allows the initial data to be in $L_2^1$, we prove well-posedness of the equation \eqref{binary-ternary equation} when $f_0\in L^1_{2+\varepsilon}$, where $\varepsilon>0$ can be arbitrarily small, see Theorem \ref{gwp theorem intro} for more details.
\end{remark}

\begin{remark}
The reason conservation laws are included in the definition of a solution is because we construct solutions for initial data with less than $2+\gamma$ moments. In that case we cannot guarantee that $f \in L^1((0,T),L^1_{2+\gamma})$, which together with $f_0 \in L^1_{2+\gamma}$ would automatically imply conservation laws by collisional averaging \eqref{averaging of mass} - \eqref{averaging of energy}. Instead, the conservation laws will hold for our solutions by construction.
\end{remark}

\begin{remark}
If initial data has mass zero i.e. $\int_{\R^d} f_0 (v) dv=0$, then by the conservation of mass, the unique solution is trivially zero. Thus, in the rest of the paper we will assume that mass of the initial data is non-zero, i.e. $\int_{\R^d} f_0 (v) dv>0$.
\end{remark}

\subsection{Main results}

We prove several results regarding the spatially homogeneous binary-ternary Boltzmann equation \eqref{binary-ternary equation} with integrable angular kernels (see \eqref{binary cut-off}, \eqref{ternary cut-off} and hard potentials in the sense that $\gamma_2,\gamma_3 \ge 0$ and $\gamma=\max\{\gamma_2,\gamma_3\}>0$ i.e. either $\gamma_2>0$ or $\gamma_3>0$. 
Our first main result in this paper is propagation and generation of polynomial moments of solutions to the equation \ref{binary-ternary equation}.

\begin{theorem}\label{polynomial moments theorem} Let $T>0$,  and  let $f\geq 0$ be a solution to the binary-ternary Boltzmann equation \eqref{binary-ternary equation} corresponding to the initial data $f_0 \in L^1_2$, $f_0\geq 0$.   
\begin{enumerate}
\item (Generation) Assume $i\in\{2,3\}$ with $\gamma_i>0$\footnote{such an $i$ always exists since $\gamma=\max\{\gamma_2,\gamma_3\}>0$.}. For $q>2$, we have the estimate
\begin{equation}\label{polynomial generation estimate gamma_i>0} 
	m_{q}[f](t)\leq K_{q,i} \max\left\{1,t^{\frac{2-q}{\gamma_i}} \right\},\quad\forall t\in (0,T],
\end{equation}
for a constant $K_{q,i}>0$, depending on $q,\gamma_i, m_0,m_2$, and independent of $T$.

If $\gamma_2,\gamma_3>0$, we have the bound 
\begin{equation}\label{polynomial generation estimate} 
	m_{q}[f](t)\leq K_q \max\left\{1,\min \left\{t^{\frac{2-q}{\gamma_2}}, t^{\frac{2-q}{\gamma_3}} \right\}\right\}
		\quad\forall t\in (0,T],
\end{equation}
where $K_q=\max\{K_{q,2},K_{q,3}\}>0$ depends on $q,\gamma_2,\gamma_3, m_0,m_2$, and is independent of $T$.

\item (Propagation) Let $q>2$ and  $m_q(0)<\infty$. Then 
\begin{equation}\label{polynomial propagation estimate}
	\sup_{t\in[0,T]}m_q[f](t)\leq M_q,
\end{equation}
for some constant $M_q>0$ depending on $q,\gamma_2,\gamma_3, m_0,m_2,m_q(0)$, and independent of $T$.
\end{enumerate}
\end{theorem}

Our second main result, and the one that exhibits the better behavior of the binary-ternary equation compared to the classical Boltzmann equation, is propagation and generation of exponential moments of weak solutions of the binary-ternary equation. 

\begin{theorem}\label{exponential moments theorem}
 Let $T>0$ and let $f\geq 0$ a solution to the binary-ternary Boltzmann equation \eqref{binary-ternary equation} corresponding to the initial data $f_0 \in L^1_2$, $f_0\geq 0$. 
\begin{itemize}
\item[(a)] (Generation of exponential moments)  
Then, there exist  $a_1, C >0$, depending on $b_2, b_3, \gamma_2, \gamma_3$, initial mass and initial energy, such that
\begin{align*}
\int_{\R^d} f(t,v) e^{a _1\min\{1,t\} \l v \r^\gamma} dv \le C, \quad \mbox{for} \,\,\, t \in [0,T]
\end{align*}

\item[(b)] (Propagation of exponential moments) Let $s \in (0,2]$ and suppose that initial data $f_0$ satisfies
\begin{align}
\int_{\R^d} f_0(v) \, e^{a_0 \l v \r^s} < C_0, \label{exp id}
\end{align}
for some positive constants  $a_0$ and $C_0$. Then there exist positive constant $a$ depends on $C_0, a_0, \|b_2\|, \|b_3\|, \gamma_2, \gamma_3$ and initial mass and initial energy such that
\begin{align*}
\sup_{ t \in [0,T]} \int_{\R^d} f(t,v) \, e^{a \l v \r^s} < 6C_0.
\end{align*}
\end{itemize}
\end{theorem}

Finally, we establish existence of a unique global  solution, when the initial data are in $L^1_{2+\varepsilon}$, for some $\varepsilon>0$.

\begin{theorem}\label{gwp theorem intro}
Let $T>0$, $\varepsilon>0$, and consider initial data $f_0\in L_{2+\varepsilon}^1$ with $f_0\geq 0$.  Then the binary-ternary Boltzmann equation \eqref{binary-ternary equation} has a unique  solution $f\geq 0$, in the sense of Definition \ref{def - strong solution}.
Moreover,  $f \in C^1((0,T), L^1_k)$ for any $k>2$.
\end{theorem}

All of our main results crucially rely on new angular averaging estimates for the binary-ternary collision operator. Many of the estimates that we derive are novel even for the binary equation. For more details, see Section \ref{sec - angular averaging est}.

\section{Angular averaging estimates}
\label{sec - angular averaging est}
In this section we introduce several angular averaging estimates, which will be crucial for obtaining moment estimates, and consequently allow us to prove generation and propagation of polynomial and exponential moments, as well as global well-posedness.
The first subsection provides estimates for the binary part of the collision operator, while the second subsection addresses the ternary part. In each of the subsections there are two angular averaging estimates - one on the gain part of the collision operator, and the other for the newly introduced modified gain and modified loss.

\subsection{Binary angular averaging estimates}
We now work on establishing two types of angular averaging estimates for the collision operator, and we begin by considering  the  binary collision operator. The first estimate has been established in earlier works \cite{bo97, bogapa04}, but is stated here for completeness, and will be used to obtain the second type of the averaging estimates for the binary kernel, see Lemma \ref{new decomposition lemma}.

\begin{lemma}[\cite{bo97,bogapa04} ]\label{binary povzner} 
Suppose $b_2$ satisfies \eqref{binary cut-off}. Let $k\geq 2$. Then there exists a strictly decreasing mapping $\{\alpha_{k/2}\}_{k \ge 2}$ with $\alpha_{1}=\|b_2\|$  
and $\alpha_{k/2}\to 0$ as $k\to\infty$, such that for all $v,v_1\in\mathbb{R}^d$ with $u=v_1-v\neq 0$, we have
	\begin{align}\label{binary povzner estimate}
		\int_{\mathbb{S}^{d-1}} &b_2(\hat{u}\cdot\omega) \, \left(\l v' \r^{k}+\l v'_1\r^{k}\right) \,d\omega
		\leq \alpha_{k/2}E_2^{k/2},
	\end{align}
where $E_2=\l v \r^2+\l v_1\r^2$ is the binary kinetic energy. The mapping $\{\alpha_{k/2}\}_{k\geq 2}$ is called the binary coercive map.
\end{lemma}

In order to state the second type of angular averaging estimate for the binary collisional operator, Lemma \ref{new decomposition lemma},   we first introduce the following notation.
For a given non-negative function $\psi$, and $v,v_1\in\mathbb{R}^d$ with $u=v_1-v\neq 0$, we write
\begin{align}\label{K binary}
	K_{2,\psi}(v,v_1) := G_{2,\psi}(v,v_1)-L_{2,\psi}(v,v_1),
\end{align}
 where
\begin{align}
G_{2,\psi}(v,v_1)&=\int_{\mathbb{S}^{d-1}}b_2(\hat{u}\cdot\omega)\left(\psi(\l v'\r^2)+\psi(\l v_1'\r^2)\right)\,d\omega\label{standard gain binary},\\
L_{2,\psi}(v,v_1)&=\int_{\mathbb{S}^{d-1}}b_2(\hat{u}\cdot\omega)\left(\psi(\l v\r^2)+\psi(\l v_1\r^2)\right)\,d\omega=\|b_2\|\left(\psi(\l v\r^2)+\psi(\l v_1\r^2)\right)\label{standard loss binary}.
\end{align}
We refer to $G_{2,\psi}$ as the binary gain operator and to $L_{2,\psi}$ as the binary loss operator. In the lemma below, we will construct a new decomposition of the collision operator $K_{2,\psi}$. We emphasize that, to the best of our knowledge, this result is novel for the binary collision operator.

\begin{lemma} \label{new decomposition lemma}
Suppose $b_2$ satisfies \eqref{binary cut-off}. Let $k>2$ and $\psi(x)= x^{k/2}$. Then, for all $v,v_1\in\mathbb{R}^d$ with $u=v_1-v\neq 0$, we can   write  $K_{2,\psi} = \tild{G}_{2,\psi} - \tild{L}_{2,\psi}$, where $\tild{G}_{2,\psi},\tild{L}_{2,\psi}$ satisfy the following estimates:
\begin{align}
	& 0\leq \tild{G}_{2,\psi}  
		\leq \alpha_{k/2} C_k  \l v \r \l v_ 1\r 
			\left(   \l v \r^{k -2}+\l v_ 1\r^{k-2}   \right),   \label{bound on poly gain}\\
	&\tild{L}_{2,\psi}\geq (\|b_2\|-\alpha_{k/2})  \left(\l v\r^k+\l v_1\r^k\right),
			 \label{bound on poly loss}
\end{align}
and where $(\alpha_{k/2})_{k\in\mathbb{N}}$ is the  corresponding binary coercive term from Lemma \ref{binary povzner},  and $C_k>1$ is an appropriate constant.

Additionally, if $\psi_n\nearrow \psi$ is the approximating sequence given in Lemma \ref{approximation lemma},  then for every $n$, we can write  $K_{2,\psi_n} = \tild{G}_{2,\psi_n} - \tild{L}_{2,\psi_n}$, where $\tild{G}_{2,\psi_n},\tild{L}_{2,\psi_n}$ satisfy the following:
\begin{align}
	&0\leq \tild{G}_{2,\psi_n}\leq \|b_2\| C_k\l v\r \l v_1\r\left(\l v\r^{k-2}+\l v_1\r^{k-2}\right),\label{bound on binary approx gain}\\
	&\tild{L}_{2,\psi_n}\geq 0,\quad \tild{L}_{2,\psi_n}\to \tild{L}_{2,\psi}, \label{positivity of binary approx loss}
	\end{align}
where $C_k>1$ is the same constant as in \eqref{bound on poly gain}
\end{lemma}

\begin{proof}${}$
Let us define the sets 
\begin{align}
	& \mathcal{A}_{0}:=\{v,v_1\in\mathbb{R}^d: \l v\r \leq 2\l v_1\r\}, \label{A_01 binary}\\
	& \mathcal{A}_{1}:=\{v,v_1\in\mathbb{R}^d: \l v_1\r \leq 2\l v\r\}, \label{A_10 binary}\\
	& \mathcal{A}  : = \mathcal{A}_{0}\cap\mathcal{A}_{1}
		= \{v,v_1 \in\mathbb{R}^d: 2^{-1}\l v\r\leq \l v_1\r\leq 2 \l v\r\}.\label{A binary}
\end{align}
Then $\mathcal{A}^c=\mathcal{A}_{0}^c\cup\mathcal{A}_{1}^c,$
and $\mathcal{A}_{0}^c\cap\mathcal{A}_{1}^c=\emptyset$. Therefore
 \begin{equation}\label{equality of characteristics binary}
 1_{\mathcal{A}}+\mathds{1}_{\mathcal{A}_{0}^c}+\mathds{1}_{\mathcal{A}_{1}^c}=1.
 \end{equation}

Let $k>2$ and $\psi(x) = x^{k/2}$. 
Recalling $G_{2,\psi}, L_{2,\psi}$ from \eqref{standard gain binary}-\eqref{standard loss binary}, we define
\begin{align}
	\tild{G}_{2,\psi} & :=  \left( \1_{\A} 
					+ \left( 1 -  \frac{ \psi(\l v \r^{2})}{\psi(E_2)} \right) \1_{\A_{0}^c}
					+  \left( 1 -  \frac{ \psi(\l v_1 \r^{2})}{\psi(E_2)} \right) \1_{\A_{1}^c}   \right)
					 G_{2,\psi} \label{new gain poly}\\
	\tild{L}_{2,\psi} & := L_{2,\psi} 
				-   \left( \frac{\psi(\l v \r^{2})}{\psi(E_2)} \,  \1_{\A_{0}^c}
				+ \frac{\psi(\l v_1 \r^{2})}{\psi(E_2)} \,    \1_{\A_{1}^c} \right) G_{2,\psi},  
				\label{new loss poly}
\end{align}
and we refer to $\tild{G}_{2,\psi}$ as the modified binary gain operator and to $\tild{L}_{2,\psi}$ as the modified binary loss operator.
Note that by \eqref{equality of characteristics binary}, 
	$
		\tild{G}_{2,\psi}- \tild{L}_{2,\psi}= G_{2,\psi} - L_{2,\psi}=K_{2,\psi}.
	$
In order to estimate  $\tild{G}_{2,\psi}$, we apply \eqref{binary povzner estimate} to obtain
\begin{align*}
	\tild{G}_{2,\psi} 
			& \le  \alpha_{k/2} \left(   E_2^{k/2} \1_{\A} 
					+ \left(   E_2^{k/2} -   \l v \r^{k} \right) \1_{\A_{0}^c}
					+  \left(   E_2^{k/2}  - \l v_1 \r^{k} \right) \1_{\A_{1}^c}   \right)
			= :  \alpha_{k/2} \, \tild{\mathcal{G}}_{2,\psi}.
\end{align*}
Since $\l v\r^2 + \l v_1 \r^2  \le \left( \l v\r + \l v_1 \r \right)^2$, we have
\begin{align*}
	 \left( E_2^{k/2} -   \l v \r^{k}  \right) \1_{\A_{0}^c}
		& \le  \left(  \left(\l v\r + \l v_1 \r \right)^{k} -    \l v \r^{k}  \right) \1_{\A_{0}^c}.
\end{align*}
Then an application of Lemma \ref{binom-max} implies
\begin{align*}
	 \left( E_2^{k/2} -   \l v \r^{k}  \right) \1_{\A_{0}^c}
		  & \le \left( \l v_ 1\r^{k} + C_{2, k}  
				\left( \l v\r^{k -1} \l v_1 \r +   \l v \r \l v_1\r^{k -1} \right)
				 \right) \1_{\A_{0}^c}\\
		& \le 2^{-1} \l v \r \l v_ 1\r^{k -1}  + C_{2, k}  
				\left( \l v\r^{k -1} \l v_1 \r +   \l v \r \l v_1\r^{k -1} \right)\\
		& =  \l v \r \l v_ 1\r
			 \left( (2^{-1} +C_{2, k}) \l v_ 1\r^{k -2} + C_{2, k} \l v \r^{k -2}  \right)
\end{align*}
where $C_{2, k}  = k \max\{ 1, 2^{k -3}\}$. A similar estimate holds for 
$ \left(   E^{k/2}  - \l v_1 \r^{k} \right) \1_{\A_{1}^c} $.

Another application of  $\l v\r^2 + \l v_1 \r^2  \le \left( \l v\r + \l v_1 \r \right)^2$
 and Lemma \ref{binom-max} yields
\begin{align*}
		E_2^{k/2} \1_{\A} 
			&\le  \left(\l v\r + \l v_1 \r \right)^{k} \1_{\A} 
			 \le  \left( \l v \r^{k} + \l v_1 \r^{k} + C_{2, k}  
				\left( \l v\r^{k -1} \l v_1 \r +   \l v \r \l v_1\r^{k -1} \right) \right)  \1_{\A} \\
			& \le (2+C_{2, k})\left( \l v\r^{k -1} \l v_1 \r+   \l v \r \l v_1\r^{k -1} \right) \\
			& =  (2+C_{2, k})  \l v \r \l v_ 1\r
					\left( \l v_ 1\r^{k -2} + \l v \r^{k -2}   \right),
\end{align*}
and so
\begin{align}\label{estimate on curly G binary}
	\mathcal{\tild{G}}_{2,\psi} 
		&\le (2+2^{-1} +3C_{2, k}) \l v \r \l v_ 1\r 
			\left(   \l v \r^{k -2} +\l v_ 1\r^{k -2}  \right),
\end{align}
which proves \eqref{bound on poly gain} with $C_k = 2+2^{-1} +3C_{2, k} >1$.

The lower bound on  $L_{2,\psi}$ \eqref{bound on poly loss}, follows immediately from  \eqref{binary povzner estimate} since:
\begin{align*}
	\tild{L}_{2,\psi} 	& \ge  L_{2,\psi} 
					-\alpha_{k/2} \left(  \l v \r^{k} \, \1_{\A_{0}^c}+ \l v_1 \r^{k} \,    \1_{\A_{1}^c} \right)
				   \ge (\|b_2\| -\alpha_{k/2})   \left(  \l v \r^{k} \,   + \l v_1 \r^{k} \right).
\end{align*}

Consider now the approximating sequence $\psi_n\nearrow\psi$ from Lemma \ref{approximation lemma}, and
 let  $\tild{G}_{2,\psi_n},\tild{L}_{2,\psi_n}$ be defined as in \eqref{new gain poly}-\eqref{new loss poly}. Then, clearly, 
	$
		\tild{G}_{2,\psi_n}- \tild{L}_{2,\psi_n} =K_{2,\psi_n}.
	$
To obtain non-negativity of $\tild{L}_{2,\psi_n}$,  recall the fact that $\l v\r^2=\nu E_2$, $\l v_1\r^2=\nu_1 E_2$, for some $\nu,\nu_1\in [0,1]$ with $\nu+\nu_1=1$. Since $\psi_n$ is convex and satisfies $\psi_n(0)=0$, Lemma \ref{dilation of convex function} implies that for all $x\geq 0$ and $\mu\in [0,1]$ we have 
 $\psi_n(\mu x)\leq\mu \psi_n(x)$. Therefore,
\begin{align} \label{first approx est}
	G_{2,\psi_n} 
		& = \int_{\mathbb{S}^{d-1}}b_2(\hat{u}\cdot\omega) 
			(\psi_n(\nu E_2)+\psi_n(\nu_1 E_2)) \,d\omega \nonumber\\
		& \leq \int_{\mathbb{S}^{d-1} }b_2(\hat{u}\cdot\omega)(\nu+\nu_1)\psi_n(E_2)\,d\omega
		=\|b_2\|\psi_n(E_2),
\end{align}
 and hence
	$$
		\tild{L}_{2,\psi_n}\geq L_{2,\psi_n}-\|b_2\|\left(\psi_n(\l v\r^2)+\psi_n(\l v_1\r^2)\right)=0.
	$$
The convergence $\tild{L}_{2,\psi_n}\to \tild{L}_{2,\psi}$ is immediate, since $\psi_n\nearrow\psi$. It remains to prove the bound \eqref{bound on binary approx gain}.
Notice that by \eqref{first approx est} we have
\begin{align}\label{bound from curly G n 2}
	\tild{G}_{2,\psi_n}
		& \leq \|b_2\| \left(\psi_n(E_2)\mathds{1}_{\mathcal{A}}
			+\left(\psi_n(E_2)-\psi_n(\l v\r^2)\right)\mathds{1}_{A_{0}^c}
			+\left(\psi_n(E_2)-\psi_n(\l v_1\r^2)\right)\mathds{1}_{A_{1}^c}\right) \nonumber \\
		& =: \|b_2\| \, \tild{\mathcal{G}}_{2,\psi_n}.
\end{align}
In order to establish  the bound \eqref{bound on binary approx gain}, it suffices to prove $\tild{\mathcal{G}}_{2,\psi_n}\nearrow \tild{\mathcal{G}}_{2,\psi}$. For this purpose, we first show that  the sequences $a_n=\psi_n(E_2)-\psi_n(\l v\r^2)$ and $b_n=\psi_n(E_2)-\psi_n(\l v_1\r^2)$ are increasing in $n$. Namely, seeing $n$ as a continuous variable for the moment, we compute
\begin{align*} 
	 \partial_n a_n & =\partial_n\Big(\psi_n( E_2)\Big)  - \partial_n\left(\psi_n( \l v \r^2)\right), \\
	\partial_n\psi_n(x) &=
	\begin{cases}
		0, & x \leq n\\
		\frac{k}{2} n^{k/2 -1}  \left(\frac{k}{2} -1 \right) \left( \frac{x}{n} -1 \right) , & x>n.
	\end{cases}
\end{align*}
If $E_2 \le n$ then $\l v \r^2 \le n$ as well, so $\partial_n a_n=0$.
If $E_2>n$ and $\l v \r^2 \leq n$, then
\begin{align*}
 \partial a_n
 	=  \partial_n\left(\psi_n( E_2)\right)=\frac{k}{2} n^{k/2 -1}  \left(\frac{k}{2} -1 \right) \left( \frac{E_2}{n} -1 \right) > 0.
\end{align*}
If $E_2> n$ and $\l v \r^2 > n$, then
\begin{align*}
 \partial_n a_n 
	= \frac{k}{2} n^{k/2 -1} \left(\frac{k}{2} -1 \right) \left(  \frac{E_2 - \l v \r^2}{n}  \right) \ge 0.
\end{align*}
In any case $\partial_n a_n\geq 0$, and so $(a_n)_n$ is increasing. Similarly, $(b_n)_n$ is increasing as well.
Therefore, $\tild{\mathcal{G}}_{2,\psi_n}\nearrow \tild{\mathcal{G}}_{2,\psi}$  and \eqref{bound on binary approx gain} follows thanks to \eqref{estimate on curly G binary}.

\end{proof}

\subsection{Ternary angular averaging estimates}

We now state and prove angular averaging estimates for the ternary collision operator, and we start with the estimate on the ternary gain operator.

\begin{lemma}\label{ternary povzner} 
Suppose $b_3$ satisfies \eqref{ternary cut-off}. Let $k\geq 2$ and $v,v_1,v_2\in\mathbb{R}^d$. Then there is a strictly decreasing mapping $\{\lambda_{k/2}\}_{k\geq 2}$  with $\lambda_{1}=\|b_3\|$ and $\lambda_{k/2}\to 0$ as $k\to\infty$, such that for all $v,v_1,v_2\in\mathbb{R}^d$ with $\bm{u}=\binom{v_1-v}{v_2-v}\neq 0$, we have
	\begin{align}\label{ternary povzner estimate}
		\int_{\mathbb{S}^{2d-1}}  b_3(\bm{\bar{u}\cdot\omega},\langle\omega_1,\omega_2\rangle)
			\left(\l v^* \r^{k} + \l v_1^*\r^{k} + \l v_2^* \r^{k} \right)\,d\bm{\omega} 
		\leq \lambda_{k/2}E_3^{k/2},
	\end{align}
where  $E_3:= \l v \r^2 + \l v_1\r^2 + \l v_2 \r^2$ denotes the ternary kinetic energy. The mapping $\{\lambda_{k/2}\}_{k\geq 2}$ is called the ternary coersive map.
\end{lemma}
\begin{proof}
Let $v,v_1,v_2\in\mathbb{R}^d$, with $\bm{u}\neq 0$. Let us define the scattering direction as:
	$$
		\bm{\sigma}
			=\binom{\sigma_1}{\sigma_2}
			:=\frac{1}{|\tild{u}|}\binom{v^*-v_1^*}{v^*-v_2^*}.
	$$
Notice that, due to \eqref{conservation of relative velocities ternary}, $\bm{\sigma}$  belongs to the ellipsoid $\mathbb{E}^{2d-1}$,  given by \eqref{ellipsoid}. Moreover, formulas \eqref{ternary collisional law} imply that $\bm{\sigma}$  depends smoothly on $\bm{\omega}\in\mathbb{S}^{2d-1}$.
Let us also denote the center of mass of the velocities $v,v_1,v_2$  by
\begin{equation}\label{ternary center of mass}
V_3=\frac{v+v_1+v_2}{3}.
\end{equation}
Then we obtain the following energy identity
\begin{equation}\label{ternary energy identity}
1+ |V_3|^2+\frac{|\tild{u}|^2}{9}=\frac{E_3}{3}.
\end{equation}

By the conservation of momentum, the post-collisional velocities can be written in terms of the scattering direction $\bm{\sigma} = \binom{\sigma_1}{\sigma_2}\in\mathbb{E}_1^{2d-1}$ and the center of mass $V_3$ as follows:
\begin{align}\label{sigma representation}
v^*=V_3-\frac{|\tild{u}|}{3}(\sigma_1+\sigma_2), \quad
v_1^*=V_3+\frac{|\tild{u}|}{3}(2\sigma_1-\sigma_2), \quad
v_2^*=V_3+\frac{|\tild{u}|}{3}(-\sigma_1+2\sigma_2),
\end{align}
and therefore
we have
\begin{align}
\langle v^*\rangle ^2&=E_3 \, \frac{1-\xi'_1\hat{V_3}\cdot(\sigma_1+\sigma_2)+\xi_1(|\sigma_1+\sigma_2|^2-1)}{3}, \nonumber \\
\langle v_1^*\rangle^2&=E_3 \, \frac{1+\xi'_1\hat{V_3}\cdot(2\sigma_1-\sigma_2)+\xi_1(|2\sigma_1-\sigma_2|^2-1)}{3},\label{v*s}\\
\langle v_2^*\rangle^2&=E_3 \, \frac{1+\xi'_1\hat{V_3}\cdot(-\sigma_1+2\sigma_2)+\xi_1(|-\sigma_1+2\sigma_2|^2-1)}{3},\nonumber
\end{align}
where
\begin{align}
\xi_1':=2|\tild{u}||V_3|/E_3\,\quad \mbox{and} \quad
\xi_1:=|\tild{u}|^2/(3E_3)\label{xi's}.
\end{align}
Note that $E_3 \neq 0$ since $\tilde{u} \neq 0$. Also note that $0< \xi_1,\xi_1'< 1$.
Indeed, $\xi_1 < 1$ follows immediately  from \eqref{ternary energy identity}, while Young's inequality and \eqref{ternary energy identity} imply
$$\xi_1' 
	=\frac{6}{E_3} \frac{|\tild{u}|}{3} |V_3|
	\leq \frac{3}{E_3} (\frac{|\tild{u}|^2}{9}+|V_3|^2)
	<1.
$$
Parameters $\xi_1$ and $\xi_1'$ are related, due to \eqref{ternary energy identity}, by the following identity:
$$
\frac{(\xi_1')^2}{4}+\xi_1^{2}=\alpha\xi_1,
$$
where $\alpha=1-\frac{3}{E_3}\in (0,1)$.
Since $\xi_1'$ is nonnegative, we obtain the formula
$$\xi_1' = 2 \sqrt{\alpha\xi_1 -\xi_1^2}.$$

For a general $\xi \in [0,1]$, and for $\hat{V_3} \in \mathbb{S}^{d-1}$ and  $\binom{\sigma_1}{\sigma_2}\in\mathbb{E}_1^{2d-1}$, we define
\begin{align}
\mu(\xi)&= \, \frac{1-2 \sqrt{\alpha\xi -\xi^2}\,\hat{V_3}\cdot(\sigma_1+\sigma_2)+\xi(|\sigma_1+\sigma_2|^2-1)}{3}, \nonumber \\
\mu_1(\xi)&= \, \frac{1+2 \sqrt{\alpha\xi -\xi^2}\,\hat{V_3}\cdot(2\sigma_1-\sigma_2)+\xi(|2\sigma_1-\sigma_2|^2-1)}{3},\label{one xi}\\
\mu_2(\xi)&= \, \frac{1+2 \sqrt{\alpha\xi -\xi^2}\,\hat{V_3}\cdot(-\sigma_1+2\sigma_2)+\xi(|-\sigma_1+2\sigma_2|^2-1)}{3},\nonumber
\end{align}
Since $|\sigma_1 + \sigma_2|^2 + |2\sigma_1 - \sigma_2|^2 + |-\sigma_1 + 2\sigma_2|^2 = 3$, which is true due to $\binom{\sigma_1}{\sigma_2}\in\mathbb{E}_1^{2d-1}$, we have 
$$\mu+\mu_1+\mu_2=1,\quad\forall\xi\in [0,1].$$
Moreover, by Cauchy-Schwarz inequality, and the fact that $\alpha<1$, we have
\begin{align*}
\mu\geq \, \frac{1-\left(\xi|\sigma_1+\sigma_2|^2+(\alpha-\xi)\right)+\xi(|\sigma_1+\sigma_2|^2-1)}{3}=\frac{1-\alpha}{3}>0,\quad\forall\xi\in [0,1].
\end{align*}
Similarly, one can show that $\mu_1,\mu_2>0$ for all $\xi\in[0,1]$. Therefore, we conclude 
\begin{equation}\label{conservation of mus}
\begin{cases}
0< \mu,\mu_1, \mu_2 < 1\\
\mu +\mu_1 + \mu_2 = 1,
\end{cases} 
\quad\forall\xi\in[0,1].
\end{equation}

Now, let
\begin{align}
I_k (\xi_1, \bm{\bar{u}}, \hat{V_3}) : = 
	\int_{\mathbb{S}_1^{2d-1}} b_3(\bm{\bar{u}\cdot\omega},\langle\omega_1,\omega_2\rangle)
			(|v^*|^{k}+|v_1^*|^{k}+|v_2^*|^{k})\,d\bm{\omega}.
\end{align}
and note from \eqref{v*s} that the post-collisional velocities can be represented as:
\begin{align} \label{fractions of energy}
\langle v^*\rangle^2 = \mu(\xi_1) E_3, \quad
\langle v_1^*\rangle^2 = \mu_1(\xi_1) E_3, \quad
\langle v_2^*\rangle^2 = \mu_2(\xi_1) E_3, \quad
\mbox{with} \,\,\xi_1 = |\tilde{u}|^2/(3E_3).
\end{align}
Therefore,
\begin{align*}
I_k (\xi_1, \bm{\bar{u}}, \hat{V_3}) 
	& = E_3^{k/2}  \int_{S_{\xi, \bm{\bar{u}}, \hat{V_3}}} b_3(\bm{\bar{u}\cdot\omega},									\langle\omega_1,\omega_2\rangle)
			(\mu(\xi_1)^{k/2} + \mu_1(\xi_1)^{k/2} +\mu_2(\xi_1)^{k/2}) \,d\bm{\omega} \\
	& =: E_3^{k/2} J_k(\xi_1, \bm{\bar{u}}, \hat{V_3}).
\end{align*}
Let
\begin{align}\label{lambda def}
\lambda_{k/2} := \sup_{(\xi, \bm{\bar{u}}, \hat{V_3}) \in [0,1]\times\mathbb{E}_1^{2d-1} \times \mathbb{S}^{d-1}}
				J_k(\xi, \bm{\bar{u}}, \hat{V_3}).
\end{align}
Then we have
\begin{align*}
I_k (\xi_1, \bm{\bar{u}}, \hat{V_3})  \le \lambda_{k/2} E_3^{k/2}.
\end{align*}
It remains to check the properties of the sequence $\{ \lambda_{k/2}\}_k$ for $k\ge 2$. Since $\mu(\xi) + \mu_1(\xi) + \mu_2(\xi) =1$, \eqref{ternary cut-off} implies that $\lambda_1 = \|b_3\|$.  Moreover, since $\mu,\mu_1,\mu_2<1$,  the following strict inequality holds for any $2 \le k_1 < k_2$
\begin{align} \label{J decreasing}
J_{k_2} (\xi, \bm{\bar{u}}, \hat{V_3}) < J_{k_1} (\xi, \bm{\bar{u}}, \hat{V_3}).
\end{align}
Since the map $(\xi, \bm{\bar{u}}, \hat{V_3})  \mapsto J_k (\xi, \bm{\bar{u}}, \hat{V_3})$ is continuous for each $k \ge 2$ on its compact domain $[0,1]\times\mathbb{E}_1^{2d-1} \times \mathbb{S}^{d-1}$,  the supremum in the definition of $\lambda_{k/2}$ \eqref{lambda def} is attained. Therefore, from \eqref{J decreasing}, since $J_{k_1} (\xi, \bm{\bar{u}}, \hat{V_3}) \le \lambda_{k_1/2}$, we have 
\begin{align} \label{one sup}
J_{k_2} (\xi, \bm{\bar{u}}, \hat{V_3})  < \lambda_{k_1/2}, 
	\qquad \forall (\xi, \bm{\bar{u}}, \hat{V_3}) \in [0,1]\times\mathbb{E}_1^{2d-1} \times \mathbb{S}^{d-1}.
\end{align}
Again, since the supremum in \eqref{lambda def} is attained, there exists 
$(\xi_2, \bm{\bar{u}}_2, (\hat{V_3})_2) \in [0,1]\times\mathbb{E}_1^{2d-1} \times \mathbb{S}^{d-1}$,
such that
\begin{align*}
\lambda_{k_2/2} 
	= J_{k_2} (\xi_2, \bm{\bar{u}}_2, (\hat{V_3})_2) 
	< \lambda_{k_1/2},
\end{align*}
where the last inequality holds by \eqref{one sup}. Therefore, $\{\lambda_{k/2}\}_k$ is strictly decreasing in $k$ for $k \ge 2$. 

Finally, by  the dominated convergence theorem, for every 
$(\xi, \bm{\bar{u}}, \hat{V_3}) \in [0,1]\times\mathbb{E}_1^{2d-1} \times \mathbb{S}^{d-1}$,
we have
\begin{align*}
\lim_{k\rightarrow \infty} J_k (\xi_1, \bm{\bar{u}}, \hat{V_3})   = 0.
\end{align*}
Therefore, Dini's theorem implies that $\lambda_{k/2} \rightarrow 0$, as $k \rightarrow \infty$.
\end{proof}

\begin{remark}
We note that in the above proof we take the advantage of working with  the scattering direction representation  $\bm{\sigma}$  (see \eqref{sigma representation}) which allows us to use the energy identity \eqref{ternary energy identity}. However, we still integrate with respect to the impact direction $\bm{\omega}$ (see \eqref{ternary collisional operator}).
\end{remark}

For a given non-negative function $\psi$, and $v,v_1,v_2\in\mathbb{R}^d$ with $\bm{u}=\binom{v_1-v}{v_2-v}\neq 0$, we write
\begin{align}\label{K ternary}
	K_{3,\psi}(v,v_1,v_2) := G_{3,\psi}(v,v_1,v_2)-L_{3,\psi}(v,v_1,v_2),
\end{align}
where
\begin{align}
	G_{3,\psi}(v,v_1,v_2)
		&=\int_{\mathbb{S}^{2d-1}}  b_3(\bm{\hat{u}}\cdot\bm{\omega},\omega_1\cdot\omega_2)
		\left(\psi(\l v^*\r^2)+\psi(\l v_1^*\r^2)+\psi(\l v_2^*\r^2)\right)\,d\bm{\omega}
		\label{standard gain ternary},\\
	L_{3,\psi}(v,v_1,v_2)
		&=\int_{\mathbb{S}^{2d-1}}  b_3(\bm{\hat{u}}\cdot\bm{\omega},\omega_1\cdot\omega_2)
		\left(\psi(\l v\r^2)+\psi(\l v_1\r^2)+\psi(\l v_2\r^2)\right)\,d\omega \label{standard loss ternary}\\
		&=\|b_3\|\left(\psi(\l v\r^2)+\psi(\l v_1\r^2)+\psi(\l v_2\r^2)\right). \nonumber
\end{align}
We refer to $G_{3,\psi}$ as the ternary gain operator and to $L_{3,\psi}$ as the ternary loss operator. In the lemma below, we will construct a modified decomposition of the collision operator.

\begin{lemma} \label{new decomposition lemma ternary}
Suppose $b_3$ satisfies \eqref{ternary cut-off}. Let $k>2$ and $\psi(x)= x^{k/2}$. Then,  for all $v,v_1,v_2\in\mathbb{R}^d$ with $\bm{u}=\binom{v_1-v}{v_2-v}\neq 0$, we can   write  $K_{3,\psi} = \tild{G}_{3,\psi} - \tild{L}_{3,\psi}$, where $\tild{G}_{3,\psi},\tild{L}_{3,\psi}$ satisfy the following:
\begin{align}
	& 0\leq \tild{G}_{3,\psi}  
		\leq \lambda_{k/2} C_k  \bigg(\l v \r \l v_ 1\r+\l v \r \l v_ 2\r+\l v_1 \r \l v_ 2\r \bigg)
			\left( \l v\r^{k-2} + \l v_1 \r^{k -2} +\l v_2 \r^{k -2}   \right),   \label{bound on poly gain ternary}\\
	&\tild{L}_{3,\psi}\geq (\|b_3\|-\lambda_{k/2})  \left(\l v\r^k+\l v_1\r^k+\l v_2\r^k\right),
			 \label{bound on poly loss ternary}
\end{align}
where $\lambda_{k/2}$ is the corresponding ternary coersive term from Lemma \ref{ternary povzner},  and $C_k>1$ is an appropriate constant.

Additionally, if $\psi_n\nearrow \psi$ is the sequence given in Lemma \ref{approximation lemma},  then for every $n$, we can write  $K_{3,\psi_n} = \tild{G}_{3,\psi_n} - \tild{L}_{3,\psi_n}$, where $\tild{G}_{3,\psi_n},\tild{L}_{3,\psi_n}$ satisfy the following:
\begin{align}
	&0\leq \tild{G}_{3,\psi_n}\leq C_k  \bigg(\l v \r \l v_ 1\r+\l v \r \l v_ 2\r+\l v_1 \r \l v_ 2\r \bigg)
			\left( \l v\r^{k-2} + \l v_1 \r^{k -2} +\l v_2 \r^{k -2}   \right),\label{bound on approx gain ternary}\\
	&\tild{L}_{3,\psi_n}\geq 0,\quad \tild{L}_{3,\psi_n}\to \tild{L}_{3,\psi} \label{positivity of approx loss ternary}.
	\end{align}
\end{lemma}

\begin{proof}${}$
We define the sets
\begin{equation}\label{A_ij ternary}
\begin{aligned}
\mathcal{A}_{0}&:=\left\{v,v_1,v_2\in\mathbb{R}^d: \l v\r\leq 2\l v_1\r \lor \l v\r\leq 2\l v_2\r\right\},\\
\mathcal{A}_{1}&:=\left\{v,v_1,v_2\in\mathbb{R}^d: \l v_1\r\leq 2\l v\r \lor \l v_1\r\leq 2\l v_2\r\right\},\\
\mathcal{A}_{2}&:=\left\{v,v_1,v_2\in\mathbb{R}^d: \l v_2\r\leq 2\l v\r \lor \l v_2\r\leq 2\l v_1\r\right\}.
\end{aligned}
\end{equation}
and
\begin{align}
	\mathcal{A} & : = \bigcap_{i=0,1,2}\mathcal{A}_{i} .\label{A ternary}
\end{align}
Then $\mathcal{A}^c=\bigcup_{i=0,1,2}\mathcal{A}_{i}^c,$ and  the above union is disjoint so
 \begin{equation}\label{equality of characteristics ternary}
 \1_{\mathcal{A}}+\sum_{i=0,1,2}\1_{\mathcal{A}_{i}^c}=1.
 \end{equation}

Let $k>2$ and $\psi(x) = x^{k/2}$.
Recalling $G_{3,\psi}, L_{3,\psi}$ from \eqref{standard gain ternary}-\eqref{standard loss ternary}, we define
\begin{align}
	\tild{G}_{3,\psi} & := \left( \1_{\A} 
				+ \sum_{i=0,1,2}\left(1-\frac{\psi(\l v_i\r^2)}{\psi(E_3)}\right)\mathds{1}_{\mathcal{A}_{i}^c} \right)
				G_{3,\psi}, \label{new gain poly ternary}\\
	\tild{L}_{3,\psi} & := L_{3,\psi} 
				-\, \left(  
					 \sum_{i=0,1,2}\frac{\psi(\l v_i\r^2)}{\psi(E_3)}\mathds{1}_{\mathcal{A}_{i}^c} \,  \right)
					G_{3,\psi},  \label{new loss poly ternary}\
\end{align}
and refer to $\tild{G}_{3,\psi}$ as the modified ternary gain operator and to $\tild{L}_{3,\psi}$ as the modified ternary loss operator.
By \eqref{equality of characteristics ternary}, we have
	$
		\tild{G}_{3,\psi}- \tild{L}_{3,\psi}= G_{3,\psi} - L_{3,\psi}=K_{3,\psi}.
	$
In order to estimate  $\tild{G}_{3,\psi}$, we apply \eqref{ternary povzner estimate} to obtain
\begin{align*}
	\tild{G}_{3,\psi} 
			& \le  \lambda_{k/2} \left(   E_3^{k/2} \1_{\A} 
					+  \sum_{i=0,1,2}\ \left(   E_3^{k/2} -   \l v_i \r^{k} \right) \1_{\A_{i}^c}  \right)
			= : \lambda_{k/2}  \, \tild{\mathcal{G}}_{3,\psi}.
\end{align*}
Since $E_3 \le \left( \l v\r + \l v_1 \r + \l v_2 \r \right)^2$, we have
\begin{align*}
	 \left(   E_3^{k/2} -   \l v \r^{k} \right) \1_{\A_{0}^c}  
		& \le  \left(  \left(\l v\r + \l v_1 \r + \l v_2 \r\right)^{k} -    \l v\r^{k}  \right) \1_{\A_{0}^c}.
\end{align*}
Lemma \ref{binom-max} then yields
\begin{align*}
	& \left( E_3^{k/2} -   \l v \r^{k}  \right) \1_{\A_{0}^c}
		 \le \left( \l v_ 1\r^{k}+\l v_ 2\r^{k} 
		 		 + C_{3, k}  \sum_{i\ne j \in \{0,1,2\}} \l v_i \r^{k-1} \l v_j \r
				 \right) \1_{\A_{0}^c}\\
		& \qquad \le 2^{-1}\left(  \l v \r \l v_ 1\r^{k -1} +\l v \r \l v_ 2\r^{k -1} \right)
				 + C_{3, k}  \sum_{i\ne j \in \{0,1,2\}} \l v_i \r^{k-1} \l v_j \r\\
		& \qquad \le  (2^{-1} +C_{3, k})  \Big(\l v \r \l v_ 1\r+\l v \r \l v_ 2\r+\l v_1 \r \l v_ 2\r \Big)
			\left(  \l v\r^{k -2} + \l v_ 1\r^{k -2} + \l v _2\r^{k -2}  \right)
\end{align*}
where $C_{3,k} = k \max\{ 1, 2^{k-3}\} + \frac{k (k-1)}{2} \max\{1, 2^{k-4} \}$. A similar calculation yields the same upper bound on $ \left(   E_3^{k/2}  - \l v_1 \r^{k} \right) \1_{\A_{1}^c} $ and 
$ \left(   E_3^{k/2}  - \l v_2 \r^{k} \right) \1_{\A_{2}^c} $. On the other hand,
another application of $E_3 \le \left( \l v\r + \l v_1 \r + \l v_2 \r \right)^2$, Lemma \ref{binom-max} yields and the definition of the set $\A$ in \eqref{A ternary} yields
\begin{align*}
		E_3^{k/2} \1_{\A} 
			&\le  \left(\l v\r + \l v_1 \r  + \l v_2 \r \right)^{k} \1_{\A} \\
			& \le  \left( \l v \r^{k} + \l v_1 \r^{k} + \l v_2 \r^{k}  
				+ C_{3, k}   \sum_{i\ne j \in \{0,1,2\}} \l v_i \r^{k-1} \l v_j \r \right)  \1_{\A} \\
			& \le 2 \left(\l v \r^{k-1} \l v_ 1\r+\l v_1 \r^{k-1} \l v_ 2\r +\l v \r \l v_ 2\r^{k-1}  \right)\\
			& \qquad + C_{3, k}   \Big(\l v \r \l v_ 1\r+\l v \r \l v_ 2\r+\l v_1 \r \l v_ 2\r \Big)
				\left(  \l v\r^{k -2} + \l v_ 1\r^{k -2} + \l v _2\r^{k -2}  \right)\\
			& \le (2 +C_{3, k})  \Big(\l v \r \l v_ 1\r+\l v \r \l v_ 2\r+\l v_1 \r \l v_ 2\r \Big)
				\sum_{i=0,1,2} \l v_ i\r^{k-2},
\end{align*}
and so
\begin{align}\label{estimate on curly G ternary}
	\mathcal{\tild{G}}_{3,\psi} 
		&\le  (2+ 3/2 +4C_{2, k}) 
			 \Big(\l v \r \l v_ 1\r+\l v \r \l v_ 2\r+\l v_1 \r \l v_ 2\r \Big)
			\sum_{i=0,1,2} \l v_ i\r^{k-2}
\end{align}
which proves \eqref{bound on poly gain ternary} with $C_k = 2+3/2 +4C_{2, k} >1$.

The  lower bound on  $L_{2,\psi}$, \eqref{bound on poly gain ternary}, is an immediate applciation of \eqref{ternary povzner estimate} since
\begin{align*}
	\tild{L}_{3,\psi} 	& \ge  L_{3,\psi} 
				-\, \lambda_{k/2} \left(  \sum_{i=0,1,2} \l v_i\r^k \mathds{1}_{\mathcal{A}_{i}^c} \right)
				 \ge (\|b_3\| -\lambda_{k/2})  \sum_{i=0,1,2} \l v_ i\r^{k}.
\end{align*}

Consider now the approximating sequence $\psi_n\nearrow\psi$ from Lemma \ref{approximation lemma}, and
 let  $\tild{G}_{3,\psi_n},\tild{L}_{3,\psi_n}$ be defined as in \eqref{new gain poly ternary}-\eqref{new loss poly ternary}. Then, clearly, 
	$
		\tild{G}_{3,\psi_n}- \tild{L}_{3,\psi_n} =K_{3,\psi_n}.
	$
To obtain non-negativity of $\tild{L}_{2,\psi_n}$, observe that by the convexity of $\psi_n$ and the fact that $\psi_n(0)=0$, we have $\psi_n(\mu x)\leq\mu \psi(x)$ for all $x\geq 0$ and $\mu\in [0,1]$.
Now, recall the fact that $\l v^*\r^2=\mu E_3$, $\l v^*_1\r^2=\mu_1 E_3$ and $\l v^*_2\r^2=\mu_2 E_3$ for some $\mu,\mu_1,\mu_2\in [0,1]$  with $\mu+\mu_1+\mu_2=1$. 
 Using the above observation,  we obtain
\begin{align}
	G_{3,\psi_n} 
		& = \int_{\mathbb{S}^{2d-1}}  b_3(\bm{\hat{u}}\cdot\bm{\omega},\omega_1\cdot\omega_2)
			(\psi_n(\mu E_3)+\psi_n(\mu_1 E_3)+\psi_n(\mu_2 E_3))\,d\omega \nonumber\\
		& \leq \int_{\mathbb{S}^{2d-1}} b_3(\bm{\hat{u}}\cdot\bm{\omega},\omega_1\cdot\omega_2)
			(\mu+\mu_1 + \mu_2)\psi_n(E_3)\,d\omega=\|b_3\| \, \psi_n(E_3), \label{second approx est}
\end{align}
	and  hence
	$$
		\tild{L}_{3,\psi_n}\geq L_{3,\psi_n}-\|b_3\|\left(\psi_n(\l v\r^2)+\psi_n(\l v_1\r^2) + \psi_n(\l v_2\r^2)\right)=0.
	$$
The convergence $\tild{L}_{2,\psi_n}\to \tild{L}_{2,\psi}$ is immediate, since $\psi_n\nearrow\psi$. In order  to prove the bound \eqref{bound on approx gain ternary}, first notice that \eqref{second approx est} implies
\begin{equation}\label{bound from curly G n 3}
	\tild{G}_{3,\psi_n}
		\leq \|b_3\|  \left(\psi_n(E_3) \mathds{1}_{\mathcal{A}} 
			+ \sum_{i = 0,1,2} \left(\psi_n(E_3)-\psi_n(\l v_i\r^2)\right)\mathds{1}_{A_{i}^c} \right)
			=: \|b_3\|  \, \tild{\mathcal{G}}_{3,\psi_n}.
\end{equation}

We aim to show that $\tild{\mathcal{G}}_{3,\psi_n}\nearrow \tild{\mathcal{G}}_{3,\psi}$. For this purpose, we first show that  the sequences $a_{n,i}=\psi_n(E_3)-\psi_n(\l v_i\r^2)$  are increasing in $n$. Namely, seeing $n$ as a continuous variable for the moment, we compute
\begin{align*} 
	 \partial_n a_{n,i} & =\partial_n\left(\psi_n( E_3)\right)  - \partial_n\left(\psi_n( \l v_i \r^2)\right), \\
	\partial_n\psi_n(x) &=
	\begin{cases}
		0, & x \leq n\\
		\frac{k}{2} n^{k/2 -1}  \left(\frac{k}{2} -1 \right) \left( \frac{x}{n} -1 \right) , & x>n.
	\end{cases}
\end{align*}
If $E_3 \le n$ then $\l v_i \r^2 \le n$ as well, so $\partial_n a_{n,i}=0$.
If $E_3>n$ and $\l v_i \r^2 \leq n$, then
\begin{align*}
 \partial a_{n,i}
 	=  \partial_n\left(\psi_n( E_3)\right)=\frac{k}{2} n^{k/2 -1}  \left(\frac{k}{2} -1 \right) \left( \frac{E_3}{n} -1 \right) > 0.
\end{align*}
If $E_3> n$ and $\l v_i\r^2 > n$, then
\begin{align*}
 \partial_n a_{n,i} 
	= \frac{k}{2} n^{k/2 -1} \left(\frac{k}{2} -1 \right) \left(  \frac{E_3 - \l v_i \r^2}{n}  \right) \ge 0.
\end{align*}
In any case, $\partial_n a_{n,i}\geq 0$ , andso $(a_{n,i})_n$ is increasing.
Therefore $\tild{\mathcal{G}}_{2,\psi_n}\nearrow \tild{\mathcal{G}}_{2,\psi}$  and \eqref{bound on approx gain ternary} follows thanks to \eqref{estimate on curly G ternary}.

\end{proof}

\section{Collision operator estimates}
\label{sec - collision operator estimates}

In this section we present estimates on the weak form of the collision operator applied on a function $f \in L^1_{q+\gamma}$ (without assuming that $f$ is a solution to the binary-ternary Boltzmann equation). This estimate will have twofold purpose - it will be used for proving quantitative  generation and propagation estimates  on polynomial moments as stated in Theorem \ref{polynomial moments theorem}, and the existence of solutions  to the binary-ternary Boltzmann equation (see Section \ref{sec - existence}).

\begin{proposition}\label{moments ode pre theorem q}
Suppose $b_2, b_3$ satisfy \eqref{cross-section 2}-\eqref{binary cut-off} and \eqref{cross-section 3}-\eqref{ternary cut-off}. Let $q>2$ and suppose $f(v)\geq 0$, $f\in L^1_{q+\gamma}$, where $\gamma=\max\{\gamma_2,\gamma_3\}$. Then the following estimates hold:
\begin{equation}\label{shifted estimate q}
	\int_{\mathbb{R}^d}Q[f]\langle v\rangle^{q}\,dv
		\leq  C_q (m_0[f],m_2[f]) m_{q}[f] - C_q'\left(m_0[f]) (m_{q+\gamma_2}[f]+m_{q+\gamma_3}[f]\right),
\end{equation}
and
\begin{equation}\label{power estimate q}
	\int_{\mathbb{R}^d}Q[f]\langle v\rangle^{q}\,dv
		\leq  C_q(m_0[f],m_2[f]) m_{q}[f]-\tild{C}_q(m_0[f],m_2[f]) \left(m_{q}[f]^{1+\frac{\gamma_2}{q-2}}+m_{q}[f]^{1+\frac{\gamma_3}{q-2}}\right)),
\end{equation}
where $C_q(m_0[f],m_2[f])>0$ is  continuous with respect to $m_0[f],m_2[f]$ and depends  on $q, \gamma_2, \gamma_3, \|b_2\|$ and $\|b_3\|$,
while the two coercive factors, $C_q'(m_0[f])$ and $\tild{C}_q(m_0[f],m_2[f])$  have the following formulas
\begin{align*}
 	C_q'(m_0[f]) & = \min\left\{  (\|b_2\|-\alpha_{\frac{q}{2}}) 2^{-1-\frac{\gamma_2}{2}} m_0[f], \,\, \frac{1}{4}(\|b_3\|-\lambda_{\frac{q}{2}})3^{-\frac{\theta_3}{2}} \left(\frac{2}{3}\right)^{\frac{\gamma_3}{2}} m_0[f]^2 \right\},\\ 
	\tild{C}_q(m_0[f]) & = C_q' (m_2[f]^{-\frac{\gamma_2}{q-2}} + m_2[f]^{-\frac{\gamma_3}{q-2}}),
\end{align*} 
 with $ \alpha_{q/2}$ and $ \lambda_{q/2}$
  as in \eqref{binary povzner estimate} and  \eqref{ternary povzner estimate}, respectively.
\end{proposition}

\begin{proof}
Let $q > 2$ and write it as $q = 2r$, with $r > 1$. The cross-section representation \eqref{cross-section 2} and  the binary angular averaging Lemma \ref{binary povzner} yield
\begin{align*}
K_{2,2r}(v,v_1)
	&:=  \int_{\mathbb{S}^{d-1}} B_2(u,\omega) 
		\left(\langle v'\rangle^{2r}+\langle v'_1\rangle^{2r}-\langle v\rangle^{2r}-\l v_1\r^{2r}\right)\,d\omega\nonumber\\
	& \le |u|^{\gamma_2} \left( \alpha_{r}  \left(E_2^{r} - \l v\r ^{2r}-\l v_1\r^{2r}\right)   
		- (\|b_2\|- \alpha_{r}) \left(\l v\r^{2r} + \l v_1\r^{2r}\right)\right).
\end{align*}
Applying Lemma \ref{binom-max} to the first term yields
\begin{align} \label{gain bound binary q}  
K_{2,2r}(v,v_1)
	& \le |u|^{\gamma_2} \left( \alpha_{r} C_{2,r} \Big(\l v \r^{2r-2} \l v_1 \r^2 +\l v \r^2 \l v_1 \r^{2r-2}\Big)
			- (\|b_2\|- \alpha_{r}) (\l v\r^{2r} + \l v_1\r^{2r})	\right),
\end{align}
where $C_{2,r} = r \max\{ 1, 2^{r-3} \}$. 
 Similarly, by \eqref{cross-section 3},  Lemma \ref{ternary povzner}, 
and the ternary  angular averaging Lemma \ref{ternary povzner}, and  Lemma \ref{binom-max} imply
\begin{align} \label{gain bound ternary q}
K_{3,2r} (v,v_1,v_2) 
	& :=  \int_{\mathbb{S}^{2d-1}}  B_3(\bm{u},\bm{\omega}) 
		\left(\l v^*\r^{2r}+\l v_1^*\r^{2r}+\l v_2^*\r^{2r}-\l v\r^{2r}-\l v_1\r^{2r}-\l v_2\r^{2r}\right)\,d\bm{\omega}	
		\nonumber\\
 &\le   |\bm{\tild{u}}|^{\gamma_3-\theta_3}|\bm{u}|^{\theta_3}  \left[
		\lambda_{r} C_{3,r} \Big( \l v \r^{2r-2} \l v_1 \r^2 + \l v \r^2 \l v_1 \r^{2r-2} 
				+ \l v \r^{2r-2}  \l v_2 \r^2  \right. \nonumber \\
	& \Big. \qquad  + \l v \r^2  \l v_2 \r^{2r-2}  
		+ \l v_1 \r^{2r-2}  \l v_2 \r^2 + \l v_1 \r^2  \l v_2 \r^{2r-2}\Big)  \nonumber \\
	&	\qquad -(\|b_3\| -\lambda_{r}) (\l v\r^{2r} + \l v_1\r^{2r} + \l v_2\r^{2r}) \Big],
\end{align}
where $C_{3,r} = r \max\{ 1, 2^{r-3}\} + \frac{r (r-1)}{2} \max\{1, 2^{r-4} \} \le 2^{2r}.$
Plugging estimates \eqref{gain bound binary q}-\eqref{gain bound ternary q}  into the weak form \eqref{weak form full equation} yields
 \begin{align}
\int_{\mathbb{R}^d} &Q[f]  \l v\r^{2r}\, dv
	=  \frac{1}{2} \int_{\mathbb{R}^{2d}} ff_1 K_{2,2r}(v,v_1) \,dv\,dv_1 
		+ \frac{1}{6} \int_{\mathbb{R}^{3d}} f f_1 f_2 K_{3,2r}(v,v_1,v_2) \,dv\,dv_{1,2} \nonumber\\
	&\leq  \alpha_{r} C_{2,r} \int_{\mathbb{R}^{2d}}|u|^{\gamma_2} 
		ff_1 \l v \r^{2r-2} \l v_1 \r^2   \,dv\,dv_1 
		 	- (\|b_2\| -\alpha_{r}) \int_{\mathbb{R}^{2d}} |u|^{\gamma_2} f f_1 \l v\r^{2r}\,dv\,dv_1\nonumber\\
	& \,\, +  \lambda_{r} C_{3,r} \int_{\mathbb{R}^{3d}}  |\bm{\tild{u}}|^{\gamma_3-\theta_3}|\bm{u}|^{\theta_3} 
			f f_1 f_2 \l v \r^{2r-2} \l v_1 \r^2 dv dv_{1,2}  \nonumber\\
	& \,\, -\frac{1}{2}(\|b_3\| -\lambda_{r}) \int_{\mathbb{R}^{3d}}  |\bm{\tild{u}}|^{\gamma_3-\theta_3}|\bm{u}|^{\theta_3} 
		 f f_1 f_2 \l v\r^{2r}dv dv_{1,2}.
\label{first moments ineq q}
\end{align}
At this point we apply upper and lower bounds on the binary and ternary potentials  \eqref{binary potential upper bound} -\eqref{ternary potential lower bound}, and switch back to the $q$ notation to obtain
\begin{align*}
\int_{\mathbb{R}^d}Q[f]  \l v\r^{q}\,dv
 	& \le \alpha_{\frac{q}{2}} C_{2,\frac{q}{2}} C_{\gamma_2} (m_{q-2+\gamma_2} \, m_{2} + m_{q-2} \, m_{\gamma_2 + 2}) 
 		 - (\|b_2\|-\alpha_{\frac{q}{2}}) \left( 2^{-\frac{\gamma_2}{2}} m_{q+\gamma_2} \, m_0 - m_q \, m_{\gamma_2}\right)\\
 	& \qquad +  \lambda_{\frac{q}{2}} C_{3,\frac{q}{2}} C_{\gamma_3} (m_{q-2+\gamma_3}\, m_2 \,m_0 + m_{q-2}\, m_{\gamma_3 + 2}\, m_0 + m_{q-2} \, m_{\gamma_3} m_2 )\\
 	&  \qquad - \frac{1}{2}(\|b_3\|-\lambda_{\frac{q}{2}}) 3^{\frac{- \theta_3}{2}}
	\left( \left(\frac{2}{3}\right)^{\frac{\gamma_3}{2}} m_{q+\gamma_3} \, m_0 \, m_0 - 2 m_q \,m_{\gamma_3} \,m_0 \right),
\end{align*}
where $ C_{\gamma_2} = \max\{1, 2^{\gamma_2-1}\} $ and $C_{\gamma_3} = 2^{\gamma_3} \max\{ 1, 3^{\gamma_3 -1}\}$.
In order to estimate the term  $m_{q-2} \, m_{\gamma_2 + 2}$ (and $m_{q-2} \, m_{\gamma_3+2}$), motivated by \cite{alga22} we combine interpolation estimates together with the $\varepsilon$-Young's inequality. Namely,  by the interpolation Lemma \ref{interpolation lemma}, we have

\begin{align*}
	m_{q-2}  \le m_0^{\frac{\gamma_2 + 2}{q + \gamma_2}}  \, m_{q+\gamma_2}^{\frac{q-2}{q+\gamma_2}} 
		\quad \mbox{and } \quad	
	m_{\gamma_2 + 2}  \le  m_2^{\frac{q-2}{q+\gamma_2 - 2}} \, m_{q+\gamma_2}^{\frac{\gamma_2}{q+\gamma_2 -2}}.
\end{align*}

Therefore,   
\begin{align*}
	m_{q-2} \, m_{\gamma_2 + 2} \le  A_{q, \gamma_2} \, m_{q+\gamma_2}^{\theta_{q,\gamma_2}},
\end{align*}
where 
\begin{align*}
	A_{q, \gamma_2} &= m_0^{\frac{\gamma_2 + 2}{q + \gamma_2}}  \, m_2^{\frac{q-2}{q+\gamma_2 - 2}},\\
	\theta_{q, \gamma_2}  & =  \frac{q-2}{q+\gamma_2} + \frac{\gamma_2}{q+\gamma_2 -2} 
		 < \frac{q-2}{q+\gamma_2 - 2} + \frac{\gamma_2}{q+\gamma_2 -2} = 1. 
\end{align*}
Next, we use $\varepsilon$-Young's inequality  to obtain
\begin{align*}
	m_{q-2} \, m_{\gamma_2+2} 
		\le  \varepsilon_{q, \gamma_2}^{-\frac{\theta_{q,\gamma_2}}{1- \theta_{q,\gamma_2}}}  B_{q,\gamma_2}
		+  \varepsilon_{q, \gamma_2} \, \theta_{q,\gamma_2} \, m_{q+\gamma_2},
\end{align*}
where
\begin{align*}
B_{q,\gamma_2} & =(1-\theta_{q,\gamma_2}) \,
		A_{q, \gamma_2}^{\frac{1}{1-\theta_{q,\gamma_2}}}
		\le A_{q, \gamma_2}^{\frac{1}{1-\theta_{q,\gamma_2}}}
		\le   \left( m_2^{\frac{\gamma_2 + 2}{q + \gamma_2}+\frac{q-2}{q+\gamma_2 - 2}} 
			\right)^{\frac{(q+\gamma_2)(q+\gamma_2 -2)}{2(q-2)}}
		= m_2^{\frac{(\gamma_2+2)(q+\gamma_2 -2) + (q-2)(q+\gamma_2)}{2(q-2)}},
\end{align*}
and where $\varepsilon_{q,\gamma_2} >0$ will be chosen in a moment.
Analogous estimates holds for $m_{q-2} \, m_{\gamma_3+2}$, where $\gamma_2$ is replaced with $\gamma_3$ appropriately.

Now, using that $\gamma_2, \gamma_3 \le 2$  and the fact that moments are monotone increasing with respect to their order, we obtain
\begin{align*}
\int_{\mathbb{R}^d}Q[f]  \l v\r^{q}\,dv
 	& \le \alpha_{\frac{q}{2}} C_{2,\frac{q}{2}} C_{\gamma_2} \left(m_q \, m_{2} 
			+  \varepsilon_{q, \gamma_2}^{-\frac{\theta_{q,\gamma_2}}{1- \theta_{q,\gamma_2}}}  B_{q,\gamma_2}
		+  \varepsilon_{q, \gamma_2} \, \theta_{q,\gamma_2} \, m_{q+\gamma_2} \right) \\
 	& \qquad 	 - (\|b_2\|-\alpha_{\frac{q}{2}}) 
		\left( 2^{-\frac{\gamma_2}{2}} m_{q+\gamma_2} \, m_0 - m_q \, m_{2}\right)\\
 	& \qquad +  \lambda_{\frac{q}{2}} C_{3,\frac{q}{2}} C_{\gamma_3} 
		\left(2m_{q}\, m_2^2  
			 +  \varepsilon_{q, \gamma_3}^{-\frac{\theta_{q,\gamma_3}}{1- \theta_{q,\gamma_3}}}  B_{q,\gamma_3}
			+  \varepsilon_{q, \gamma_3} \, \theta_{q,\gamma_3} \, m_{q+\gamma_3} \right)\\
 	&  \qquad - \frac{1}{2}(\|b_3\|-\lambda_{\frac{q}{2}}) 3^{-\frac{\theta_3}{2}}\left( \left(\frac{2}{3}\right)^{\frac{\gamma_3}{2}} m_{q+\gamma_3} \, m_0^2  - 2 m_q \,m_{2}^2 \right).
\end{align*}
In order to ensure that $m_{q+\gamma_2}$ and $m_{q+\gamma_3}$ terms are negative, we choose $\varepsilon_{q,\gamma_2}$ and $\varepsilon_{q,\gamma_3}$ to be 
\begin{align*}
	\varepsilon_{q,\gamma_2} = \frac{(1- \alpha_{\frac{q}{2}}) 2^{\frac{-\gamma_2}{2}} m_0}{ 2 \alpha_{\frac{q}{2}} C_{2,\frac{q}{2}} C_{\gamma_2} \theta_{q,\gamma_2}}
		\quad \mbox{and} \quad 
	\varepsilon_{q,\gamma_3} = \frac{ \frac{1}{2}(1-\lambda_{\frac{q}{2}})3^{-\frac{\theta_3}{2}}  \left(\frac{2}{3}\right)^{\frac{\gamma_3}{2}}  m_0^2}{2  \lambda_{\frac{q}{2}} C_{3,\frac{q}{2}} C_{\gamma_3} \theta_{q,\gamma_3}}.
\end{align*}
With such a choice of $\varepsilon_{q,\gamma_2}$ and $\varepsilon_{q,\gamma_3}$, we have
\begin{align*}
\int_{\mathbb{R}^d}Q[f]  \l v\r^{q}\,dv 
	& \le E_{q} +  D_{q} m_q
		- (\|b_2\|-\alpha_{\frac{q}{2}}) 2^{-1-\frac{\gamma_2}{2}} m_0\, m_{q+\gamma_2}
		-  \frac{1}{4}(\|b_3\|-\lambda_{\frac{q}{2}}) 3^{-\frac{\theta_3}{2}}\left(\frac{2}{3}\right)^{\frac{\gamma_3}{2}} m_0^2 \, m_{q+\gamma_3},
\end{align*}
where the coefficients $E_q$ and $D_q$ are given by
\begin{align*}
E_q & = \alpha_{\frac{q}{2}} C_{2,\frac{q}{2}} C_{\gamma_2}  \varepsilon_{q, \gamma_2}^{-\frac{\theta_{q,\gamma_2}}{1- \theta_{q,\gamma_2}}}  B_{q,\gamma_2}
	+ \lambda_{\frac{q}{2}} C_{3,\frac{q}{2}} C_{\gamma_3}   \varepsilon_{q, \gamma_3}^{-\frac{\theta_{q,\gamma_3}}{1- \theta_{q,\gamma_3}}}  B_{q,\gamma_3},\\
D_q & = \alpha_{\frac{q}{2}} C_{2,\frac{q}{2}} C_{\gamma_2} m_2 
		 +  (\|b_2\|-\alpha_{\frac{q}{2}}) m_{2}
		 + 2  \lambda_{\frac{q}{2}} C_{3,\frac{q}{2}} C_{\gamma_3} m_2^2
		 + (\|b_3\|-\lambda_{\frac{q}{2}})  3^{-\frac{\theta_3}{2}} m_{2}^2.
\end{align*}
Finally, by using the following simple estimate $E_q = \frac{E_q }{m_0} m_0 \le \frac{E_q }{m_0}  m_q$, we obtain
  \begin{align*}
\int_{\mathbb{R}^d}Q[f]\langle v\rangle^{q}\,dv
	\leq  C_q(m_0[f],m_2[f]) m_{q}[f]-C_q'(m_0[f]) (m_{q+\gamma_2}[f]+m_{q+\gamma_3}[f]),
\end{align*}
where 
\begin{align*}
C_q(m_0,m_2) &= \frac{E_q }{m_0}  + D_q,\\
C'_q(m_0)  &= \min\left\{  (\|b_2\|-\alpha_{\frac{q}{2}}) 2^{-1-\frac{\gamma_2}{2}} m_0, \,\, \frac{1}{4}(\|b_3\|-\lambda_{\frac{q}{2}})3^{-\frac{\theta_3}{2}} \left(\frac{2}{3}\right)^{\frac{\gamma_3}{2}} m_0^2 \right\}.
\end{align*}

Lastly, the estimate \eqref{power estimate q} follows by applying interpolation lemma, Lemma \ref{interpolation lemma}, in order to estimate $m_{q+\gamma_2}$ and $m_{q+\gamma_3}$ in terms of $m_2$ and $m_q$. This yields  \eqref{power estimate q} with $\tilde{C}_q = C_q' (m_2^{-\frac{\gamma_2}{q-2}} + m_2^{-\frac{\gamma_3}{q-2}})$.

\end{proof}

\section{Generation and propagation of polynomial moments}
\label{sec - polynomial moments}

The goal of this section is  to prove generation and propagation properties of higher order moments of  solutions to the binary-ternary Boltzmann equation \eqref{binary-ternary equation}  as formulated in Theorem \ref{polynomial moments theorem}.
In order to achieve this,  angular averaging estimates derived in Section \ref{sec - angular averaging est} are used to  first establish finiteness and differentiability of polynomial moments (as stated in Theorem \ref{finiteness of moments}).
This, in turn, will enable us to  obtain an ordinary differential inequality for moments  (see Proposition \ref{moments ode pre theorem q}) which will yield quantitative estimates of polynomial moments, including their propagation and generation in time, thanks to a comparison principle with a Bernoulli-type ODE.

\subsection{Phase 1: Finiteness and differentiability of moments}
\label{sec - phase 1}

Our first goal is to prove the following result on finiteness and differentiability of polynomial moments for  solutions to the binary-ternary Boltzmann equation.

\begin{theorem}\label{finiteness of moments}
Let $T>0$,  and let  $f\geq 0$ be a solution to the binary-ternary Boltzmann equation \eqref{binary-ternary equation} corresponding to the initial data $f_0 \in L^1_2$, $f_0\geq 0$. Then  $f \in  C^1((0,T], L^1_k)$ and $Q[f]\in C((0,T],L^1_k)$ for any $k > 2$.  Additionally, 
\begin{equation}\label{moments derivative}
m_k'[f]=\int_{\mathbb{R}^d}Q[f]\l v\r^k\,dv,\quad\forall t\in (0,T].
\end{equation}
\end{theorem}

The proof of the theorem is based on the following lemma, which establishes time integrability of moments via an inductive argument. The lemma is inspired by a related result in \cite{miwe99} for the homogeneous Boltzmann equation \eqref{binary equation}. However, in the proof of the lemma we use novel angular averaging estimates of Section \ref{sec - angular averaging est}.

\begin{lemma}\label{step}
Let $T>0$ and let $f\geq 0$ be a solution to the binary-ternary Boltzmann equation \eqref{binary-ternary equation} corresponding to the initial data $f_0 \in L^1_2$, $f_0\geq 0$. Then for any $0\le s<t<T$, any $k \ge 2+\frac{\gamma}{2}$, and any $\varepsilon \in (0, t-s)$, 
\begin{align}
	\int_s^t m_k(\tau) d\tau < \infty 
		\quad \Rightarrow \quad
		 m_{k-\frac{\gamma}{2}} (t) + \int_{s + \varepsilon}^t m_{k+\frac{\gamma}{2}}(\tau) d\tau < \infty.
 \label{new inductive step}
 \end{align}
Moreover, for any  $k>2$ and  any $0<s<t<T$, we have $m_k(t)<\infty$ and 
\begin{align}\label{time integral finite}
\int_{s}^tm_k(\tau)\,d\tau<\infty.
\end{align}
\end{lemma}

\noindent Before proving the lemma, we show how we use it to prove Theorem \ref{finiteness of moments}.

\begin{proof}[Proof of Theorem \ref{finiteness of moments}.]
Let $k >2$ and $t_0\in (0,T)$ arbitrarily small. Now,  Lemmata \ref{lemma on regularity of collisional operator binary}-\ref{lemma on regularity of collisional operator ternary}, the conservation laws, and estimate \eqref{time integral finite}  yield that $\int_{t_0}^T\|Q[f(\tau)]\|_{L^1_k}\,d\tau<\infty$, so $Q[f]\in L^1([t_0,T],L_k^1)$. Now, integrating \eqref{IVP binary-ternary} in time, testing with $\l v\r ^k$, and using Fubini's theorem,  we obtain

\begin{equation}\label{weak solution conclusion}
m_k(t)=m_k(t_0)+\int_{t_0}^t \int_{\mathbb{R}^d}Q[f]\l v\r^k\,dv\,d\tau,\quad\forall t\in[t_0,T].
\end{equation}
In particular, $f\in C([t_0,T],L^1_{k})$. Since $k > 2$ was arbitrary, we have $f\in C([t_0,T],L^1_{k})$, for all $k >2$, which in turn implies (by Lemmata \ref{lemma on regularity of collisional operator binary}-\ref{lemma on regularity of collisional operator ternary} respectively)  that $Q[f]\in C([t_0,T],L^1_k)$, for all $k>2$. Then, differentiating \eqref{weak solution conclusion}, we obtain \eqref{moments derivative} in $[t_0,T]$ and that $f\in C^1([t_0,T],L^1_k)$.
Since $t_0$ was chosen arbitrarily small,  \eqref{moments derivative} holds $(0,T]$ and $f\in C^1((0,T],L^1_k)$.
\end{proof}

 We next prove  Lemma \ref{step}.

\begin{proof}[Proof of Lemma \ref{step}]
Let $0 \le s < t<T$ , $\varepsilon \in (0, t-s)$, $k \ge 2+\frac{\gamma}{2}$ and assume that
$
	\int_s^t m_k(\tau) d\tau < \infty.
$
Then, by the monotonicity of moments 
\begin{align}
		\int_s^t m_{k - \frac{\gamma}{2}}(\tau) \, d\tau   < \infty,
		\label{induction hypothesis}
\end{align}
and so there exists $s_0$ such that $0 <s<s_0< s+\varepsilon <t$  and  $m_{k - \frac{\gamma}{2}} (s_0)  < \infty$.
Let 
$
		\psi(x) =  :x^{\frac{k}{2} - \frac{\gamma}{4}}.
$
 Since $\frac{k}{2} - \frac{\gamma}{4}\ge1$, this is a convex function, so by Lemma \ref{approximation lemma},  there exists a sequence of functions $\psi_n \nearrow \psi$, defined by
\begin{equation*} 
\psi_n(x)=\begin{cases}
\psi(x),\quad x\leq n\\
p_n(x),\quad x>n,
\end{cases}
\end{equation*}
 where $p_n$ is a polynomial of degree one given by $p_n(x)=\psi'(n)x+\psi(n)-n\psi'(n)$. Since for each $n \in \mathbb{N}$, 
 $\psi_n(\l v \r^2) \le C_n \l v \r^2$, and $f$ is a solution to the binary-ternary Boltzmann equation, by Lemmata  \ref{lemma on regularity of collisional operator binary}-\ref{lemma on regularity of collisional operator ternary}, $\psi_n(\l v \r^2)$ can be used as a test function in the weak formulation to obtain
\begin{align}\label{step with K}
	 \int_{\mathbb{R}^d}  (f(t,v) - f(s_0, v)) \, \psi_n(\l v \r^2) dv 
	 &= \frac{1}{2}  \int_{s_0}^t  \! \int_{\mathbb{R}^{2d}} f f_1 |v-v_1|^{\gamma_2} K_{2,\psi_n}(v,v_1) \,	dv dv_1 d\tau \nonumber\\
	 & \qquad +  \frac{1}{6}  \int_{s_0}^t \! \int_{\mathbb{R}^{3d}} f f_1 f_2 |\tild{\bm{u}}|^{\gamma_3 - \theta_3} |\bm{u}|^{\theta_3}K_{3,\psi_n}(v,v_1) \,	dv dv_{1,2} d\tau,
\end{align}
where $K_{2,\psi_n}$ and $K_{3,\psi_n}$ are defined as in \eqref{K binary} and \eqref{K ternary}.
Using Lemma \ref{new decomposition lemma} and Lemma \ref{new decomposition lemma ternary}, we write  $K_{2,\psi_n} = \tild{G}_{2,\psi_n} - \tild{L}_{2,\psi_n}$ and $K_{3,\psi_n} = \tild{G}_{3,\psi_n} - \tild{L}_{3,\psi_n}$, where these quantities are defined as in  \eqref{new gain poly},  \eqref{new loss poly}, \eqref{new gain poly ternary} and  \eqref{new loss poly ternary}. Next, we show that integrals $ \int_{s_0}^t \! \int_{\mathbb{R}^{2d}} f f_1 |v-v_1|^{\gamma_2} \tild{L}_{2,\psi_n}   dv dv_1 d\tau$ and $\int_{s_0}^t \!\int_{\mathbb{R}^{3d}} f f_1 f_2 |\tild{\bm{u}}|^{\gamma_3 - \theta_3} |\bm{u}|^{\theta_3}\tild{L}_{3,\psi_n} 
\, dv dv_{1,2} d\tau $
are finite, and thus can be added to both sides of the above equation. Namely,
\begin{align*}
 	 \int_{s_0}^t \! \int_{\mathbb{R}^{2d}}  f f_1 |v-v_1|^{\gamma_2} \tild{L}_{2,\psi_n}   dv dv_1
		& \le  \int_{s_0}^t \!\int_{\mathbb{R}^{2d}} f f_1 |v-v_1|^{\gamma_2} L_{2,\psi_n}   dv dv_1\\
		& = \| b_2 \| \, \int_{s_0}^t \!  \int_{\mathbb{R}^{2d}} f f_1 |v-v_1|^{\gamma_2} 
			\Big(\psi_n(\l v\r^2)  + \psi_n(\l v_1\r^2)\Big)   dv dv_1.
\end{align*}
Due to the symmetry of this expression with respect to $v \leftrightarrow v_1$, it suffices to show that the integral 
$ \int_{s_0}^t \! \int_{\mathbb{R}^{2d}} f f_1 |v-v_1|^{\gamma_2} \psi_n(\l v\r^2)  dv dv_1$ is finite. In order to establish this, we use the definition of the approximation function $\psi_n$ given in \eqref{psi_n} to obtain
\begin{align*}
	\int_{\mathbb{R}^{2d}} & f f_1 |v-v_1|^{\gamma_2} \psi_n(\l v\r^2)  dv dv_1\\
		&= \int_{\mathbb{R}^{d}} \int_{\l v\r^2 \le n } f f_1 |v-v_1|^{\gamma_2} \psi(\l v\r^2)  dv dv_1
			+ \int_{\mathbb{R}^{d}}  \int_{\l v\r^2 > n }  f f_1 |v-v_1|^{\gamma_2} p_n(\l v\r^2)  dv dv_1\\
		& \le n^k  \int_{\mathbb{R}^{d}} \int_{\l v\r^2 \le n } f f_1 |v-v_1|^{\gamma_2} dv dv_1
			+ \int_{\mathbb{R}^{d}}  \int_{\l v\r^2 > n }  f f_1 |v-v_1|^{\gamma_2} (A \l v \r^2 + B)  dv dv_1,
\end{align*}
where $A, B$ are coefficients (that depend on $n$) of the first order polynomial $p_n$ defined in \eqref{p_n}. Therefore,
\begin{align*}
	 \int_{s_0}^t \!  \int_{\mathbb{R}^{2d}}  f f_1 |v-v_1|^{\gamma_2} \psi_n(\l v\r^2)  dv dv_1
		\le C m_2  \int_{s_0}^t \, m_{2+\gamma_2}(\tau) d\tau,
\end{align*}
where $C>0$ is a constant that depends on $n$ and $\gamma_2$.  This is a finite quantity since $f$ is solution. Similarly, one can show that $\int_{s_0}^t \!\! \int_{\mathbb{R}^{3d}} f f_1 f_2 |\tild{\bm{u}}|^{\gamma_3 - \theta_3} |\bm{u}|^{\theta_3}\tild{L}_{3,\psi_n} \, dv dv_{1,2} d\tau $ is finite. Additionally, by an analogous domain-splitting one can show that
$\int_{\R^d} f(s_0, v) \psi_n(\l v \r^2) \le C m_2(s_0) \le C m_{k - \frac{\gamma}{2}}(s_0) <\infty$ by the choice of time $s_0$, where $C>0$ depends on $n$ and $\gamma_2$. 
Therefore, adding these finite integrals to both sides of the equation \eqref{step with K} yields
\begin{align*}
	&  \int_{\mathbb{R}^d} f(t,v) \psi_n(\l v \r^2) dv 
		 + \frac{1}{2} \int_{s_0}^t \! \int_{\mathbb{R}^{2d}} f f_1 |v-v_1|^{\gamma_2} \tild{L}_{2,\psi_n}   dv dv_1 d\tau\\
		& \qquad \qquad \qquad \qquad \quad
			+ \frac{1}{6} \int_{s_0}^t \! \int_{\mathbb{R}^{3d}} f f_1 f_2 
			|\tild{\bm{u}}|^{\gamma_3 - \theta_3} |\bm{u}|^{\theta_3}\tild{L}_{3,\psi_n} 
			\, dv dv_{1,2} d\tau  \\
		& =  \int f(s_0,v) \psi_n(\l v \r^2)  dv 
			+ \frac{1}{2} \int_{s_0}^t \! \int_{\mathbb{R}^{2d}} f f_1 
				|v-v_1|^{\gamma_2}  \tild{G}_{2,\psi_n}  dv dv_1 d\tau \\
		&\qquad \qquad  \qquad \qquad \quad	
			+ \frac{1}{6} \int_{s_0}^t \!  \int_{\mathbb{R}^{3d}} f f_1 f_2
				|\tild{\bm{u}}|^{\gamma_3 - \theta_3} |\bm{u}|^{\theta_3}
				 \tild{G}_{3,\psi_n}  dv dv_{1,2} d\tau. 
\end{align*}
In the rest of the proof we abuse notation and denote by $C_k$ various positive constants that depend on $k, \gamma, \gamma_2, \gamma_3, b_2, b_3, m_0$ and $m_2$.

Since $\frac{k}{2} - \frac{\gamma}{4}\ge1$, estimates \eqref{bound on binary approx gain} and \eqref{bound on approx gain ternary} can be used to obtain
\begin{align*}
	&  \int_{\mathbb{R}^d} f(t,v) \psi_n(\l v \r^2) dv 
		 + \frac{1}{2} \int_{s_0}^t \! \int_{\mathbb{R}^{2d}} f f_1 |v-v_1|^{\gamma_2} \tild{L}_{2,\psi_n} 
		+ \frac{1}{6} \int_{s_0}^t \! \int_{\mathbb{R}^{3d}} f f_1 f_2 
			|\tild{\bm{u}}|^{\gamma_3 - \theta_3} |\bm{u}|^{\theta_3}\tild{L}_{3,\psi_n} 
			\, dv dv_{1,2} d\tau \\
		& \le  \int_{\mathbb{R}^d} f(s_0,v) \psi_n(\l v \r^2)  dv 
			+ C_k\int_{s_0}^t \! \int_{\mathbb{R}^{2d}} f f_1 
				|v-v_1|^{\gamma_2}   \l v \r \l v_ 1\r 
			\left( \l v_ 1\r^{k-\frac{\gamma}{2}-2} + \l v \r^{k-\frac{\gamma}{2} -2}   \right) dv dv_1 d\tau\\
		& \quad	+C_k\int_{s_0}^t \!\int_{\mathbb{R}^{3d}} f f_1 f_2
				|\tild{\bm{u}}|^{\gamma_3} 
				\left(  \l v \r \l v_ 1\r +\l v \r \l v_ 2\r +\l v_1 \r \l v_ 2\r  \right)
			  \sum_{i=0,1,2} \l v_ i\r^{k-\frac{\gamma}{2}-2}  dv dv_{1,2} d\tau.
\end{align*}
 Since
$\tild{L}_{2,\psi_n},  \tild{L}_{3,\psi_n}  \ge 0$, and $ \tild{L}_{2,\psi_n} \to  \tild{L}_{2,\psi}$,  $\tild{L}_{3,\psi_n} \to  \tild{L}_{3,\psi} $,  Fatou's lemma can be applied on the integrals containing $\tild{L}_{2,\psi_n}$ and  $\tild{L}_{3,\psi_n}$.  Therefore, using the monotone convergence theorem on the other two terms containing $\psi_n$, letting $n \to \infty$ yields
\begin{align}
	&  \int_{\mathbb{R}^d} f(t,v) \psi(\l v \r^2) dv 
		 + \frac{1}{2} \int_{s_0}^t \!\int_{\mathbb{R}^{2d}} f f_1 |v-v_1|^{\gamma_2} \tild{L}_{2,\psi} 
		+ \frac{1}{6} \int_{s_0}^t \!\int_{\mathbb{R}^{3d}} f f_1 f_2 
			|\tild{\bm{u}}|^{\gamma_3 - \theta_3} |\bm{u}|^{\theta_3}\tild{L}_{3,\psi} 
			\, dv dv_{1,2} d\tau \nonumber\\
		& \le  \int_{\mathbb{R}^d} f(s_0,v) \psi(\l v \r^2)  dv 
			+ C_k\int_{s_0}^t \! \int_{\mathbb{R}^{2d}} f f_1 
				|v-v_1|^{\gamma_2}   \l v \r \l v_ 1\r 
			\left( \l v_ 1\r^{k- \frac{\gamma}{2} -2} + \l v \r^{k - \frac{\gamma}{2} -2}   \right) dv dv_1 d\tau \nonumber\\
		& \quad	+ C_k \int_{s_0}^t \! \int_{\mathbb{R}^{3d}} f f_1 f_2
				|\tild{\bm{u}}|^{\gamma_3} 
				\left(  \l v \r \l v_ 1\r +\l v \r \l v_ 2\r +\l v_1 \r \l v_ 2\r  \right)
			  \sum_{i=0,1,2} \l v_ i\r^{k-\frac{\gamma}{2}-2}  dv dv_{1,2} d\tau. 
			\label{induction step limit}
\end{align}
Using the upper bounds \eqref{binary potential upper bound} and  \eqref{ternary potential upper bound}  on the potentials $|v-v_1|^{\gamma_2}$ and  $|v-v_1|^{\gamma_3}$, one obtains
\begin{align}
	& \int_{\mathbb{R}^{2d}} f f_1 |v-v_1|^{\gamma_2}   \l v \r \l v_ 1\r 
			\left( \l v_ 1\r^{k - \frac{\gamma}{2} -2} + \l v \r^{k - \frac{\gamma}{2} -2}   \right) dv dv_1 
		 \le C_k  \, m_{k-1+\frac{\gamma}{2}},
		\label{induction step gain binary} \\
	& \int_{\mathbb{R}^{3d}}  f f_1 f_2
				|\tild{\bm{u}}|^{\gamma_3} 
				\left(  \l v \r \l v_ 1\r +\l v \r \l v_ 2\r +\l v_1 \r \l v_ 2\r  \right)
			 \sum_{i=0,1,2} \l v_ i\r^{k-\frac{\gamma}{2}-2} dv dv_{1,2}
		 \le  C_k  \, m_{k-1+\frac{\gamma}{2}}. \label{induction step gain ternary}
\end{align}
On the other hand, lower bounds \eqref{bound on poly loss} and \eqref{binary potential lower bound} yield the following lower bound:
\begin{align}
	&  \int_{\mathbb{R}^{2d}} f f_1 |v-v_1|^{\gamma_2} \tild{ L}_{2,\psi} \, dv dv_1 \nonumber \\
		& \quad \ge (\|b_2\| -\alpha_{\frac{k}{2}})   \int_{\mathbb{R}^{2d}} \! f f_1 \!
				\left( 2^{\frac{\gamma_2}{2}} \sum_{i=0,1} \l v_ i\r^{k-\frac{\gamma}{2} + \gamma_2}
						- \l v \r^{\gamma_2} \l v_1\r^{k-  \frac{\gamma}{2}}
						- \l v\r^{k-  \frac{\gamma}{2}} \l v_1\r^{\gamma_2}
					\!\! \right) \! dv dv_1 \nonumber \\
		& \quad \ge  C_k \, 
			\left(m_0 \, m_{k-  \frac{\gamma}{2}+\gamma_2 }\, - \, \,\, m_{\gamma_2} \, m_{k-  \frac{\gamma}{2}} \right),
			\label{induction step loss binary}
\end{align}
while bounds \eqref{bound on poly loss ternary}, \eqref{ternary potential lower bound} and \eqref{equivalence of relative velocities ternary} imply
\begin{align}
	&  \int_{\mathbb{R}^{3d}} f f_1 f_2 
			|\tild{\bm{u}}|^{\gamma_3 - \theta_3} |\bm{u}|^{\theta_3}\tild{L}_{3,\psi} 
			\, dv dv_{1,2} \nonumber\\
		&\quad \ge (\|b_3\| -\lambda_{k/2}) 3^{\frac{-\theta_3}{2}} 
			\int_{\mathbb{R}^{3d}} f f_1 f_2 
			\sum_{i=0,1,2} \left( \left( \frac{2}{3}\right)^{\frac{\gamma_3}{2}}  \l v_i \r^{\gamma_3}
			 	 - \sum_{j \in\{ 0,1,2\}\setminus\{i\}} \l v_j \r^{\gamma_3} \right)  \l v_i \r^{k-\frac{\gamma}{2}} \nonumber\\
		& \quad \ge  C_k\,\,
			\left( m_0 \, m_{k-  \frac{\gamma}{2}+\gamma_3 }\, - \,  \,\, m_{\gamma_3} \, m_{k-  \frac{\gamma}{2}}\right)  .
			\label{induction step loss ternary}
\end{align}
Combining estimates \eqref{induction step gain binary} - \eqref{induction step loss ternary} with \eqref{induction step limit} yields
\begin{align*}
 	 & m_{k- \frac{\gamma}{2}}(t) + C_{k}
	 	 \int_{s_0}^t \left(m_{k-  \frac{\gamma}{2}+ \gamma_2 }+ m_{k-  \frac{\gamma}{2}+\gamma_3 } \right)d\tau 
\le  m_{k- \frac{\gamma}{2}}(s_0) + C_{k} \int_{s_0}^t  m_{ k} \, d\tau.
\end{align*}
Therefore, since one of $\gamma_2$ or $\gamma_3$ coincides with $\gamma$, we have
\begin{align*}
 	 & m_{k- \frac{\gamma}{2}}(t) + C_{k}
	 	 \int_{s_0}^t m_{k+ \frac{\gamma}{2}}d\tau 
 \le  m_{k- \frac{\gamma}{2}}(s_0) + C_{k} \int_{s_0}^t  m_{ k}\,  d\tau < \infty,
\end{align*}
which proves \eqref{new inductive step}.

Finally, to prove \eqref{time integral finite}, let $k >2$ and fix any $0<s<t$.  Let $n \in \mathbb{N}$ be the smallest positive integer so that $k \le 2 + \frac{n \gamma}{2}$, and choose $\varepsilon_0>0$ so that 
$\frac{s}{2} + n \, \varepsilon_0 \le s$.
By the definition of solutions 
\begin{align*}
	\int_{\frac{s}{2}}^t m_{2+\gamma}(\tau) \, d\tau < \infty.
\end{align*}
Then by \eqref{new inductive step}, we have
\begin{align*}
	 m_{2+\frac{\gamma}{2}} (t) + \int_{\frac{s}{2} + \varepsilon_0}^t m_{2+\frac{3\gamma}{2}}(\tau) d\tau < \infty.
\end{align*}
In fact, applying \eqref{new inductive step} inductively yields
\begin{align*}
	 m_{2+\frac{n\gamma}{2}} (t) + \int_{\frac{s}{2} + n\varepsilon_0}^t m_{2+\frac{(n+2)\gamma}{2}}(\tau) d\tau < \infty.
\end{align*}
Since $k \le 2 + \frac{(n+2) \gamma}{2}$ and $\frac{s}{2} + n \, \varepsilon_0 \le s$, we conclude that $m_k(t) <\infty$ and $ \int_{s}^t m_{k}(\tau) d\tau < \infty.$
\end{proof}

\subsection{Phase 2: Quantitative estimates of moments} 
\label{sec - phase 2}
Now that Phase 1 is completed and finiteness of moments is established, we proceed to prove quantitative moment estimates on  generation and propagation  in time of polynomial moments of  solutions to the binary-ternary Boltzmann equation as  stated in  Theorem \ref{polynomial moments theorem}. The proof relies on the already established finiteness and differentiability of moments Theorem \ref{finiteness of moments}, as well as the estimate on the weak form of the collision operator Proposition \ref{moments ode pre theorem q}.

\subsection*{Proof of Theorem \ref{polynomial moments theorem}} 
Without loss of generality, we assume that $m_0[f_0]>0$, otherwise by the conservation of mass the only solution is zero, so the claim trivially holds.
\begin{proof}[Proof of  (i):]
 Fix $q>2$. 
 By Theorem \ref{finiteness of moments}, $f\in C^1((0,T], L^1_{q+\gamma})$.    Testing \eqref{binary-ternary equation}  against $\l v\r^q$ and integrating, differentiability of $m_q$ and estimate \eqref{power estimate q}  yield
\begin{equation}\label{full polynomial inequality}
	m_q'[f](t)\le C_q m_{q}[f](t)- \tild{C}_q \left(m_{q}[f](t)^{1+\frac{\gamma_2}{q-2}}+m_{q}[f](t)^{1+\frac{\gamma_3}{q-2}}\right).
\end{equation}
In particular, \eqref{full polynomial inequality} implies
\begin{equation}\label{modified polynomial inequality}
	m_q'[f](t)\le C_q m_{q}[f](t)- \tild{C}_q m_{q}[f](t)^{1+\frac{\gamma_i}{q-2}},
\end{equation}
for each $i\in\{2,3\}$.
Now by Lemma 3.8 in \cite{lumo12},  $m_q[f]$ satisfying \eqref{modified polynomial inequality}, with the additional constraint that $\gamma_i >0$, is a sub-solution to the Bernoulli-type initial value problem: 
\begin{equation}\label{Bernoulli IVP prop}
	\begin{cases}
		y'=C_q y-\tild{C}_q y^{1+\frac{\gamma_i}{q-2}},\quad t>0 \\
		\lim_{t\to 0^+} y(t)=+\infty, 
	\end{cases}
\end{equation}
and we have 
\begin{align}\label{first bound on m_a}
m_{q}[f](t)\ 
	\leq \ y(t)	&=\left(
		\frac{\tild{C}_q}{C_q}\right)^{\frac{2-q}{\gamma_i}}\left(1-e^{-t C_q \frac{\gamma_i}{q-2}}\right)^{\frac{2-q}{\gamma_i}},\quad t> 0.
\end{align}

For $t>1$, estimate \eqref{first bound on m_a} implies
$$
m_q[f](t)
	\leq \left(
		\frac{\tild{C}_q}{C_q}\right)^{\frac{2-q}{\gamma_i}}\left(1-e^{- C_q \frac{\gamma_i}{q-2}}\right)^{\frac{2-q}{\gamma_i}}.
$$
If $t\leq 1$, then for any $A>0$, we have $1-e^{-tA} \ge At e^{-A}$. In particular, 
for $A= \frac{C_q\gamma_i}{q-2}$, we have 
$$
	1-e^{- tC_q \frac{\gamma_i}{q-2}}\geq \frac{C_q\gamma_i}{q-2} t e^{-\frac{C_q\gamma_i}{q-2}} 
$$ 
for all $t\in (0,1]$. Therefore, \eqref{first bound on m_a} implies that for all $t\in (0,1]$ we have
$$
	m_q[f](t)
		\leq \left(\frac{\tild{C}_q}{C_q}\right)^{\frac{2-q}{\gamma_i}} e^{C_q} \left(\frac{C_q \gamma_i}{q-2} t \right)^{\frac{2-q}{\gamma_i}}
$$
Defining
\begin{align}\label{K_qi}
	K_{q,i} = \left(\frac{\tild{C}_q}{C_q}\right)^{\frac{2-q}{\gamma_i}}
		\max\left\{\left(1-e^{- \frac{C_q\gamma_i}{q-2}}\right)^{\frac{2-q}{\gamma_i}},  
				\, e^{C_q} \left( \frac{C_q \gamma_i}{q-2} \right)^{\frac{2-q}{\gamma_i}}\right\},
\end{align}
 we obtain 
$$
	m_q[f](t)\leq K_{q,i}\max\{1,t^{\frac{2-q}{\gamma_i}}\}.
$$
Therefore, estimate \eqref{polynomial generation estimate gamma_i>0} has been shown.
Now if both $\gamma_2,\gamma_3>0$, we have
$$ 
	m_q[f](t)\leq K_{q,i}\max\{1,t^{\frac{2-q}{\gamma_i}}\},\quad\forall t>0,\quad i=2,3,
$$
which implies bound \eqref{polynomial generation estimate} for $K_q=\max\{K_{q,2},K_{q,3}\}$.

{\it Proof of (ii):} 
 Now assume $m_q(0)<\infty$. To control the behavior of $m_q$ for $t\in[0,\min\{1,T\}]$, we will use the fact that $m_q$ is initially finite. Indeed, estimate \eqref{full polynomial inequality} yields
$$m_q'[f](t)\le C_q m_{q}[f](t)- \tild{C}_q m_{q}[f](t)^{1+\frac{\gamma}{q-2}},$$
where $\gamma=\max\{\gamma_2,\gamma_3\}>0$. Consider now the Bernoulli IVP 
\begin{equation*}
	\begin{cases}
		y'=C_q y-\tild{C}_q y^{1+\frac{\gamma}{q-2}},\quad t \in[0,\min\{1,T\}] \\
		y(0)=m_q(0), 
	\end{cases}
\end{equation*}
which has the solution 
\begin{align*}
y(t)&=\left(m_q(0)^{-\frac{\gamma}{q-2}}e^{-tC_q\frac{\gamma}{q-2}}+\frac{\tild{C}_q}{C_q}\left(1-e^{-tC_q \frac{\gamma}{q-2}}\right)\right)^{\frac{2-q}{\gamma}} \leq m_q(0) e^{tC_q}.
\end{align*}
By the comparison principle, we have $m_q[f](t)\leq y(t)$, thus 
\begin{equation}\label{prop t<=1}
\sup_{t\in [0,\min\{1,T\}]}m_q(t)\leq m_q(0)e^{C_q}.
\end{equation}
If $T\le 1$, this completes the proof of the lemma. Now if, $T>1$ by the generation estimate \eqref{polynomial generation estimate}, we have that 
\begin{equation}\label{prog t>1}
\sup_{1<t\le T} m_q(t)\leq K_q,
\end{equation}
which implies estimate \eqref{polynomial propagation estimate} for $M_q=\max\{m_q(0)e^{C_q},K_q\}$.
\end{proof}

\bigskip
\section{Generation and propagation of exponential moments}
\label{sec - exponential moments}

In this section we prove generation and propagation of exponential moments of  solutions to the spatially homogeneous  binary-ternary Boltzmann equation \eqref{binary-ternary equation}  as formulated in Theorem  \ref{exponential moments theorem}. Our proof is inspired by \cite{alcagamo13}, where an analogous result was established for the  homogeneous  binary Boltzmann equation \eqref{binary equation}. As \cite{alcagamo13}, we rely on angular averaging estimates, but we use those derived specifically for the binary-ternary operator in  Section \ref{sec - angular averaging est}. In fact, angular averaging estimates are first used  to obtain an upper bound  on the binary-ternary collision operator $Q[f]$, \eqref{binary-ternary operator} (see Lemma \ref{sp moment ODI}) for a general function which is not necessarily a  solution to the binary-ternary Boltzmann equation, and in addition to being used in this section, it  will play an important  role in the proof of well-posedness in Section \ref{sec - existence}. We note that compared to Section \ref{sec - polynomial moments} where there were two phases of the proof, since finiteness of moments is already established (see Theorem \ref{finiteness of moments}), in this section there is only one phase of proving the quantitative propagation and generation of exponential moments.

Recall the definition of exponential moments \eqref{definition of exponential moment}, and as first exploited by Bobylev \cite{bo97}, note that they can be expressed as an infinite sum of weighted polynomial moments thanks the the Taylor expansion of the exponential weight $e^{z \l v \r^s}$ as follows
\begin{align*}
E_s(t,z) = \sum_{p=0}^\infty m_{sp}(t) \frac{z^p}{p!}.
\end{align*}
Let us denote the partial sum of this expansion by  $E_s^n$, and the partial sum shifted by a parameter $\tilde{\gamma}>0$ by $I_{s, \tilde{\gamma}}^n$, where 
\begin{align}\label{E}
E_s^n(t,z) &:= \sum_{p=0}^n m_{sp}(t) \frac{z^p}{p!},\\ 
I_{s, \tilde{\gamma}}^n(t,z) &:= \sum_{p=0}^n m_{sp+\tilde\gamma}(t) \frac{z^p}{p!}. \label{I}
\end{align}
Throughout the paper, the shift $\tilde{\gamma}$ will be  $\gamma_2$ or $\gamma_3$.

We also use the following notation for binomial coefficients
\begin{align}\label{binomial coefficient}
\displaystyle\binom{p}{k,k_1}=\frac{p!}{k!k_1!},
\end{align}
 where $p=k+k_1$, in order to observe the similarity with the calculations for the ternary term. Similarly, we use the notation for trinomial coefficients, for $p = k+k_1+k_2$,
 \begin{align}\label{trinomial coefficient}
 \displaystyle \binom{p}{k,k_1,k_2}=\frac{p!}{k!k_1!k_2!}.
 \end{align}

 Due to the expansion \eqref{E} and the use of averaging Lemmata \ref{binary povzner}, \ref{ternary povzner}, in this section we work with polynomial weights of order $sp>2$, where $p$ is an integer and $s \in (0,2]$. The following lemma provides an estimate of the collision operator integrated against a polynomial weight of order $sp$.

 \begin{lemma}\label{sp moment ODI} 
Consider the binary-ternary collision operator \eqref{binary-ternary operator} with \eqref{cross-section 2}-\eqref{binary cut-off} and \eqref{cross-section 3}-\eqref{ternary cut-off}. Let  $s\in (0,2]$, $p \in \mathbb{N}$  with $sp>2$, and suppose a non-negative function $f \in L^1_{sp + \gamma} $ conserves mass and energy, i.e. for some $0<m_0<m_2<\infty$ and any $t \ge 0$,   $m_0(t)=m_0,  \, m_2(t)=m_2$. Then the following estimate holds:
\begin{align}
	\int_{\mathbb{R}^d}Q[f]\l v\r^{sp}\,dv
		\leq   \, - K_{1, sp} m_{sp+\gamma_2} \,- \, K_{2, sp}  m_{sp+\gamma_3}
			+K_{3, sp} m_{sp}
			+  2 C_{\gamma_2} \alpha_{sp/2} S^p_{2,s,\gamma_2} \, 
			+  3C_{\gamma_3} \lambda_{sp/2} S^p_{3,s,\gamma_3},  \label{moments ineq}
\end{align}
where
\begin{align}
	K_{1, sp} &=   2^{1-\frac{\gamma_2}{2}} m_0\,(\|b_2\|-\alpha_{sp/2}), \nonumber\\ 
	K_{2, sp} &=  3 \, \left(\frac{2}{3}\right)^{\gamma_3/2} \, m_0^2 \, (\|b_3\|-\lambda_{sp/2}) \label {Ks}, \\
	K_{3, sp} &= 2 m_2 \,(\|b_2\|-\alpha_{sp/2})  \, + \, 6 m_0 m_2 \, (\|b_3\|-\lambda_{sp/2}) \nonumber
\end{align} 
are positive constants that are increasing in $sp$, and where
\begin{align}
S^p_{2,s,\gamma_2} &=  \sum_{\substack{0<k,k_1<p\\k+k_1=p}}\binom{p}{k,k_1} m_{sk+\gamma_2}\, m_{sk_1}\label{S_2}, \\
S^p_{3,s,\gamma_3} & =  \sum_{\substack{ 0\le k, k_1, k_2 < p \\ k+k_1+k_2 = p}} \binom{p}{k, k_1, k_2}  
	\, m_{sk +\gamma_3} \, m_{sk_1} \, m_{sk_2}. \label{S_3}
\end{align}
\end{lemma}

 \begin{proof}
By the properties of the binary cross-section \eqref{conservation of relative velocities binary} -  \eqref{binary cut-off} and  the binary angular averaging Lemma \ref{binary povzner},   we  have
\begin{align}
G_{2,sp}(v,v_1)&:= \int_{\mathbb{S}^{d-1}}B_2(u,\omega)\left(\langle v'\rangle^{sp}+\langle v'_1\rangle^{sp}-\langle v\rangle^{sp}-\l v_1\r^{sp}\right)\,d\omega\nonumber\\
& \le |u|^{\gamma_2} \left(  \alpha_{sp/2}  (E_2^{sp/2} - \l v\r ^{sp}-\l v_1\r^{sp})   - (\|b_2\|- \alpha_{sp/2}) (\l v\r^{sp} + \l v_1\r^{sp})\right).\label{gain bound binary}  
\end{align}
Similarly, properties of the ternary cross-section \eqref{cross-section 3}, \eqref{conservation of relative velocities ternary}-\eqref{ternary cut-off} and the ternary angular averaging Lemma \ref{ternary povzner} imply that
\begin{align}
&G_{3,sp}(v,v_1,v_2) :=
 \int_{\mathbb{S}_1^{2d-1}}B_3(\bm{u},\bm{\omega})\left(\l v^*\r^{sp}+\l v_1^*\r^{sp}+\l v_2^*\r^{sp}-\l v\r^{sp}-\l v_1\r^{sp}-\l v_2\r^{sp}\right)\,d\bm{\omega}\nonumber\\
&\le |\bm{\tild{u}}|^{\gamma_3}\left[\left(\lambda_{sp/2}(E_3^{sp/2}-\l v\r^{sp}-\l v_1\r^{sp}-\l v_2\r^{sp})-(\|b_3\|-\lambda_{sp/2})(\l v\r^{sp}+\l v_1\r^{sp}+\l v_2\r^{sp})\right)\right].\label{gain bound ternary}
\end{align}

By the weak form \eqref{weak form full equation}, estimates \eqref{gain bound binary}-\eqref{gain bound ternary}, and symmetry of $\bm{|\tild{u}}|$ with respect to renaming velocities, for a general non-negative function $f \in L^1_{sp + \gamma} $ conserves mass and energy, we obtain
\begin{align}
\int_{\mathbb{R}^d} &Q[f]\l v\r^{sp}\,dv
	=  \int_{\mathbb{R}^{2d}} ff_1G_{2,sp}(v,v_1)\,dv\,dv_1 
	+ \int_{\mathbb{R}^{3d}} ff_1f_2G_{3,sp}(v,v_1,v_2)\,dv\,dv_{1,2}\nonumber\\
&\leq \alpha_{sp/2} \int_{\mathbb{R}^{2d}}|u|^{\gamma_2} ff_1(E_2^{sp/2}-\l v\r^{sp}-\l v_1\r^{sp})\,dv\,dv_1	
	- 2(\|b_2\|-\alpha_{sp/2}) \int_{\mathbb{R}^{2d}}|u|^{\gamma_2} ff_1\l v\r^{sp}\,dv\,dv_1\nonumber\\
	& \qquad + \lambda_{sp/2} \int_{\mathbb{R}^{3d}}|\bm{\tild{u}}|^{\gamma_3} ff_1f_2
		(E_3^{sp/2}-\l v\r^{sp}-\l v_1\r^{sp}-\l v_2\r^{sp})\,dv\,dv_{1,2}\nonumber\\
& \qquad -3(\|b_3\|-\lambda_{sp/2}) \int_{\mathbb{R}^{3d}}|\bm{\tild{u}}|^{\gamma_3} ff_1f_2\l v\r^{sp}\,dv\,dv_{1,2}\nonumber\\
&:=  \alpha_{sp/2} A_1 - 2 (\|b_2\|-\alpha_{sp/2})A_2 + \lambda_{sp/2}B_1-3(\|b_3\|-\lambda_{sp/2})B_2.\label{first moments ineq}
\end{align}

\textbf{Upper bound on  $A_1$:}
Since $s\le2$, $(\l v\r^2+\l v_1\r^2)^{s/2}\leq \l v\r^s+\l v_1\r^s$, so by the binomial theorem,
\begin{equation*}
E_2^{sp/2}-\l v\r^{sp}-\l v_1\r^{sp}
\leq (\l v\r^s+\l v_1\r^s)^p-\l v\r^{sp}-\l v_1\r^{sp}
=\sum_{\substack{0<k,k_1<p\\k+k_1=p}}\binom{p}{k,k_1}\l v\r^{sk}\l v_1\r^{sk_1}.
\end{equation*}
Therefore,  \eqref{binary potential upper bound} yields, with $C_{\gamma_2} = \max\{1, 2^{\gamma_2-1}\}$,
\begin{align}
A_1
&\le C_{\gamma_2} \sum_{\substack{0<k,k_1<p\\k+k_1=p}}\binom{p}{k,k_1}\int_{\mathbb{R}^{2d}} 
	f f_1 \left( \l v\r^{sk+\gamma_2}  \l v_1\r^{sk_1} \, + \, \l v\r^{sk}  \l v_1\r^{sk_1+\gamma_2} \right)   \,dv\,dv_1  \nonumber \\
& =2C_{\gamma_2}  \sum_{\substack{0<k,k_1<p\\k+k_1=p}}\binom{p}{k,k_1} m_{sk+\gamma_2} \, m_{sk_1}. \label{bound on A_1}
\end{align}

\textbf{Lower bound on $A_2$:}
Using \eqref{binary potential lower bound}, we obtain
\begin{align}
A_2 
& =  \int_{\mathbb{R}^{2d}}|u|^{\gamma_2} ff_1\l v\r^{sp}\,dv\,dv_1 
\geq  \frac{1}{2^{\gamma_2/2}} m_{sp+\gamma_2} \, m_0 \, - \, m_{sp}\, m_{\gamma_2}.
\label{bound on A_2}
\end{align}

\textbf{Upper bound on $B_1$:}
Since $s\le2$, we have $(\l v\r^2+\l v_1\r^2 + \l v_2\r^2)^{s/2} \leq \l v\r^s+\l v_1\r^s + \l v_2\r^s$, therefore using the multinomial expansion, we have
\begin{align*}
E_3^{sp/2} - \l v\r^{sp} - \l v_1\r^{sp} - \l v_2\r^{sp} 
& \leq (\l v\r^s + \l v_1\r^s + \l v_2\r^s)^p - \l v\r^{sp} - \l v_1\r^{sp} - \l v_2\r^{sp}\\
& :=\sum_{\substack{ 0\le k, k_1, k_2 < p \\ k+k_1+k_2 = p}} \binom{p}{k, k_1, k_2} \l v\r^{sk} \l v_1\r^{s k_1} \l v_2\r^{s k_2}.
\end{align*}
Therefore, \eqref{ternary potential upper bound} yields, with $C_{\gamma_3} = 2^{\gamma_3}\max\{1, 3^{\gamma_3-1}\}$,
\begin{align}
B_1
&\le C_{\gamma_3}   \sum_{\substack{ 0\le k, k_1, k_2 < p \\ k+k_1+k_2 = p}} \binom{p}{k, k_1, k_2}  
	\left( m_{sk +\gamma_3} \, m_{sk_1} \, m_{sk_2} 
		+ m_{sk} \, m_{sk_1+\gamma_3} \, m_{sk_2} 
		+ m_{sk} \, m_{sk_1}\, m_{sk_2+\gamma_3} \right)  \nonumber \\
& =  3C_{\gamma_3}  \sum_{\substack{ 0\le k, k_1, k_2 < p \\ k+k_1+k_2 = p}} \binom{p}{k, k_1, k_2}  
	\, m_{sk +\gamma_3} \, m_{sk_1} \, m_{sk_2}.
\label{bound on B_1}
\end{align}

\bigskip 
\textbf{Lower bound on term $B_2$:}
Estimate \eqref{ternary potential lower bound} implies
\begin{align}
B_2 
& = \int_{\mathbb{R}^{3d}}|\bm{\tild{u}}|^{\gamma_3} ff_1f_2 \l v \r^{sp}\,dv\,dv_{1,2} 
\geq \left(\frac{2}{3}\right)^{\gamma_3/2}  \, m_{sp+\gamma_3} \, m_0^2 \, - \, 2\, m_{sp}\, m_{\gamma_3} \, m_0 . 
\label{bound on B_2}
\end{align}

\bigskip

Plugging \eqref{bound on A_1}-\eqref{bound on B_2} into \eqref{first moments ineq},  we conclude that  for $sp> 2$, we have
\begin{align*}
\int_{\mathbb{R}^d}&Q[f]\l v\r^{sp}\,dv
\leq  2C_{\gamma_2} \alpha_{sp/2} S^p_{2,s,\gamma_2} \, - \, 2(\|b_2\| -\alpha_{sp/2}) \left(\frac{1}{2^{\gamma_2/2}}m_{sp+\gamma_2} \, m_0 -  m_{sp} \, m_{\gamma_2}\right) \nonumber \\
	&\quad +\, 3\,C_{\gamma_3} \lambda_{sp/2} S^p_{3,s,\gamma_3}
	\, - \, 3 \,(\|b_3\|-\lambda_{sp/2} )
	\left(  \left(\frac{2}{3}\right)^{\gamma_3/2}  m_{sp+\gamma_3} \, m_0^2 \, - \, 2 m_{sp} \, m_{\gamma_3}\, m_0\right).
\end{align*} 
Monotonicity of moments and the fact that $\gamma_2,\gamma_3\leq 2$ imply \eqref{moments ineq}.
\end{proof}

\begin{lemma}\label{EI}
For $E_s^n(t,z)$, $I_{s,\gamma_2}^n(t,z) $ and $I_{s,\gamma_3}^n(t,z)$ defined as in \eqref{E}-\eqref{I}, we have:
\begin{align}
\sum_{p=p_0}^n \frac{z^p}{p!} S^p_{2,s,\gamma_2}(t) & \le  I^n_{s,\gamma_2} (t,z) \, E^n_s (t,z), \label{IE}\\
\sum_{p=p_0}^n \frac{z^p}{p!} S^p_{3,s,\gamma_3}(t) & \le I_{s, \gamma_3}^n (t,z) \, \left(E_s^n(t,z) \right)^2. \label{JEE}
\end{align}
\end{lemma}

\begin{proof}
Inequality \eqref{IE} is proven by exchanging the order of summations and observing that the inner sum can be bounded by $E^n_s$. Namely, recalling \eqref{S_2}, we have
\begin{align*}
& \sum_{p=p_0}^n  \frac{z^p}{p!}  \sum_{\substack{0<k,k_1<p\\k+k_1=p}}\binom{p}{k,k_1}  m_{sk+\gamma_2} \, m_{sk_1} 
  =  \sum_{k=1}^{n-1}  \sum_{p=\max\{p_0, k+1\}}^n  
	\frac{z^p}{p!} \binom{p}{k, p-k} m_{sk+\gamma_2} \, m_{s(p-k)} \\
& \quad =  \sum_{k=1}^{n-1}  \frac{z^k}{k!}\,  m_{sk+\gamma_2} 
	 \sum_{p=\max\{p_0, k+1\}}^n \, \frac{z^{p-k}}{(p-k)!}  \, m_{s(p-k)} 
 \le I^n_{s,\gamma_2} (t,z) \, E^n_s (t,z).
\end{align*}

To prove inequality \eqref{JEE}, we begin by rewriting the summation in $k, k_1, k_2$ as two sums - one in $k$ and and the other one in $k_1$ - and then we exchange the order of summation in $p$ and $k$:
\begin{align*}
&\sum_{p=p_0}^n \frac{z^p}{p!}   \sum_{\substack{ 0\le k, k_1, k_2 < p \\ k+k_1+k_2 = p}} 
	\binom{p}{k, k_1, k_2}  \, m_{sk +\gamma_3} \, m_{sk_1} \, m_{sk_2} \\
&\qquad =  \sum_{p=p_0}^n \sum_{k=0}^{p-1} \sum_{\substack{k_1 = 0 \\ 0< k+k_1 \le p}}^{p-1}  
	\frac{z^p}{p!} \binom{p}{k, k_1, p-k-k_1}  \, m_{sk +\gamma_3} \, m_{sk_1} \, m_{s(p-k-k_1)} \\
&\qquad =  \sum_{p=p_0}^n \sum_{k=0}^{p-1} \sum_{\substack{k_1 = 0 \\ 0< k+k_1}}^{\min\{p-1, p-k\}}  
	\frac{z^p}{p!} \binom{p}{k, k_1, p-k-k_1}  \, m_{sk +\gamma_3} \, m_{sk_1} \, m_{s(p-k-k_1)} \\
& \qquad = \sum_{k=0}^{n-1}   \sum_{p=\max\{p_0, k+1 \}}^n  \sum_{\substack{k_1 = 0 \\ 0< k+k_1 }}^{\min\{p-1, p-k\}}  
	\frac{z^p}{p!} \binom{p}{k, k_1, p-k-k_1}  \, m_{sk +\gamma_3} \, m_{sk_1} \, m_{s(p-k-k_1)}.
\end{align*}

In order to exchange the second two sums, we first separate the term $k=0$ from the rest of the sum, and then exchange summations in $k_1$ and $p$:
\begin{align*}
 \sum_{k=0}^{n-1}   \sum_{p=\max\{p_0, k+1 \}}^n  \sum_{\substack{k_1 = 0 \\ 0< k+k_1 }}^{\min\{p-1, p-k\}}  
& =   \sum_{p=\max\{p_0, 1 \}}^n  \sum_{k_1 = 1 }^{p-1}  
	\,+ \,\, \sum_{k=1}^{n-1}   \sum_{p=\max\{p_0, k+1 \}}^n  
	\sum_{k_1 = 0 }^{ p-k} \\
& =   \sum_{k_1=1}^{n-1}  \sum_{p=\max\{ p_0, 1, k_1 +1  \}}^{n}  
	\, + \,\,   \sum_{k=1}^{n-1}   \sum_{k_1=0}^{n-k}  \sum_{p=\max\{ p_0, k+1, k_1 +k  \}}^{n}  \\
& \le   \sum_{k_1=0}^{n}  \sum_{p=\max\{ p_0, 1, k_1  \}}^{n}  
	\, + \,\,   \sum_{k=1}^{n}   \sum_{k_1=0}^{n-k}  \sum_{p=\max\{ p_0, k+1, k_1 +k  \}}^{n}  \\
& =  \sum_{k=0}^{n}   \sum_{k_1=0}^{n-k}  \sum_{p=\max\{ p_0, k+1, k_1 +k  \}}^{n}.
\end{align*}

Therefore,
\begin{align*}
&\sum_{p=p_0}^n \frac{z^p}{p!}   \sum_{\substack{ 0\le k, k_1, k_2 < p \\ k+k_1+k_2 = p}} 
	\binom{p}{k, k_1, k_2}  \, m_{sk +\gamma_3} \, m_{sk_1} \, m_{sk_2} \\
& \qquad \le  \sum_{k=0}^{n}   \sum_{k_1=0}^{n-k}  \sum_{p=\max\{ p_0, k+1, k_1 +k  \}}^{n}  
	\frac{z^p}{p!} \binom{p}{k, k_1, p-k-k_1}  \, m_{sk +\gamma_3} \, m_{sk_1} \, m_{s(p-k-k_1)}\\
& \qquad =   \sum_{k=0}^{n}  \,  \frac{z^k}{k!} \, m_{sk +\gamma_3}
	\sum_{k_1=0}^{n-k}   \,  \frac{z^{k_1}}{k_1!} \, m_{sk _1} 
	\sum_{p=\max\{ p_0, k+1, k_1 +k  \}}^{n}   \,  \frac{z^{p-k-k_1}}{(p-k-k_1)!} \, m_{s(p-k-k_1)} \\
& \qquad \le I_{s,\gamma_3}^n \, E_s^n \, E_s^n.
\end{align*}

\end{proof}

We are now ready to prove one of our main results, Theorem \ref{exponential moments theorem} on propagation and generation of exponential moments.
\begin{proof}[Proof of Theorem \ref{exponential moments theorem}]
Without loss of generality, we assume that $m_0>0$, otherwise by the conservation of mass the only solution is zero, so the claim trivially holds. We first prove propagation of exponential moments since it will be used in the proof of the generation of moments. In addition to the notation introduced in \eqref{E} and \eqref{I},
let us also introduce the following truncated partial sum notation:
\begin{align*}
P^n_{s,p_0}(t,z) := \sum_{p=p_0}^n  m_{sp}\frac{z^p }{p!}.
\end{align*}

(b) {\it (Proof of exponential moments' propagation)}
 Let $\gamma_2, \gamma_3$ be as in \eqref{cross-section 2} and \eqref{cross-section 3}, and suppose $s \in (0,2]$. For a positive constant $a<a_0$ that will be fixed later, and fixed $n\in\mathbb{N}$, we define $T_n>0$ as
\begin{align*}
T_n : = \sup \{t\in(0,T] : E^n_s(t,a) < 6\, C_0\},
\end{align*}
where $C_0$ is the constant in \eqref{exp id}. The goal is to show that $T_n=T$ and then let $n\to +\infty$. Let us fix $n\in\mathbb{N}$.
Since $E^n_s(0,a) \le \int_{\R^d} f(0,t) \, e^{a \l v \r^s} \, dv \le  \int_{\R^d} f(0,t) \, e^{a_0 \l v \r^s} \, dv < C_0 $, we have that $T_n>0$ by the continuity of $E_s^n$. 

For an integer $p_0 > 2/s$, to be chosen later,  the moment differential inequality \eqref{moments ineq} yields
\begin{align}
  \sum_{p=p_0}^n  m'_{sp}\frac{a^p }{p!}
& \le  \, \sum_{p=p_0}^n  \frac{a^p }{p!} 
	\left(  - K_{1,sp}\, m_{sp+\gamma_2} \,- \, K_{2,sp} \, m_{sp+\gamma_3} + K_{3,sp} m_{sp}
	+  2  C_{\gamma_2}  \alpha_{sp/2}  S^p_{2,s,\gamma_2} \, + \, 3  C_{\gamma_3}  \lambda_{sp/2} S^p_{3,s,\gamma_3} \right) \nonumber  \\
& \le - \K I^n_{s,\gamma_2}(t,a) + \K  \sum_{p=0}^{p_0 -1} m_{sp+\gamma_2} \frac{a^p }{p!}
	- \2 I^n_{s,\gamma_3}(t,a) + \2  \sum_{p=0}^{p_0 -1} m_{sp+\gamma_3} \frac{a^p }{p!} \nonumber  \\
& \qquad +\3 E^n_s(t,a)
	+ 2  C_{\gamma_2} \sum_{p=p_0}^n \alpha_{sp/2} S^p_{2,s,\gamma_2} \frac{a^p }{p!} 
	+ 3 C_{\gamma_3}  \sum_{p=p_0}^n \lambda_{sp/2} S^p_{3,s,\gamma_3} \frac{a^p }{p!} , \label{prop 1}
\end{align}
where, we used the fact that $-K_{1,sp}$ and $-K_{2,sp}$ in  \eqref{Ks} are decreasing in $sp$. The positive constants $\K, \2, \3$ above are defined by
\begin{align}
\K &=   2^{1- \frac{\gamma_2}{2}}m_0 \,(\|b_2\| -\alpha_{sp_0/2}), \nonumber \\
\2 &= 3 \, \left(\frac{2}{3}\right)^{\gamma_3/2} \, m_0^2 \, (\|b_3\| -\lambda_{sp_0/2}), \label {tild Ks}\\
\3 &= 2 m_2 \|b_2\|  \, + \, 6 m_0 \, m_2 \|b_3\|. \nonumber
\end{align}

We note that for $a<1$ (this will be one of conditions on $a$), the propagation of polynomial moments (Theorem \ref{polynomial moments theorem}) implies that
\begin{align}
\K  \sum_{p=0}^{p_0 -1} m_{sp+\gamma_2} \frac{a^p }{p!}
	+  \2  \sum_{p=0}^{p_0 -1} m_{sp+\gamma_3} \frac{a^p }{p!}
& \le \K  \sum_{p=0}^{p_0 -1} m_{sp+\gamma_2}
	+  \2 \sum_{p=0}^{p_0 -1} m_{sp+\gamma_3} 
\le \4, \label{below p_0}
\end{align}
where $\4$ depends on $s, p_0, \gamma_2, \gamma_3, m_0, m_2,\|b_2\|, \|b_3\|$.

Since $\gamma_{sp/2}$ and $\lambda_{sp/2}$ are decreasing in $p$, we can apply Lemma \ref{EI} to obtain
\begin{align}
  2  C_{\gamma_2}  \sum_{p=p_0}^n & \alpha_{sp/2} S^p_{2,s,\gamma_2} \frac{a^p }{p!} 
	+  3  C_{\gamma_3}   \sum_{p=p_0}^n \lambda_{sp/2} S^p_{3,s,\gamma_3} \frac{a^p }{p!} \nonumber \\
& \le  2  C_{\gamma_2}  \alpha_{sp_0/2} \, I^n_{s,\gamma_2}(t,a) \, E^n_s(t,a) 
	\, + \,  3  C_{\gamma_3}  \lambda_{sp_0/2} \, I^n_{s,\gamma_3}(t,a) \,  \left(E^n_s(t,a)\right)^2. \label{sum S}
\end{align}

Combining \eqref{prop 1} - \eqref{sum S}, we obtain
\begin{align*}
\sum_{p=p_0}^n  m'_{sp}\frac{a^p }{p!}
& \le  - \K I^n_{s,\gamma_2}(t,a) \, - \, \2 I^n_{s,\gamma_3}(t,a) 
	\,+ \, \3 E^n_s(t,a) \,+\,  \4 \nonumber \\
& \qquad \,+ \,   2  C_{\gamma_2}  \alpha_{sp_0/2} \, I^n_{\gamma_2}s(t,a) \, E^n_s(t,a) 
	\, + \,  3  C_{\gamma_3}  \lambda_{sp_0/2} \, I^n_{s,\gamma_3}(t,a) \, ( E^n_s(t,a))^2.  \nonumber \\
\end{align*}

By further regrouping the terms, we have
\begin{align}
\sum_{p=p_0}^n  m'_{sp}\frac{a^p }{p!}
& \le \,	 I^n_{s,\gamma_2}(t,a) \, \left(  2  C_{\gamma_2}  \alpha_{sp_0/2}\, E^n_s(t,a)  \,- \, \K \right)
		\,+ \, \3 E^n_s(t,a) \,+ \, \4  \nonumber \\
& \quad 	\, +\, I^n_{s,\gamma_3}(t,a) \,  \left(   3  C_{\gamma_3}  \lambda_{sp_0/2} \, (E^n_s(t,a))^2   \,- \, \2 \right).
		\label{prop 2}
\end{align} 

Now we choose $p_0$ large enough so that 
\begin{align}
12 \, C_{\gamma_2} \alpha_{sp_0/2} C_0 \le \K/2
\qquad \mbox{and} \qquad
108 \, C_{\gamma_3} \lambda_{sp_0/2} (C_0)^2 \le \2/2. \label {p_0}
\end{align}

For such a choice of $p_0$ we then have
\begin{align}
\sum_{p=p_0}^n  m'_{sp}\frac{a^p }{p!}
& \le  \, - \frac{ \K}{2} \,   I^n_{s,\gamma_2}(t,a)  - \frac{ \2}{2} \, I^n_{s,\gamma_3}(t,a) \,+ \, 6 \3 C_0 \, + \, \4  \nonumber\\
& \le  \, - \frac{ \K}{2} \,   I^n_{s,\gamma_2}(t,a)  - \frac{ \2}{2} \, I^n_{s,\gamma_3}(t,a) \,+ \, \5, \label{prop IJ}
\end{align}
where $\5 = 6 \3  C_0 \, + \, \4$ and so it depends on $s, p_0, \gamma_2, \gamma_3, m_0, m_2, \|b_2\|, \|b_3\|$.

Next we need a lower bound on $ I^n_{s,\gamma_2}(t,a)$ and $I^n_{s,\gamma_3}(t,a)$ in terms of $E^n_s(t,a)$.
\begin{align*}
I^n_{s,\gamma_2}(t,a) 
& \ge \frac{1}{a^{\gamma_2/2}} \sum_{p=0}^n  \int_{\l v\r \ge \frac{1}{\sqrt{a}}} f(t,v)\, \l v \r^{sp} \, \frac{a^p}{p!} \, dv\\
& = \frac{1}{a^{\gamma_2/2}}  \left(   \sum_{p=0}^n  \int_{\R^d} f(t,v)\, \l v \r^{sp} \frac{a^p}{p!} \, dv
	\, - \,  \sum_{p=0}^n  \int_{\l v \r < \frac{1}{\sqrt{a}}} f(t,v)\, \l v \r^{sp} \frac{a^p}{p!} \, dv	\right) \\
& \ge  \frac{1}{a^{\gamma_2/2}}  \left( E^n_s(t,a)  
	\, - \,  \sum_{p=0}^n  \int_{\R^d} f(t,v)\,  \frac{a^{p (1-\frac{s}{2}) }}{p!} \, dv \right) \\
& \ge  \frac{1}{a^{\gamma_2/2}}  \left( E^n_s(t,a)  
	\, - \, m_0 \sum_{p=0}^\infty \frac{a^{p (1-\frac{s}{2}) }}{p!}  \right)\\
& =    \frac{1}{a^{\gamma_2/2}}  \left( E^n_s(t,a)  
	\, - \, m_0 e^{a^{1-\frac{s}{2}}} \right).
\end{align*}

Similarly, we have 
\begin{align*}
I^n_{s,\gamma_3}(t,a)  \ge    \frac{1}{a^{\gamma_3/2}}  \left( E^n_s(t,a)  
	\, - \, m_0 \,e^{a^{1-\frac{s}{2}}} \right).
\end{align*}

Therefore, plugging the lower bounds for $I^n_{s,\gamma_2}$ and $I^n_{s,\gamma_3}(t,a)$ into \eqref{prop IJ} yields
\begin{align}
\sum_{p=p_0}^n  m'_{sp}\frac{a^p }{p!}
& \le \, - \left( \frac{\K}{2 a^{\gamma_2/2}} + \frac{\2}{2a^{\gamma_3/2}}\right) E^n_s(t,a) 
 	  \,+\, \frac{m_0 \, \K \, e^{a^{1-\frac{s}{2}}}}{2 a^{\gamma_2/2}} 
	\,+\, \frac{m_0 \, \2 \, e^{a^{1-\frac{s}{2}}}}{2 a^{\gamma_3/2}} 
	\,+\, \5.  \label{prop 3}
\end{align}

Since clearly,  $E^n_s(t,a) \ge P^n_{s,p_0}(t,a)$, we have
\begin{align}
\frac{d}{dt} P^n_{p_0} (t,a)
& \le \, - \left( \frac{\K}{2 a^{\gamma_2/2}} + \frac{\2}{2a^{\gamma_3/2}}\right) P^n_{p_0}(t,a) 
 	  \,+\, \frac{m_0 \, \K \, e^{a^{1-\frac{s}{2}}}}{2 a^{\gamma_2/2}} 
	\,+\, \frac{m_0 \, \2 \, e^{a^{1-\frac{s}{2}}}}{2 a^{\gamma_3/2}} 
	\,+\, \5.  \label{prop 3a}
\end{align}

Therefore,
\begin{align}
P^n_{s,p_0} (t,a)
&\le P^n_{s,p_0} (0,a) \, + \, m_0 \, e^{a^{1-\frac{s}{2}}} \,+\, \frac{\5}{ \frac{\K}{2 a^{\gamma_2/2}} 
	\,+\, \frac{\2}{2a^{\gamma_3/2}}}  \nonumber \\
& \le  C_0 \, + \, m_0 \, e^{a^{1-\frac{s}{2}}} 
	\, + \, \frac{\5}{ \frac{\K}{2 a^{\gamma/2}}} 
 =   C_0 \, + \, m_0 \, e^{a^{1-\frac{s}{2}}}  \,+\, a^{\gamma/2} \6,
\label{P est}
\end{align}
where $\6 = 2 \frac{\5}{\K}$ depends on $s, p_0, \gamma_2, \gamma_3, m_0, m_2, \|b_2\|, \|b_3\|$.

In order to have an estimate on $E^n_s$, we also need an estimate on $\displaystyle\sum_{p=0}^{p_0-1} m_{sp} \frac{a^p}{p!}$. Recalling the propagation of polynomial moments result (Theorem \ref{polynomial moments theorem}) we have
$
m_{sp} \le C_{s,p,m_0, m_2}
$
and so
\begin{align}
\sum_{p=0}^{p_0-1} m_{sp} \frac{a^p}{p!} 
\le m_0 + a \sum_{p=1}^{p_0-1} m_{sp} \frac{a^{p-1}}{p!} 
\le C_0 + a \, C_{s,p_0,m_0, m_2}.
\label{first terms est}
\end{align}

Combining \eqref{P est} with \eqref{first terms est}, we obtain
\begin{align*}
E^n_s(t,a) \le  2C_0 \,+\, a \, C_{s,p_0,m_0, m_2}  \, + \, m_0 \, e^{a^{1-\frac{s}{2}}}  \,+\, a^{\gamma/2} \6,
\end{align*}
We can choose $a < \min\{a_0,1\}$ small enough so that 
\begin{align}
 a \, C_{s,p_0,m_0, m_2}  \, + \, m_0 \, e^{a^{1-\frac{s}{2}}}  \,+\, a^{\gamma/2} \6 \label{a}
< 4 C_0,
\end{align}
which, in turn, implies that for such $a$ we have
 \begin{align*}
E^n_s(t,a) < 6 C_0.
\end{align*}

In conclusion, if $p_0$ is chosen according to \eqref{p_0}, and if $a$ is small enough that it satisfies \eqref{a}, we have that the strict inequality $E^n_s(t,a) < 6 C_0$ holds on the closed interval $[0, T_n]$. The continuity of  $E^n_s(t,a)$ then implies that $E^n_s(t,a) < 6 C_0$ holds on a larger time interval which would contradict the maximality of $T_n$ unless $T_n = T$. Thus, we conclude $T_n = T$ for all $n \in \N$. Therefore, in fact, we have
\begin{align*}
E^n_s(t,a) < 6C_0, \qquad \mbox{ for all } \, t \in [0,T], \mbox{ for all } \, n\in \N.
\end{align*}
Letting $n \rightarrow +\infty$, we conclude 
\begin{align*}
\int_{\R^d} f(t,v) \, e^{\alpha \l v \r^s} \, dv \, \le \, 6C_0, \qquad \mbox{ for all }\, t\in[0,T].
\end{align*}

(a) {\it (Proof of exponential moments' generation)} Let us write 
$z(t)=at$,
where $a<1$ is a positive constant that will be fixed later. For this  constant $a$ and for a fixed $n \in \N$,  we define $T_n>0$ as
\begin{align}\label{generation time step}
T_n : = \min\left\{ 1, \, \, \sup \{t \in[0, T] : E^n_\gamma(\tau,a\tau) < 4\,m_0  \,\,\, \mbox{for all} \,\, \tau \in [0,t] \} \right\}.
\end{align}
Indeed, $T_n>0$ for every $n$ because  $E^n_\gamma(0,0) = m_0 < 4\,m_0$.  Our goal is to prove that, in fact, $T_n =\min\{1,T\}$ for all $n$, at which point one can restart the argument at $t=\min\{1,T\}$ and use the propagation result to conclude finiteness of the exponential moment for all times. More details will be provided below. All equations below that depend on time are valid for $t \in (0,T_n]$ unless noted otherwise.

For $p_0>2/\gamma$, we have 
\begin{align*}
\frac{\,d}{\,dt}P^n_{\gamma, p_0}(t,z(t))
	&=\sum_{p=p_0}^n m'_{\gamma p}(t)\frac{z^p(t)}{p!}
		+ a\sum_{p=p_0 -1}^{n-1} m_{\gamma p + \gamma}(t)\frac{z^{p}(t)}{p!}.\nonumber
\end{align*}
The last term can be bounded by $a I_{\gamma,\gamma}^n(t,z(t))$. Therefore,  applying  \eqref{moments ineq} to the first term yields
\begin{align}
	\frac{\,d}{\,dt} &P_{\gamma, p_0}^n(t,z(t))
	\leq  a I_{\gamma,\gamma}^n(t,z(t)) 
		-\bar{K}_1\sum_{p=p_0}^nm_{\gamma p+\gamma_2}(t)\frac{z^p(t)}{p!}
		-\bar{K}_2\sum_{p=p_0}^nm_{\gamma p+\gamma_3}(t)\frac{z^p(t)}{p!}\nonumber\\
	& +\bar{K}_3\sum_{p=p_0}^nm_{\gamma p}(t)\frac{z^{p}(t)}{p!}
		+2 C_{\gamma_2} \sum_{p=p_0}^n\alpha_{\gamma p/2}S_{2,\gamma,\gamma_2}^p\frac{z^p(t)}{p!}
		+3 C_{\gamma_3} \sum_{p=p_0}^n\lambda_{\gamma p/2}S_{3,\gamma,\gamma_3}^p\frac{z^p(t)}{p!}\nonumber\\
	&=:a I_{\gamma,\gamma}^n(t,z(t)) -\bar{K}_1S_1-\bar{K}_2S_2+\bar{K}_3S_3+S_4+S_5,\label{exponential generation first ineq}
\end{align}
where, by  \eqref{Ks}, the coefficients are given by
\begin{align*}
\k_1 &=   2^{1-\gamma_2/2} m_0 \,(\|b_2\|-\alpha_{\gamma p_0/2}), \nonumber \\
\k_2 &= 3 \, \left(\frac{2}{3}\right)^{\gamma_3/2} \, m_0^2 \, (\|b_3\|-\lambda_{\gamma p_0/2}), \\
\k_3 &= 2 m_{\gamma_2}\|b_2\|   \, + \, 6 m_0 \, m_{\gamma_3}\|b_3\|. \nonumber
\end{align*}

\bigskip

\noindent{\textbf{Bound on $\boldsymbol{S_3}$}:} The sum $S_3$ can be simply bounded as follows
\begin{align}\label{bound S_3 gen}
S_3 :=  \sum_{p=p_0}^n m_{\gamma p}(t) \frac{z^p(t)}{p!} \le E^n_\gamma (t, z(t))=E^n_\gamma (t, a t)<4m_0,
\end{align}
where the last inequality holds by the definition of $T_n$ given in \eqref{generation time step}.
\bigskip

\noindent{\textbf{Bound on $\boldsymbol{-S_1}$:}} By the generation of polynomial moments estimate \eqref{polynomial generation estimate gamma_i>0}, we have
\begin{align}
	-S_1
	&= -I_{\gamma,\gamma_2}^n(t,z(t))
		+\sum_{p=0}^{p_0-1}m_{\gamma p+\gamma_2}(t)\frac{z^p(t)}{p!} \nonumber\\
	&\leq -I_{\gamma,\gamma_2}^n(t,z(t))
		+ C_{p_0}  \sum_{p=0}^{p_0-1} 
			t^{\frac{2-(\gamma p+\gamma_2)}{\gamma}}\frac{z^p(t)}{p!}\nonumber\\
	&= -I_{\gamma,\gamma_2}^n(t,a t)
		+C_{p_0} t^{\frac{2-\gamma_2}{\gamma}}
			\sum_{p=0}^{p_0-1}\frac{a^p}{p!}
			\nonumber \\
&\leq -I_{\gamma,\gamma_2}^n(t,a t)+C_{p_0} t^{\frac{2-\gamma}{\gamma}} e,\label{bound S_1 gen}
\end{align}
since $t \le 1$, $\gamma_2 \le \gamma$ and $a<1$.

\bigskip

\noindent{\textbf{Bound on $\boldsymbol{-S_2}$:}} Similarly we obtain 
\begin{equation}\label{bound S_2 gen}
-S_2\leq -I_{\gamma,\gamma_3}^n(t,a t)+C_{p_0} t^{\frac{2-\gamma}{\gamma}} e.
\end{equation}

\bigskip

\noindent {\bf Bound on $\boldsymbol{S_4}$:}  By \eqref{IE}, we have
\begin{align}\label{bound S_4 gen}
S_4 &\le  2 C_{\gamma_2}  \alpha_{\gamma p_0/2} \, I^n_{\gamma, \gamma_2} (t,z(t)) \, E^n_\gamma(t, z(t))=2 C_{\gamma_2} \alpha_{\gamma p_0/2} \, I^n_{\gamma, \gamma_2} (t,a t) \, E^n_\gamma(t, a t).
\end{align}

\noindent {\bf Bound on $\boldsymbol{S_5}$:}  Similarly, by \eqref{JEE}, we have
\begin{align}\label{bound S_5 gen}
S_5 & \le   3 C_{\gamma_3}  \lambda_{\gamma p_0/2} \, I^n_{\gamma, \gamma_3} (t,z(t)) \, \left(E^n_\gamma(t, z(t)) \right)^2
	= 3 C_{\gamma_3}  \lambda_{\gamma p_0/2} \, I^n_{\gamma, \gamma_3} (t,a t) \, \left(E^n_\gamma(t, a t) \right)^2.
\end{align}

Combining estimates \eqref{exponential generation first ineq} -\eqref{bound S_5 gen}, and using  $E_{\gamma}^n(t,at)<4m_0$ for $t\in [0,T_n]$, we obtain
\begin{align*}
\frac{d}{dt} P^n_{\gamma, p_0} (t, a t)  
	&\le  a\, I^n_{\gamma,\gamma}(t, a t) 
		- \k_1 I^n_{\gamma,\gamma_2}(t,a t)  
		-  \k_2  I^n_{\gamma,\gamma_3}(t,a t) + (\k_1+\k_2) C_{p_0} et^{\frac{2-\gamma}{\gamma}}+4m_0\bar{K}_3 \\
	&\quad \quad+  2 C_{\gamma_2} \alpha_{\gamma p_0/2} \, I^n_{\gamma, \gamma_2} (t,a t) \, E^n_\gamma(t, a t)
		+ 3C_{\gamma_3} \lambda_{\gamma p_0/2} \, I^n_{\gamma, \gamma_3} (t,a t) 
		 \left(E^n_\gamma(t,a t) \right)^2\\
	&\leq a\, I^n_{\gamma,\gamma}(t, a t)  + (\k_1+\k_2) C_{p_0} et^{\frac{2-\gamma}{\gamma}}+4m_0 \bar{K}_3\\
	&\quad \quad  -(\bar{K}_1-8m_0 C_{\gamma_2} \alpha_{\gamma p_0/2})\, I^n_{\gamma, \gamma_2} (t,a t) 
		-(\bar{K}_2 - 48m_0^2 C_{\gamma_3} \lambda_{\gamma p_0/2})\, I^n_{\gamma, \gamma_3} (t,a t).
\end{align*}

Since the sequences $p\mapsto\alpha_{\gamma p/2}$, $p\mapsto\lambda_{\gamma p/2}$ tend to zero, so for $p_0$ large enough, we have
\begin{align*}
\frac{d}{dt} P^n_{\gamma, p_0} (t, a t)  
	&\le  a\, I^n_{\gamma,\gamma}(t, a t) 
		+ (\k_1+\k_2)C_{p_0} et^{\frac{2-\gamma}{\gamma}} + 4m_0 \bar{K}_3
		-\frac{\bar{K}_1}{2}\, I^n_{\gamma,\gamma_2}(t, a t)
		-\frac{\bar{K}_2}{2}\, I^n_{\gamma,\gamma_3}(t, a t).
\end{align*}
But since $\gamma=\max\{\gamma_2,\gamma_3\}$, one of the terms $I^n_{\gamma,\gamma_2}$, $I^n_{\gamma,\gamma_3}$ coincides with $I^n_{\gamma,\gamma}$, and so we have that 
$$-\frac{\bar{K}_1}{2}\, I^n_{\gamma,\gamma_2}(t, at)-\frac{\bar{K}_2}{2}\, I^n_{\gamma,\gamma_3}(t, at)\leq-\frac{1}{2}\min\{\bar{K}_1,\bar{K}_2\}\, I^n_{\gamma,\gamma}(t, at),$$
and thus
\begin{align*}
\frac{d}{dt} P^n_{\gamma, p_0} (t, a t)  
	&\le  \left(a - \frac{1}{2}\min\{\bar{K}_1,\bar{K}_2\}\right)\, I^n_{\gamma,\gamma}(t, a t) 
	 + (\k_1+\k_2) C_{p_0} et^{\frac{2-\gamma}{\gamma}}+4m_0 \bar{K}_3.
\end{align*}
For $a$ small enough, we have
\begin{align*}
\frac{d}{dt} P^n_{\gamma, p_0}(t, a t)  
	&\le -\frac{1}{4}\min\{\bar{K}_1,\bar{K}_2\}\, I^n_{\gamma,\gamma}(t, a t)  
		+ (\k_1+\k_2) C_{p_0} et^{\frac{2-\gamma}{\gamma}}+4m_0 \bar{K}_3.
\end{align*}
Now notice that
\begin{align*}
I_{\gamma,\gamma}^n(t,z)
	&=\sum_{p=0}^{n}m_{\gamma(p+1)}\frac{z^p}{p!}
	=\frac{1}{z}  \sum_{p=1}^{n+1}p\,m_{p\gamma}(t)\frac{z^p}{p!}
	\geq\frac{1}{z}\sum_{p=1}^{n}m_{p\gamma}(t)\frac{z^p}{p!}
	=\frac{1}{z}(E_{\gamma}^n(t,z)-m_0).
\end{align*}
Using this inequality with $z=a t$, we have
\begin{align*}
\frac{d}{dt} P^n_{\gamma, p_0} (t, a t)
	&\leq -\frac{\min\{\bar{K}_1,\bar{K}_2\}}{4a t} 
		\left( E_{\gamma}^n(t,a t)-m_0 
			- \frac{4a}{\min\{\bar{K}_1,\bar{K}_2\}}\left[ \left(\bar{K}_1+\bar{K}_2\right)C_{p_0}et^{\frac{2}{\gamma}}+4m_0 \bar{K}_3 t \right]\right)\\
		&\leq -\frac{\min\{\bar{K}_1,\bar{K}_2\}}{4a t} 
		\left( E_{\gamma}^n(t,a t)-m_0 
			- \frac{4a}{\min\{\bar{K}_1,\bar{K}_2\}}\left[ \left(\bar{K}_1+\bar{K}_2\right)C_{p_0}e+4m_0 \bar{K}_3\right]\right),	
			\end{align*}
			where to obtain the last inequality we used the fact that $t\leq 1$.
Using that $E_{\gamma}^n(t,a t) \ge P^n_{\gamma,p_0}(t,at),$ for $a$ small enough, we have
$$
\frac{d}{dt} P^n_{\gamma, p_0} (t, a t)
	\leq -\frac{\min\{\bar{K}_1,\bar{K}_2\}}{4a t} \left( P_{\gamma,p_0}^n(t,a t)-2m_0\right).
$$
Therefore, whenever $P^n_{\gamma,p_0}(t,at)  > 2m_0$, $P^n_{\gamma, p_0} (t, at) $ decreases in time. Since $P^n_{\gamma, p_0}(0,0) = 0 < 2m_0$, we conclude that 
\begin{align}
P^n_{\gamma, p_0} (t, at) \le 2m_0 \label{p_0 to n}
\end{align}
 holds uniformly on the closed interval $[0,T_n]$.

On the other hand, the first $p_0$ terms can be bounded as follows:
\begin{align}
	\sum_{p = 0}^{p_0 -1} m_{\gamma p}(t)  \frac{(at)^p}{p!} 
		= m_0  + \sum_{p = 1}^{p_0 -1} m_{\gamma p}(t)  \frac{(at)^{p}}{p!} 
		\le m_0 + ae  C^*_{p_0}, \label{first p_0}
\end{align}
for $t \in [0,T_n]$.  Namely, from Theorem \ref{polynomial moments theorem} and the fact that $t\le 1$,  we have
\begin{align*}
m_{\gamma p}(t) &\le C_{\gamma p} \, t^{\frac{2-\gamma p}{\gamma}}=C_{\gamma p} \, t^{- p}.
\end{align*}
If we define $C^*_{p_0} := \max_{p \in \{0, 1, ..., p_0-1 \}} C_{\gamma p},$
we have
\begin{align*}
m_{\gamma p}(t) 
	 \, \le \, C^*_{p_0} t^{-p}
	 \qquad \mbox{for all} \,\, p \in \{0, 1, ..., p_0-1 \}.
\end{align*}
Therefore, since $0<t\leq T_n\leq 1$,
\begin{align*}
\sum_{p=1}^{p_0-1} m_{\gamma p}(t) \frac{(a t)^p}{p!}
	\le C^*_{p_0}  \sum_{p=1}^{p_0-1}  t^{- p}  \frac{a^p t^{ p}}{p!}
	\le a C^*_{p_0}\sum_{p=1}^{p_0-1}\frac{a^{p-1}}{(p-1)!}
	\leq a C_{p_0}^*e,
\end{align*}
since $a<1$.

Finally, combining \eqref{p_0 to n} with \eqref{first p_0} we obtain:
\begin{align*}
E^n_\gamma(t, at) \le 3m_0 + aeC^*_{p_0}, 
\end{align*}
therefore for $a$ small enough we have 
\begin{align*}
E^n_\gamma(t, at) < 4m_0, 
\qquad \mbox{for all} \,\, t \in [0,T_n].
\end{align*}
Now, the continuity of the partial sum implies that the strict inequality $E^n_\gamma(t, at) < 4m_0$ holds beyond $T_n$, which contradicts the maximality of $T_n$ unless $T_n = \min\{1,T\}.$ Therefore, $T_n=\min\{1,T\}$ for all $n$. This implies that
\begin{align*}
E^n_\gamma(t, at) < 4m_0, 
\quad \mbox{for all} \,\, t\in[0,\min\{1,T\}] \,\, \mbox{for all} \,\, n\in\N.
\end{align*}
Let $n\rightarrow \infty$ to conclude:
\begin{align*}
\int_{\R^d} f(t,v) e^{at \l v \r^\gamma} \, dv \le4m_0, 
\qquad \mbox{for all}\,\, t \in[0,\min\{1,T\}].
\end{align*}
This implies that at time $t=\min\{1,T\}$, the exponential moment of order $\gamma$ and rate $a$ is finite. The propagation of exponential moments result then implies that there exists $0<a_1 <a$ such that the exponential moment of the same order $\gamma$ and a rate $a_1$ remains bounded uniformly for all $t\ge \min\{1,T\}$. In conclusion
\begin{align*}
&\int_{\R^d} f(t,v) e^{a_1 t \l v \r^\gamma} \, dv < C,   \qquad \mbox{for all}\,\, t \in[0,\min\{1,T\}],\\
&\int_{\R^d} f(t,v) e^{a_1 \l v \r^\gamma} \, dv < C,   \qquad \mbox{for all}\,\, t \ge \min\{1,T\}.\\
\end{align*}
Therefore,
\begin{align*}
\int_{\R^d} f(t,v) e^{a_1 \min\{ 1, t \} \l v \r^\gamma} \, dv < C,   \qquad \mbox{for all}\,\, t \in (0,T].
\end{align*}

\end{proof}

\bigskip

\section{Global well-posedness of the binary-ternary Boltzmann equation} 
\label{sec - existence}

In this last section, we prove Theorem \ref{gwp theorem intro}, which establishes existence and uniqueness of global in time solutions for the binary-ternary Boltzmann equation \eqref{binary-ternary equation} for initial data in $L^1_{2+\varepsilon}$, where $\varepsilon>0$.  Without loss of generality, we assume that $m_0[f_0]>0$, otherwise by the conservation of mass the unique solution is trivially zero.

As mentioned in the introduction, motivated by analogous results \cite{br05, alga22,gapa20,algatr16} in the context of the Boltzmann equation, system of Boltzmann equations for gas mixtures, and the quantum Boltzmann equation,   we will rely on the general theory for ODEs in Banach spaces, namely Theorem \ref{thm - banach} stated in the appendix. The idea is to first construct a unique solution assuming the initial data have $2+2\gamma$ moments, where $\gamma=\max\{\gamma_2,\gamma_3\}>0$, see Proposition \ref{lemma on existence} below. Then, in order to prove Theorem \ref{gwp theorem intro}, we will relax the assumption on the initial data using generation and propagation of polynomial moments (Theorem \ref{polynomial moments theorem}). 

More specifically, given $f_0\in L^1_{2+2\gamma}$, we first prove  that all the conditions of Theorem \ref{thm - banach} are satisfied for the operator $\mathcal{Q}:=Q$, the Banach space $X:=L_{2}^1$ and the invariant subset $S:=\Omega[f_0]\subset L_2^1$, where $\Omega[f_0]$ is given by:
\begin{equation}\label{Omega set}
\Omega[f_0]=\bigg\{f\in L_{2}^1:\,f\geq 0,\,\, m_0[f]=m_0[f_0],\,\,m_2[f]=m_2[f_0],\,\, m_{2+2\gamma}[f]\leq \max\{A_{2+2\gamma},m_{2+2\gamma}[f_0]\}\bigg\},
\end{equation}
where $A_{2+2\gamma}$ is an appropriate constant  defined in \eqref{Ck*}.

\begin{proposition}\label{lemma on existence}Let $T>0$, $\gamma=\max\{\gamma_2,\gamma_3\}$, and $f_0\in L^1_{2+2\gamma}$ with $f_0\geq 0$. Then the binary-ternary Boltzmann equation \eqref{binary-ternary equation} has a   solution $f\geq 0$, in the sense of Definition \ref{def - strong solution}, with\footnote{writing $f\in C([0,T],\Omega[f_0])$, we mean that $f\in C([0,T],L_2^1)$, and that $f(t)\in\Omega[f_0]$ for all $t\in[0,T]$.} $f\in C([0,T],\Omega[f_0])\cap C^1((0,T),L_2^1)$. In particular, the solution $f$ conserves the mass, momentum and energy of the initial data $f_0$. Moreover $f\in C^1((0,T),L_k^1)$ for all $k> 2$.
\end{proposition}
\begin{proof}
Let us first note that  $\Omega[f_0]$ is clearly convex and bounded and $f_0\in\Omega[f_0]$. Moreover, by Fatou's Lemma, $\Omega[f_0]$ is also a closed subset of $L^1_2$.  Additionally, by the definition of $\Omega[f_0]$, we have $C([0,T],\Omega[f_0]) \subseteq L^1((0,T),L^1_{2+\gamma})$. Therefore, to prove existence of a solution, it remains to prove that the assumptions \textit{(1)-(3)} of Theorem \ref{thm - banach} are satisfied. For that, we  will strongly rely on the generalized description of the collisional operators $Q_2$ and $Q_3$, and their weak formulation, both of which are presented in the Appendix. In the following, we denote $m_0:=m_0[f_0]$, $m_2:=m_2[f_0]$. 

{\it Proof of condition \textit{(1)}:} We  show that $Q$ satisfies \eqref{holder condition} in the set $\widetilde{\Omega}[f_0]\supset\Omega[f_0]$ where
\begin{equation}\label{Omega tilde}
\widetilde{\Omega}[f_0]=\left\{f\in L_{2}^1:\,\,f\geq 0, \,\, m_{2+2\gamma}[f]\leq \max\{A_{2+2\gamma}, m_{2+2\gamma}[f_0]\}\right\},
\end{equation}
  and therefore in $\Omega[f_0]$ as well. Note that proving \eqref{holder condition} in $\widetilde{\Omega}[f_0]$ also proves that $Q:\widetilde{\Omega}[f_0]\to L_{2}^1$.
  
We first prove an estimate on $\| Q[f] - Q[g] \|_{L^1_2} $ for any two functions $f,g\in L^1_{2+\gamma}$ (see \eqref{pre Holder}), that in addition to being used to prove H\"older continuity condition \eqref{holder condition}, will also be needed to relax the  initial data condition at the end of this section. So, let $f,g\in L^1_{2+\gamma}$.
Due to bilinearity-trilinearity and symmetry of the operators (see \eqref{diferrence binary}, \eqref{diferrence ternary}), we have:
\begin{equation}\label{difference of operators}
\begin{aligned}
Q[f] - Q[g]  
& =  Q_2\left(f-g,f+g\right) + Q_3(f-g,f,f)+Q_3(f-g,f,g)+Q_3(f-g,g,g).
\end{aligned}
\end{equation}
Let $s_0$ be the sign of $Q_2(f-g, f+g) $, $s_1$ the sign of $Q_3(f-g, f,f)$, $s_2$ the sign of $Q_3(f- g,f,g)$ and $s_3$ the sign of $Q_3(f-g, g,g)$. Then, by the triangle inequality, we have
\begin{align*}
\| Q[f] - Q[g] \|_{L^1_2}  
& \le  \int_{\R^d}  Q_2(f-g, f+g) s_0    \langle v \rangle^2 \, dv 
	+   \int_{\R^d}  Q_3(f-g,f,f) s_1 \langle v \rangle^2 \, dv \\
&\quad +  \int_{\R^d} Q_3(f-g, f, g) s_2  \langle v \rangle^2 \, dv
	+  \int_{\R^d} Q_3(f-g, g, g) s_3  \langle v \rangle^2 \, dv.
\end{align*}
Since  $f,g\in  L^1_{2+\gamma}$, we use the weak formulations \eqref{binary weak form}, \eqref{ternary weak form}  with $s_i   \langle v \rangle^2$, $i=0,1,2,3$ as test functions. We also use the triangle inequality, bound all signs by one, and we use conservation of energy by the collisions to obtain
\begin{align*}
\| Q[f] - Q[g] \|_{L^1_2}  
& \le \int_{\mathbb{S}^{d-1}\times\mathbb{R}^{2d}} B_2(u,\omega)|f-g| (f_{1}+g_{1})  \left(  \langle v \rangle^2 +  \langle v _1\rangle^2 \right) \, d\omega dv dv_1\nonumber \\
&  \quad+\frac{1}{18}\sum_{\pi\in\ S_{0,1,2}} \int_{\mathbb{S}^{2d-1}\times\mathbb{R}^{3d}} B_3(\bm{u},\bm{\omega})  |f_{\pi_0}-g_{\pi_0}| \left(f_{\pi_1}f_{\pi_2} + f_{\pi_1} g_{\pi_2} + g_{\pi_1} g_{\pi_2}\right) \\
&\hspace{3cm}\times\left( \langle v \rangle^2 +  \langle v _1\rangle^2   + \langle v _2 \rangle^2 \right) \, d\bm{\omega} \,dv \,dv_{1,2}.
\end{align*}
Now, using the form of the cross-sections \eqref{cross-section 2}, \eqref{cross-section 3}, the cut-offs \eqref{binary cut-off}, \eqref{ternary cut-off}  , and  the potential bounds \eqref{binary potential upper bound},  \eqref{ternary potential upper bound},  as well as the symmetry with respect to the integration variables, we obtain
\begin{align}
\|& Q[f] - Q[g] \|_{L^1_2} \nonumber \\ 
& \le  C_{\gamma_2} \|b_2\| \int_{\mathbb{R}^{2d}} |f-g| (f_1+g_1) \Big( \l v \r^{2+\gamma_2}  +  \l v \r^{\gamma_2} \l v_1\r^2  +  \l v\r^2 \l v_1 \r^{\gamma_2} + \l v_1 \r^{2+\gamma_2} \Big) \, dv dv_1\nonumber \\
& \qquad + 
	\frac{C_{\gamma_3}}{3}\|b_3\|
 \int_{\mathbb{R}^{3d}}  |f-g|  \left(f_1f_2 + f_1 g_2 + g_1 g_2\right)
	\Big( \l v \r^{2+\gamma_3}  + \l v\r^2 \l v_1 \r^{\gamma_3} +  \l v\r^2 \l v_2 \r^{\gamma_3} \Big.\nonumber \\
& \quad \quad \Big.+ \! \l v \r^{\gamma_3} \l v_1 \r^2 \!+\! \l v_1 \r^{2+\gamma_3}  \!+\!  \l v_2 \r^{\gamma_3} \l v_1 \r^2 
\!+\!  \l v \r^{\gamma_3} \l v_2 \r^2   +  \l v_1 \r^{\gamma_3} \l v_2 \r^2 + \l v_2 \r^{2+\gamma_3}\!\Big) dv dv_{1,2}\nonumber\\
&\le 3   C_{\gamma_2} \|b_2\|  \|f-g\|_{L^1_{2+\gamma_2}}  (\|f\|_{L^1_2} +\|g\|_{L^1_2}  )
	+  C_{\gamma_2} \|b_2\|  \|f-g\|_{L^1_{0}} (\|f\|_{L^1_{2+\gamma_2}} +\|g\|_{L^1_{2+\gamma_2}}  )  \nonumber \\
 & \qquad + \frac{7 C_{\gamma_3}}{3}\|b_3\| \|f-g\|_{L^1_{2+\gamma_3}}
 			 \left(\|f\|_{L^1_{2}} +\|g\|_{L^1_2}\right)^2 \nonumber \\
& \qquad +  \frac{ 2C_{\gamma_3}}{3}\|b_3\| \|f-g\|_{L^1_{0}} 
		\left( \|f\|_{L^1_{0}}  + \|g\|_{L^1_{0}} \right)
		\left( \|f\|_{L^1_{2+\gamma_3}} +\|g\|_{L^1_{2+\gamma_3}} \right) \nonumber\\
&\le  C_1  \|f-g\|_{L^1_{2+\gamma}}  \left(\|f\|_{L^1_2} +\|g\|_{L^1_2} \right) \left(1+\|f\|_{L^1_2} +\|g\|_{L^1_2}\right)
\nonumber \\
	&\quad + C_2   \|f-g\|_{L^1_{0}} \left(1+\|f\|_{L^1_0} +\|g\|_{L^1_0}\right)
	\left( \|f\|_{L^1_{2+\gamma}} +\|g\|_{L^1_{2+\gamma}} \right),
\label{pre Holder}
\end{align}
where $C_1 = \max\{ 3   C_{\gamma_2} \|b_2\|,  \frac{7 C_{\gamma_3}}{3}\|b_3\|  \}$ and
 $C_2 = \max \{C_{\gamma_2} \|b_2\|,  \frac{ 2C_{\gamma_3}}{3}\|b_3\| \}.$ Let us note, again, that estimate \eqref{pre Holder} is valid for any $f,g\in L^1_{2+\gamma}$.

Now to complete the proof of the  H\"older condition \eqref{holder condition}, assume  $f,g\in\widetilde{\Omega}[f_0]$. Then we have $\|f\|_{L^1_{2+2\gamma}},\|g\|_{L^1_{2+2\gamma}}\leq \max\{A_{2+2\gamma},m_{2+2\gamma}[f_0]\}$. 
  Therefore, estimate \eqref{pre Holder} and the interpolation inequality \eqref{general interpolation} yield
$$ \| Q[f] - Q[g] \|_{L^1_2}\leq  C_H\|f-g\|_{L_2^1}^{1/2},$$
for some appropriate constant $C_H>0$, depending on $\max\{A_{2+2\gamma},m_{2+2\gamma}[f_0]\}$. 
Condition \textit{(1)} is proved.

{\it Proof of condition (2):}
We will now show that $Q$ satisfies \eqref{lipschitz condition} in $\tild{\Omega}[f_0]$, and therefore in $\Omega[f_0]$ as well. First notice that for $u,w\in L_2^1$ we have 
\begin{equation}\label{lip inequality}
[u,w]:=\lim_{h\to 0^-}\frac{\|w+hu\|_{L^1_2}-\|w\|_{L^1_2}}{h} = \int_{\mathbb{R}^d} u \sign(w) \l v \r^2 \, dv.
\end{equation} Indeed, for $h\neq 0$, 
triangle inequality implies
$$\frac{\big||w+hu|-|w|\big|}{|h|}\leq |u|\in L_2^1.$$
Moreover, we have
$$\lim_{h\to 0}\frac{|w+hu|-|w|}{h}=\lim_{h\to 0}\frac{2wu +hu^2}{|w+hu|+|w|}= u\sign(w). $$
Therefore, by the dominated convergence theorem, we take
\begin{align*}
[u,w]=\lim_{h\to 0}\frac{\|w+hu\|_{L_2^1}-\|w\|_{L_2^1}}{h}=\lim_{h\to 0}\int_{\mathbb{R}^d}\frac{|w+hu|-|w|}{h}\langle v\rangle^2\,dv=\int_{\mathbb{R}^d} u\sign(w)\langle v\rangle^2\,dv.
\end{align*}

We first prove an estimate on $[Q[f] - Q[g], \, f-g] $ for any two functions $f,g\in L^1_{2+\gamma}$ (see \eqref{estimate on bracket}), that in addition to being used to prove one-sided Lipschitz condition \eqref{lipschitz condition}, will also be needed to relax the  initial data condition at the end of this section. So, let $f,g\in L^1_{2+\gamma}$ and let us write  $s:=\sign(f-g)$. Since $f,g\in L^1_{2+\gamma}$, we have $Q[f]-Q[g]\in L_2^1$, due to \eqref{pre Holder} (or Lemmata \ref{lemma on regularity of collisional operator binary}-\ref{lemma on regularity of collisional operator ternary} in the Appendix). Thus  \eqref{lip inequality} and \eqref{difference of operators} yield

\begin{align*}
[Q[f] - Q[g], \, f-g]
& =  \int_{\R^d}   Q_2(f-g, f+g)   \, s \, \l v \r^2 \, dv  +  \int_{\R^d} Q_3(f-g, f, f) \, s \, \l v \r^2 \, dv \\
	& \quad +   \int_{\R^d} Q_3(f-g,f, g) \, s \, \l v \r^2 \, dv 
	+  \int_{\R^d} Q_3(f-g, g, g) \, s \, \l v \r^2 \, dv, \\
\end{align*}
Since $f,g\in L^1_{2+\gamma}$, using the weak formulations \eqref{binary weak form}, \eqref{ternary weak form} with $s\l v\r^2$ as a test function, we obtain
\begin{align*}
& [Q[f] - Q[g], \, f-g] \\
& = \frac{1}{2}\int_{\mathbb{S}^{d-1}\times\mathbb{R}^{2d}} B_2(u,\omega) \, (f-g)(f_1+g_1) \Big( s' \l v' \r^2  + s_1' \l v_1' \r^2  -s \l v \r^2  - s_1 \l v_1\r^2  \Big) \, d\omega  \,dv\,dv_1\\
&\quad +\frac{1}{36}\sum_{\pi\in S_{0,1,2}}\int_{\mathbb{S}^{2d-1}\times\mathbb{R}^{2d}} B_3(\bm{u},\bm{\omega})(f_{\pi_0}-g_{\pi_0})(f_{\pi_1}f_{\pi_2}+f_{\pi_1}g_{\pi_2}+g_{\pi_1}g_{\pi_2})\\
&\quad \times\left(s_{\pi_0}^*\l v_{\pi_0}^*\r^2+s_{\pi_1}^*\l v_{\pi_1}^*\r^2+s_{\pi_2}^*\l v_{\pi_2}^*\r^2-s_{\pi_0}\l v_{\pi_0}\r^2-s_{\pi_1}\l v_{\pi_1}\r^2-s_{\pi_2}\l v_{\pi_2}\r^2\right)\,d\bm{\omega}\,dv\,dv_{1,2}.
\end{align*}
Using the fact 
 $(f-g)s=|f-g|,\quad (f_{\pi_0}-g_{\pi_0})s_{\pi_0}=|f_{\pi_0}-g_{\pi_0}|$, 
 bounding the rest of the signs by one, and  using the conservation of energy by the collisions, we obtain

\begin{align*} 
&[Q[f] - Q[g], \, f-g] \nonumber \\
 & \le \frac{1}{2}\int_{\mathbb{S}^{d-1}\times\mathbb{R}^{2d}} B_2(u,\omega) \, |f-g||f_1+g_1| \Big(  \l v' \r^2  +  \l v_1' \r^2  - \l v \r^2  +  \l v_1\r^2  \Big) \, d\omega  \,dv\,dv_1 \nonumber\\
&\quad +\frac{1}{36}\sum_{\pi\in S_{0,1,2}}\int_{\mathbb{S}^{2d-1}\times\mathbb{R}^{3d}} B_3(\bm{u},\bm{\omega})|f_{\pi_0}-g_{\pi_0}||f_{\pi_1}f_{\pi_2}+f_{\pi_1}g_{\pi_2}+g_{\pi_1}g_{\pi_2}| \nonumber\\
&\quad\quad \times\left(\l v_{\pi_0}^*\r^2+\l v_{\pi_1}^*\r^2+\l v_{\pi_2}^*\r^2-\l v_{\pi_0}\r^2+\l v_{\pi_1}\r^2+\l v_{\pi_2}\r^2\right)\,d\bm{\omega}\,dv\,dv_{1,2} \nonumber\\
&=\int_{\mathbb{S}^{d-1}\times\mathbb{R}^{2d}} B_2(u,\omega) \, |f-g||f_1+g_1|\l v_1\r^2\, d\omega  \,dv\,dv_1 \nonumber\\
&\quad+\frac{1}{18}\sum_{\pi\in S_{0,1,2}}\int_{\mathbb{S}^{2d-1}\times\mathbb{R}^{3d}} B_3(\bm{u},\bm{\omega})|f_{\pi_0}-g_{\pi_0}||f_{\pi_1}f_{\pi_2}+f_{\pi_1}g_{\pi_2}+g_{\pi_1}g_{\pi_2}|\left(\l v_{\pi_1}\r^2+\l v_{\pi_2}\r^2\right)\,d\bm{\omega}\,dv\,dv_{1,2}
\end{align*}
Now, using the form of the cross-sections \eqref{cross-section 2}, \eqref{cross-section 3}, the cut-offs \eqref{binary cut-off}, \eqref{ternary cut-off}  , and  the potential bounds \eqref{binary potential upper bound},  \eqref{ternary potential upper bound},  as well as the symmetry with respect to the integration variables, we obtain
\begin{align} \label{estimate on bracket}
&[Q[f] - Q[g], \, f-g] \nonumber \\
& \le C_{\gamma_2} \|b_2\| \int_{\R^{2d}} |f-g| (f_1+g_1) 
		(\l v\r^{\gamma_2} \l v_1\r^2 + \l v_1 \r^{2+\gamma_2}) dv dv_1 \nonumber \\
	& \qquad + \frac{C_{\gamma_3}}{3} \|b_3\| 
		\int_{\R^{3d}} |f-g| (f_1 f_2 + f_1 g_2 + g_1 g_2) \nonumber \\
	& \qquad\qquad 	( \l v \r^{\gamma_3} \l v_1\r^2 + \l v \r^{\gamma_3} \l v_2\r^2
		+ \l v_1\r^{2+\gamma_3}  +  \l v_2\r^{2+\gamma_3}    
		+ \l v_1 \r^{\gamma_3} \l v_2\r^2 
		+ \l v_1 \r^{2} \l v_2\r^{\gamma_3} ) dv dv_{1,2}\nonumber \\
& \le   2C_{\gamma_2} \|b_2\| 
	\|f-g\|_{L^1_2}  \left(\|f\|_{L^1_{2+\gamma_2}} + \|g\|_{L^1_{2+\gamma_2}} \right)\nonumber \\
& \quad + 2C_{\gamma_3} \|b_3\| 
	\|f-g\|_{L^1_2} \left( \|f\|_{L^1_{2}} + \|g\|_{L^1_{2}}\right)
	 \,\left(\|f\|_{L^1_{2+\gamma_3}} + \|g\|_{L^1_{2+\gamma_3}} \right)	\nonumber \\
& \le \tild{C} \|f-g\|_{L^1_2} \,\left(1+ \|f\|_{L^1_{2}} + \|g\|_{L^1_{2}}\right)
	 \,\left(\|f\|_{L^1_{2+\gamma}} + \|g\|_{L^1_{2+\gamma}} \right),
\end{align}
 where $\tild{C} = 2 \max\{C_{\gamma_2} \|b_2\|,  C_{\gamma_3} \|b_3\| \}$. Let us note that estimate \eqref{estimate on bracket} is valid for any $f,g\in L^1_{2+\gamma}$.

Now, assume $f,g\in\widetilde{\Omega}[f_0]$, hence $\|f\|_{L^1_{2+2\gamma}},\|g\|_{L^1_{2+2\gamma}}\leq \max\{A_{2+2\gamma},m_{2+2\gamma}[f_0]\}$. Then, monotonicity of moments implies
$$ [Q[f] - Q[g], \, f-g] \leq C_L\|f-g\|_{L_2^1},$$
for some appropriate constant $C_L>0$, depending on $\max\{A_{2+2\gamma},m_{2+2\gamma}[f_0]\}$. Condition \textit{(2)} is proved.

{\it Proof of condition (3):}
First we bound the collision frequency for the binary-ternary Boltzmann equation. For $f\in\Omega[f_0]$, one can represent $Q[f]$ in gain and loss form as follows:
\begin{align}\label{gain and loss form}
Q[f] & = Q^+_2(f,f) + Q^+_3(f,f,f) - f \left( \nu_2(f) + 3\nu_3(f,f)\right)  =: Q^+[f] - f\nu [f],
\end{align}
where $Q^+[f]\geq 0$ and 
\begin{align*}
\nu_2(f) &= \int_{\mathbb{S}^{d-1}\times\mathbb{R}^d} |u|^{\gamma_2}b_2(\hat{u}\cdot\omega) f_1 \,d\omega\, dv_1=\|b_2\|\int_{\mathbb{R}^d}|u|^{\gamma_2}f_1\,dv_1, \\
\nu_3(f,f) &= \int_{\mathbb{S}^{2d-1}\times\mathbb{R}^{2d}} |\bm{\widetilde{u}}|^{\gamma_3-\theta_3}|\bm{u}|^{\theta_3}b_3(\bm{\bar{u}}\cdot\bm{\omega},\omega_1\cdot\omega_2) f_1 f_2  \,d\bm{\omega}dv_{1,2}= \|b_3\|\int_{\mathbb{R}^{2d}}|\bm{\widetilde{u}}|^{\gamma_3-\theta_3}|\bm{u}|^{\theta_3}f_1f_2\,dv_{1,2}.
\end{align*}
By \eqref{binary potential upper bound},  monotonicity of moments, and the fact that $f\in \Omega[f_0]$,  we  have
\begin{align} \label{binary nu}
\nu_2(f) &\leq C_{\gamma_2}\|b_2\|\int_{\mathbb{R}^d}(\l v\r^{\gamma_2}+\l v_1\r^{\gamma_2})f_1 \, dv_1 = C_{\gamma_2}\|b_2\|(m_{\gamma_2}+m_0|v|^{\gamma_2})\leq  C_2(1+|v|^{\gamma_2}),
\end{align}
for some constant $C_2>0$, depending on $\gamma_2,m_0,m_2,\|b_2\|$. Similarly, by \eqref{ternary potential upper bound} we have
\begin{align*}
\nu_3(f,f)\leq C_{\gamma_3}\|b_3\|\int_{\mathbb{R}^{2d}}(\l v\r^{\gamma_3}+\l v_1\r^{\gamma_3}+\l v_2\r^{\gamma_3})f_1f_2\,dv_{1,2}=C_{\gamma_3}\|b_3\|(2m_0m_{\gamma_3}+m_0^2|v|^{\gamma_3})\leq C_3(1+|v|^{\gamma_3}),
\end{align*}
for some constant $C_3>0$ depending on $\gamma_3,m_0,m_2,\|b_3\|$.
Combining estimates for $\nu_2$ and $\nu_3$, we get
\begin{align}\label{nu est}
\nu[f] = \nu_2(f) + 3\nu_3(f,f) \le C (1+ \l v\r^{\gamma_2}+ \l v\r^{\gamma_3}),
\end{align}
for some constant $C>0$, depending on $\gamma_2,\gamma_3,m_0,m_2,\|b_2\|,\|b_3\|$.

Now, in order to prove the sub-tangent condition \eqref{sub-tangent condition} for $f\in\Omega[f_0]$, it suffices to prove that
 for any $\e >0$ there exists $h^*=h^*(\e)>0$ so that for any $h\in(0,h^*)$
\begin{align*}
B_{L_2^1}(f+hQ[f], \, h \e) \cap \Omega[f_0] \neq  \emptyset.
\end{align*}
Let $0<h<1$ and $R>1$. We define
\begin{align*}
f_R & := \mathds{1}_{|v|\leq R} f, \quad
W_{R,h}  = f + h Q[f_R].
\end{align*}
First notice that, since $f\in\Omega[f_0]\subset L^1_{2+2\gamma}$ and $f_R$ is compactly supported, we have $W_{h,R}\in L_{2+2\gamma}^1$.  Our goal is to choose $R$ large enough and $h$ small enough so that $W_{R,h} \in B_{L^1_2}(f+hQ[f], \, h \e) \cap \Omega[f_0].$  We achieve that in the following steps:
\begin{itemize}
\item Given $R>0$, there exists $h_1=h_1(R)$ such that for any $h\in(0,h_1)$ we have $W_{R,h} \ge 0.$ Indeed, if $\l v \r>R$, we have $f_R=0$, thus $W_{R,h}=f\geq 0$ since $f\in\Omega[f_0]$. If $\l v \r\leq R$, we have $f_R=f$, thus decomposition \eqref{gain and loss form}, positivity of $Q^+$, and bound \eqref{nu est} imply
\begin{align*}
W_{R,h} &= f+ h Q^+[f] - h f \nu [f] \ge f (1- h \nu[f]) \ge f\left( 1- h C  \left(1+R^{\gamma_2}+ R^{\gamma_3}\right)\right),
\end{align*}
where $C$ is the constant appearing in \eqref{nu est}.
Therefore,  defining $h_1=h_1(R)=\frac{1}{C(1+R^{\gamma_2}+R^{\gamma_3})}$, we have $W_{R,h}\geq 0$ for all $h\in(0,h_1)$.

\item Since $f_R$ is compactly supported, by \eqref{averaging of mass}, \eqref{averaging of energy}, we have $$\int_{\mathbb{R}^d}Q[f_R]\,dv=\int_{\mathbb{R}^d}Q[f_R]\l v\r^2\,dv=0.$$
Therefore, since $f\in\Omega[f_0]$, we have $m_0[W_{R,h}]=m_0[f_0]$ and $m_2[W_{R,h}]=m_2[f_0]$.

\item Finally, since $f_R$ is compactly supported, bound \eqref{power estimate q}  and the fact that $m_0[f_R]\to m_0$, $m_2[f_R]\to m_2$ as $R\to\infty$,    yield the estimate
\begin{align}\label{pre-differential inequality existence}
\int_{\mathbb{R}^d}&Q[f_R](v)\langle v\rangle^{2+2\gamma}\,dv\nonumber \\
	& \leq C_{2+2\gamma}( m_0[f_R], m_2[f_R])m_{2+2\gamma
	}[f_R]-\tild{C}_{2+2\gamma}(m_0[f_R], m_2[f_R]) m_{2+2\gamma}[f_R]^{3/2} \nonumber \\
	&  \leq 2C_{2+2\gamma}( m_0, m_2)m_{2+2\gamma}[f_R]- \frac{1}{2}\tild{C}_{2+2\gamma}( m_0, m_2)m_{2+2\gamma}[f_R]^{3/2},
\end{align}
for $R>R^*$, where $R^*$ is large enough.
Consider the mapping $\L : [0,\infty) \rightarrow \R$, defined by
\begin{align}
\L(x) =  2C_{2+2\gamma}( m_0, m_2) x-\frac{1}{2}\tild{C}_{2+2\gamma}( m_0, m_2) x^{3/2}.
\label{L}
\end{align} 
 Besides zero, the map $\mathcal{L}$ has a unique positive root 
 $$x^*=\left(\frac{4C_{2+2\gamma}(m_0,m_2)}{\tild{C}_{2+2\gamma}(m_0,m_2)}\right)^2,$$
  where it changes from positive to negative, and a global maximum 
 $$\L^*=\frac{8}{27}\left(\frac{4C_{2+2\gamma}(m_0,m_2)}{\tild{C}_{2+2\gamma}(m_0,m_2)}\right)^2 C_{2+2\gamma}(m_0,m_2),$$
  achieved in $[0,x^*]$.  We define $A_{2+2\gamma}$ as follows:
\begin{align} \label{Ck*}
A_{2+2\gamma} &:= x^* + \L^*=\left(\frac{4C_{2+2\gamma}(m_0,m_2)}{\tild{C}_{2+2\gamma}(m_0,m_2)}\right)^2\left(1+\frac{8}{27}C_{2+2\gamma}(m_0,m_2)\right).
\end{align}

\noindent We now show that, for appropriate $h,R$, we have $m_{2+2\gamma}[W_{h,R}]\leq \max\{A_{2+2\gamma},m_{2+2\gamma}[f_0]\}$. Indeed,   for $R$ large enough,  \eqref{pre-differential inequality existence}-\eqref{L} imply
\begin{align*}
\int_{\R^d} Q[f_R] \l v \r^{2+2\gamma} \, dv \le \L(m_{2+2\gamma}[f_R]) \le \L^*,
\end{align*}
so if  $m_{2+2\gamma}[f] \le x^*$, we have
\begin{align*}
m_{2+2\gamma}[W_{h,R}]  \le x^* + h \int_{\R^d} Q[f_R] \l v\r^{2+2\gamma} \, dv \le x^* + h \L^* < x^*+\mathcal{L}^*= A_{2+2\gamma}.
\end{align*}
If, on the other hand, $m_{2+2\gamma}[f] > x^*$, we choose $R^{**}$ large enough so that  $m_{2+2\gamma}[f_{R}] > x^*$ for all $R>R^{**}$. Hence, $\L(m_{2+2\gamma}[f_R]) \le 0,$ which yields
 $$m_{2+2\gamma}[W_{h,R}]  \le m_{2+2\gamma}[f] \le \max\{ A_{2+2\gamma},m_{2+2\gamma}[f_0]\}.$$
\end{itemize}
We conclude that for $R>\max\{R^*,R^{**}\}$ and $h<h_1(R)$, we have $W_{h,R}\in\Omega[f_0]$. 

\begin{itemize}
\item  Moreover, by  H\"older continuity in $\widetilde{\Omega}[f_0]$, we have
\begin{align*}
\frac{\| f + hQ[f] - W_{h,R}\|_{L^1_2}}{h}  = \|  Q[f]  - Q[f_R] \|_{L^1_2} \le C_H \|f-f_R\|^{1/2}_{L^1_2} \le \e,
\end{align*}
for $R\geq R(\e)$ sufficiently large. For such $R$, we have $W_{h,R} \in B_{L_2^1}(f + h Q[f],\, h\e).$
\end{itemize}
Finally, let $R = \max\{ R^*,R^{**}, R(\e)\}$ and 
$h_1 =  \frac{1}{C\left(1+R^{\gamma_2} + R^{\gamma_3}\right) },$
where $C$ is determined by \eqref{nu est}. Then $W_{h,R} \in B(f +hQ[f], \, h\e) \cap \Omega$, for all $h\in (0,h_1)$, and condition \textit{(3)} follows.

By  Theorem \ref{thm - banach} we conclude that there exists a unique strong solution $f\geq 0$ (since $f(t)\in \Omega[f_0]$ for all $t\in[0,T])$) to the binary-ternary Boltzmann equation \eqref{binary-ternary equation} with $f\in C([0,T],\Omega[f_0]\cap C^1((0,T),L_2^1)$.  Note that the conservation of mass and energy hold by the definition of $\Omega[f_0]$, while the conservation of momentum holds due to collision averaging \eqref{averaging of momentum} which can be applied since $f\in C([0,T], L^1_{2+2\gamma})$. Moreover, by Theorem \ref{finiteness of moments}, $f \in  C^1((0,T), L^1_k)$ for any $k>2$, which completes the proof of Proposition \ref{lemma on existence}.
\end{proof}

Now, we will prove Theorem \ref{gwp theorem intro} by relaxing the assumption on the initial data to $f_0\in L^1_{2+\varepsilon}$, where $\varepsilon>0$. Inspired by the  relaxation of initial data  argument for the classical Boltzmann equation in \cite{alga22},  we will rely on the generation and propagation of polynomial moments (Theorem \ref{polynomial moments theorem}). 

\begin{proof}[Proof of Theorem \ref{gwp theorem intro}] Assume that $f_0\in L^1_{2+\varepsilon}$, where $\varepsilon>0$. Without loss of generality, we may assume that  $\varepsilon<\gamma$. Let $(f_0^j)_j$ be a sequence such that $f_0^j\in \Omega[f_0^j]\subset L^1_{2+2\gamma}$ with $f_0^j\overset{L^1_{2+\varepsilon}}\longrightarrow f_0$. Such a sequence exists, take for instance $f_0^j=\mathds{1}_{\l v\r<j}f_0$.  Let $f^j\in C([0,T],\Omega[f_0^j])\cap C^1((0,T],L^1_2)$ be the solution of equation \eqref{binary-ternary equation}  with initial data $f_0^j$ obtained by Proposition \ref{lemma on existence}. We aim to construct the solution by taking the limit of $f^j$ as $j\to\infty$. To do that, we will first show that for fixed $\delta<\varepsilon$, the sequence $(f^j)_j$ converges to some function $f$ in $C([0,T],L^1_{2+\delta})$.

Note that, by a standard regularization argument, for any $j, l \in \mathbb{N}$ we have
\begin{align}\label{equation j}
\frac{d}{dt} \| f^j(t) - f^l(t) \|_{L^1_2} 
	&= [ Q[f^j(t)] - Q[f^l(t)],  f^j(t) - f^l(t)],
\end{align}
where the bracket notation is defined in \eqref{lip inequality}. Then from \eqref{estimate on bracket}, we have
\begin{align}\label{inequality j}
\frac{d}{dt} \| f^j(t) - f^l(t) \|_{L^1_2} 
	&\le\mathcal{A}(t)   \| f^j(t) - f^l(t) \|_{L^1_2},
\end{align}
where 
\begin{align}\label{new A}
\mathcal{A}(t) = \tild{C}\left(1+ \|f^j_0\|_{L^1_{2}} + \|f^l_0\|_{L^1_{2}}\right)
	 \,\left(\|f^j(t)\|_{L^1_{2+\gamma}} + \|f^l(t)\|_{L^1_{2+\gamma}} \right),
\end{align} 
and $\tild{C}>0$ is a constant that depends on $b_2, b_3, \gamma_2, \gamma_3$.
In order to estimate $\mathcal{A}(t) $,  we use interpolation to obtain
\begin{equation}\label{interpolation bound delta}
\|f^{k}(t)\|_{L^1_{2+\gamma}}\leq \|f^k(t)\|^{\frac{\delta}{\gamma}}_{L^1_{2+\delta}}\|f^k(t)\|^{\frac{(\gamma-\delta)}{\gamma}}_{L^1_{2+\gamma+\delta}}, \quad k=j,l.
\end{equation}
Since $f_0^j\to f_0$ in $L^{1}_{2+\varepsilon}$ as $j\to\infty$, the propagation estimate \eqref{polynomial propagation estimate} yields
\begin{equation}\label{propagation estimate relaxation}
\|f^j(t)\|_{L^1_{2+\delta}}\leq \|f^j(t)\|_{L^1_{2+\varepsilon}}\leq M,\quad\forall t\in [0,T),
\end{equation}
where $M$ is a constant, uniform in $k$, that depends on $f_0, \gamma, \varepsilon, b_2$, and $b_3$.
 Invoking the generation estimate \eqref{polynomial generation estimate gamma_i>0} as well, for all $0<t\leq T$, we have
\begin{align}\label{estimate on j moment}
	\|f^k(t)\|_{L^1_{2+\gamma}}\leq M^{\frac{\delta}{\gamma}}\|f^k(t)\|_{L^1_{2+\gamma+\delta}}^{\frac{\gamma-\delta}{\gamma}}
		\le  M^{\frac{\delta}{\gamma}}K \left(1+t^{- \frac{\gamma+\delta}{\gamma}}\right)^{\frac{\gamma-\delta}{\gamma}}
		\le  M^{\frac{\delta}{\gamma}}K  \left(1+t^{- \theta}\right),
\end{align}
where $\theta=\frac{\gamma^2-\delta^2}{\gamma^2}\in (0,1)$ since $0<\delta<\varepsilon<\gamma$, and $K$ is a constant, uniform in $k$, that depends on $  \varepsilon, \gamma_2, \gamma_3, b_2$, and $b_3$.
 Therefore, for all $0<t\le T$, we have
\begin{align}\label{estimate on A}
	\mathcal{A}(t)  \le  C \left(1+t^{- \theta}\right),
\end{align}
for some uniform constant $C$. Now,  \eqref{inequality j}, \eqref{estimate on A} and Gronwall's inequality yield
\begin{align}\label{Gronwall}
\|f^j(t)-f^l(t)\|_{L^1_2}
	&\leq \|f_0^j-f_0^l\|_{L^1_2} \exp\left(\frac{C}{1-\theta}  \left( T+ T^{1-\theta} \right) \right),\quad\forall t\in[0,T],
\end{align}
since $\theta\in (0,1)$. Hence, interpolating again and using the right hand side inequality of \eqref{propagation estimate relaxation}, for all $t\in[0,T]$, we have
\begin{align}
\|f^j(t)-f^l(t)\|_{L^1_{2+\delta}}&\leq \|f^j(t)-f^l(t)\|_{L^1_2}^{\frac{\varepsilon-\delta}{\varepsilon}}\left(\|f^{j}(t)\|_{L^1_{2+\varepsilon}}+\|f^{l}(t)\|_{L^1_{2+\varepsilon}}\right)^{\delta/\varepsilon}\nonumber\\
&\le (2C_0)^{\delta/\varepsilon}\|f_0^j-f_0^l\|_{L^1_2}^{\frac{\varepsilon-\delta}{\varepsilon}} \exp\left(\frac{C (\varepsilon-\delta)}{\varepsilon(1-\theta)} \left( T+ T^{1-\theta} \right) \right).\label{convergence j,l}
\end{align}
Since $f_0^j\to f_0$ in $L^1_{2+\varepsilon}$, bound \eqref{convergence j,l} implies that $(f^j)_j$ is a Cauchy sequence in $C([0,T], L^1_{2+\delta})$, thus it converges to some $f\in  C([0,T], L^1_{2+\delta})$.  Clearly, $f\geq 0$ and $f(t=0)=f_0$, and conservation laws hold.

\noindent  Next, we show that for arbitrarily small $t_0\in (0,T)$, we have that $Q[f^j],Q[f]\in C([t_0,T],L^1_2)$ and $Q[f^j]\to Q[f]$ in $C([t_0,T],L^1_2)$. Indeed fix such a $t_0$. By Proposition \ref{lemma on existence}, for any $j\in\mathbb{N}$, we have that $f^j\in C^1((0,T],L^1_{2+\gamma})$, thus Lemmata \ref{lemma on regularity of collisional operator binary}-\ref{lemma on regularity of collisional operator ternary} imply that $Q[f^j]\in C([t_0,T], L^1_{2})$. 
Now, for $j,l\in\mathbb{N}$, and  $t\in[t_0,T]$, estimates \eqref{pre Holder} and \eqref{estimate on j moment} and the triangle inequality yield
\begin{align}
&\left\|Q[f^j(t)]-Q[f^l(t)]\right\|_{L^1_2}\nonumber\\
&\leq C\|f^j(t)-f^l(t)\|_{L^1_{2+\gamma}}+C\left(\|f^j(t)\|_{L^1_{2+\gamma}}+\|f^l(t)\|_{L^1_{2+\gamma}}\right)\|f^j(t)-f^l(t)\|_{L^1_0}\nonumber\\
&\leq C\|f^j(t)-f^l(t)\|_{L^1_{2+\delta}}^{\frac{\delta}{\gamma}}\|f^j(t)-f^l(t)\|_{L^1_{2+\gamma+\delta}}^{\frac{\gamma-\delta}{\gamma}}+C\left(\|f^j(t)\|_{L^1_{2+\gamma}}+\|f^l(t)\|_{L^1_{2+\gamma}}\right)\|f^j(t)-f^l(t)\|_{L^1_0}\nonumber \\
&\leq C(1+t^{-\theta})\left(\|f^j(t)-f^l(t)\|_{L^1_{2+\delta}}^{\frac{\delta}{\gamma}}+\|f^j(t)-f^l(t)\|_{L^1_0}\right)\nonumber\\
&\leq C(1+t_0^{-\theta})\left(\|f^j-f^l\|_{C([t_0,T],L^1_{2+\delta})}^{\frac{\delta}{\gamma}}+\|f^j-f^l\|_{C([t_0,T],L^1_0)}\right)\label{estimate on difference of Q}.
\end{align}
Since the sequence $(f^j)_j$ converges in $C([0,T],L^1_{2+\delta})$, estimate \eqref{estimate on difference of Q} implies that the sequence $\left(Q[f^j]\right)_j$ is Cauchy in $C([t_0,T],L_2^1)$ so it converges to some element $\alpha\in C([t_0,T],L_2^1)$. At the same time,   Lemmata \ref{lemma on regularity of collisional operator binary}-\ref{lemma on regularity of collisional operator ternary}  imply that  $Q[f^j]\to Q[f]$ in $C([t_0,T],L^1_{2+\delta-\gamma})$. Therefore, we conclude that $\alpha=Q[f]$, so $Q[f]\in C([t_0,T],L^1_2)$ and $Q[f^j]\to Q[f]$ in $C([t_0,T],L^1_2)$.

\noindent Now we show that $f$ is a solution to \eqref{binary-ternary equation} with initial data $f_0$. First, we have already shown that $f\in C([0,T],L_2^1)$ and that $f(t=0)=f_0$. Moreover, $f$ satisfies conservation laws \eqref{conservation laws} since $(f^j)_j$ does as well.  Also $f\in L^1_{loc}((0,T),L^1_{2+\gamma})$ Namely, estimates  \eqref{pre Holder} and \eqref{estimate on difference of Q} are also valid if  the collision operator $Q$ is replaced with the loss operator $L$. By the same reasoning applied to $Q$ one can conclude that  $L[f] \in C([t_0,T],L^1_2)$, and therefore $L[f] \in L^1([t_0,T],L^1_2)$. Thus, applying Lemma \ref{lower bound} we get
\begin{align*}
 \infty & >  \int_{t_0}^T \left\| L[f(t)]\right\|_{L^1_2} dt \\
		& = \int_{t_0}^T\left( \|b_2\| \int_{\R^{2d}} f(t,v) f(t,v_1) |v-v_1|^{\gamma_2} \l v \r^2 dv_1 dv 
				+\|b_3\| \int_{\R^{3d}} f(t,v) f(t,v_1)f(t,v_2)  |\tild{\bm{u}}|^{\gamma_3} \l v \r^2 dv_{1,2} dv\right) dt \\
		& \ge \int_{t_0}^T\left( \|b_2\| C_T  \int_{\R^d} f(t, v) \l v \r^{2+\gamma_2} dv 
			+\|b_3\| \int_{\R^{3d}} f(t,v) f(t,v_1)f(t,v_2)  |v - v_2|^{\gamma_3} \l v \r^2 dv_2 dv_1 dv\right) dt \\
		& \ge \int_{t_0}^T\left( \|b_2\| C_T  \int_{\R^d} f(t, v) \l v \r^{2+\gamma_2} dv 
			+C_T\|b_3\| \int_{\R^{3d}} f(t,v) f(t,v_1) \l v \r^{2+\gamma_3} dv_1 dv\right) dt \\
		& = \int_{t_0}^T \left(  \|b_2\| C_T  \int_{\R^d} f(t, v) \l v \r^{2+\gamma_2} dv 
			+C_T\|b_3\| \|f\|_{L^1_0} \int_{\R^{3d}} f(t,v)  \l v \r^{2+\gamma_3} dv\right) \\
		&  \ge C(T) \int_{t_0}^T   \int_{\R^d} f(t, v) \l v \r^{2+\gamma} dv dt,
\end{align*}
where $C(T)=  \|b_2\| C_T $ if $\gamma_2 = \gamma$ and $C=  \|b_3\| C_T \|f\|_{L^1_0} $ if $\gamma_3 = \gamma$. If $\gamma_2 = \gamma_3=\gamma$, one can set $C(T)$ to be either of the two values.  Since $t_0 \in (0,T)$ was arbitrary, we conclude  $f\in L^1_{loc}((0,T),L^1_{2+\gamma})$.

Finally, since $f^j$ solves \eqref{binary-ternary equation}, for arbitrarily small $t_0\in (0,T]$  we have
$$f^j(t)=f^j(t_0)+\int_{t_0}^tQ[f^j(\tau)]\,d\tau,\quad\forall t\in[t_0,T].$$
Letting $j\to\infty$, and using the fact that for any $t_0\in (0,T)$, $Q[f^j]\to Q[f]$ in $C([t_0,T],L^1_2)$, we obtain
\begin{equation*}
f(t)=f(t_0)+\int_{t_0}^tQ[f(\tau)]\,d\tau,\quad\forall \, t\in [t_0,T],
\end{equation*}
thus differentiating, we obtain $\partial_t f=Q[f]$ for all $t\in [t_0,T]$ and $f\in C^1([t_0,T],L^1_2)$. 
 Since $t_0$ was chosen arbitrarily small, we conclude that $\partial_t f=Q[f]$ for all $t\in (0,T]$ and  $f\in C^1((0,T],L^1_2)$.

\noindent Uniqueness of the solution follows by a similar Gronwall's inequality type of  argument. Finally, by Theorem \ref{finiteness of moments}, $f \in  C^1((0,T), L^1_k)$ for any $k>2$, which completes the proof of  Theorem \ref{gwp theorem intro}.
\end{proof}

\appendix
\section{}
\subsection{Multilinear collisional operators and weak formulation}\label{appendix:multilinear}
Here we recall the notation introduced in Section \ref{sec:binary-ternary equation}. 
\subsubsection*{The generalized binary collisional operator} 
The generalized binary collisional operator is given by
\begin{equation}\label{binary multilinear}
Q_2(f,g)=Q_2^+(f,g)-Q_2^-(f,g),
\end{equation}
where the gain and loss operators $Q_2^+,Q_2^-$ are given by
\begin{align}
Q_2^+(f,g)&=\frac{1}{2}\int_{\mathbb{S}_1^{d-1}\times\mathbb{R}^d}B_2(u,\omega)\left(f'g_{1}'+f_1'g'\right)\,d\omega\,dv_1,\label{binary multilinear gain}\\
Q_2^-(f,g)&=\frac{1}{2}\int_{\mathbb{S}_1^{d-1}\times\mathbb{R}^d}B_2(u,\omega)\left(fg_{1}+f_1g\right)\,d\omega\,dv_1,\label{binary multilinear loss}
\end{align}
 $u=v_1-v$ and $B_2$ is given by \eqref{cross-section 2}. Notice that for $f=g$ in \eqref{binary multilinear}, one recovers \eqref{binary collisional operator}.
Clearly $Q_2$ is symmetric and bilinear and the following identity holds:
\begin{equation}\label{diferrence binary}
Q_2(f,f)-Q_2(g,g)=Q_2(f-g,f+g).
\end{equation}
Assuming sufficient integrability conditions for $f$ and a test function $\phi$ for all integrals involved to make sense, we have the weak formulations (see \cite{ceilpu94})
\begin{align}
\int_{\mathbb{R}^d}Q_2^-(f,g)\phi\,dv&=\frac{1}{2}\int_{\mathbb{S}_1^{d-1}\times\mathbb{R}^{2d}}B_2(u,\omega)fg_{1}\left(\phi+\phi_1\right)\,d\omega\,dv_1\,dv\label{binary weak form loss}\\
\int_{\mathbb{R}^d}Q_2^+(f,g)\phi\,dv&=\frac{1}{2}\int_{\mathbb{S}_1^{d-1}\times\mathbb{R}^{2d}}B_2(u,\omega)fg_{1}\left(\phi'+\phi_1'\right)\,d\omega\,dv_1\,dv,\label{binary weak form gain}
\end{align}
which yield
\begin{equation}\label{binary weak form}
\int_{\mathbb{R}^d}Q_2(f,g)\phi\,dv=\frac{1}{2}\int_{\mathbb{S}_1^{d-1}\times\mathbb{R}^{2d}}B_2(u,\omega)fg_{1}\left(\phi'+\phi_1'-\phi-\phi_1\right)\,d\omega\,dv_1\,dv,
\end{equation}
as well as
\begin{equation}\label{binary weak form entropy}
\int_{\mathbb{R}^d}Q_2(f,g)\phi\,dv=\frac{1}{4}\int_{\mathbb{S}_1^{d-1}\times\mathbb{R}^{2d}}B_2(u,\omega)\left(f'g_1'-fg_1\right)\left(\phi+\phi_1-\phi'-\phi_1'\right)\,d\omega\,dv_1\,dv.
\end{equation}
Although we will not use \eqref{binary weak form entropy}, it is worth mentioning, since it implies entropy dissipation for the binary collisional operator, see \cite{ceilpu94}.

\noindent A sufficient integrability condition for \eqref{binary weak form loss}-\eqref{binary weak form entropy} to hold is $f,g\in L^1_{q+\gamma_2}$ and $|\phi(v)|\leq \psi(v)\l v\r^q$, where  $q\geq 0$ and $\psi\in L^\infty$. In particular, when $f,g\in L^1_2$, we can have $\phi\in C_c(\mathbb{R}^d)$. This is justified by the following Lemma: 

\begin{lemma}\label{lemma on regularity of collisional operator binary}
Let $q\geq 0$. Then for any test function $\phi$ satisfying $|\phi(v)|\leq \psi(v)\l v\r^q$, where $\psi\in L^\infty$, the weak formulations \eqref{binary weak form loss}-\eqref{binary weak form entropy} hold. In particular, $Q_2^-,Q_2^+, Q_2:L^1_{q+\gamma_2}\times L^1_{q+\gamma_2}\to L^1_q$, and there hold the bounds:
\begin{align}
\|&Q_2^-(f,g)\|_{L^1_q},\,\|Q_2^+(f,g)\|_{L^1_q},\,\|Q_2(f,g)\|_{L^1_q}\leq C\|b_2\|(\|f\|_{L^1_{\gamma_2}}+\|g\|_{L^1_{\gamma_2}})(\|f\|_{L^1_{q+\gamma_2}}+\|g\|_{L^1_{q+\gamma_2}}),\label{continuity estimate on binary static}
\end{align}
for some constant $C>0$ depending on $\gamma_2$. Additionally, if $I\subset\R$ is a  time interval,  and $f,g\in C(\overline{I},L^1_{\gamma_2})\cap L^1(\mathring{I},L^1_{q+\gamma_2})$ then $Q_2^-(f,g),Q_2^+(f,g),Q_2(f,g)\in L^1(\mathring{I},L^1_q)$. If $f,g\in C(\overline{I},L^1_{q+\gamma_2})$, then $Q_2^-(f,g),Q_2^+(f,g),Q_2(f,g)\in C(\overline{I},L^1_q)$. 
\end{lemma}

\noindent  The proof essentially reduces to proving estimate \eqref{continuity estimate on binary static}.  To achieve that, one needs to use \eqref{binary potential upper bound}, \eqref{binary cut-off} for the loss term,  while for the gain term one needs to use \eqref{binary potential upper bound} and Lemma \ref{binary povzner} instead. 
\subsubsection*{The generalized ternary collisional operator} The properties of the ternary collisional operator have been studied for the first time in \cite{am20}, for hard ternary interactions. Here, we discuss these properties in more generality. Denoting by $S_{0,1,2}$ the set of permutations of the set $\{0,1,2\}$,  the generalized ternary collisional operator is given by
\begin{equation}\label{ternary multilinear}
Q_3(f,g,h)=Q_3(f,g,h)^+-Q_3^-(f,g,h)
\end{equation}
where the gain and loss operators $Q_3^+$, $Q_3^-$ are given by
\begin{equation}\label{ternary multilinear gain}
\begin{aligned}
Q_3^+(f,g,h)&=\frac{1}{36}\sum_{\pi\in S_{0,1,2}}\int_{\mathbb{S}_1^{2d-1}\times\mathbb{R}^{2d}}B_3(\bm{u},\bm{\omega})f_{\pi_0}^*g_{\pi_1}^*h_{\pi_2}^*\,d\bm{\omega}\,dv_{1,2}\\
&+\frac{1}{18}\sum_{\pi\in S_{0,1,2}}\int_{\mathbb{S}_1^{2d-1}\times\mathbb{R}^{2d}}B_3(\bm{u_1},\bm{\omega})f_{\pi_0}^{1*}g_{\pi_1}^{1*}h_{\pi_2}^{1*}\,d\bm{\omega}\,dv_{1,2}\\
&:=Q_3^{+(c)}(f,g,h)+2Q_3^{+(a)}(f,g,h),
\end{aligned}
\end{equation}
\begin{equation}
\begin{aligned}
Q_3^-(f,g,h)&=\frac{1}{36}\sum_{\pi\in S_{0,1,2}}\int_{\mathbb{S}_1^{2d-1}\times\mathbb{R}^{2d}}B_3(\bm{u},\bm{\omega})f_{\pi_0}g_{\pi_1}h_{\pi_2}\,d\bm{\omega}\,dv_{1,2}\\
&+\frac{1}{18}\sum_{\pi\in S_{0,1,2}}\int_{\mathbb{S}_1^{2d-1}\times\mathbb{R}^{2d}}B_3(\bm{u_1},\bm{\omega})f_{\pi_0}g_{\pi_1}h_{\pi_2}\,d\bm{\omega}\,dv_{1,2}\\
&:=Q_3^{-(c)}(f,g,h)+2Q_3^{-(a)}(f,g,h),
\end{aligned}
\end{equation}
where $\bm{u}=\binom{v_1-v}{v_2-v}$, $\bm{u_1}=\binom{v-v_1}{v_2-v_1}$ and $B_3$ is given by \eqref{cross-section 3}.  The operators $Q_3^{+(c)}, Q_3^{-(c)}$ correspond to the situation where the tracked particle is the central particle of a ternary interaction, while the operators $Q_3^{+(a)},Q_3^{-(a)}$  correspond to the situation where the tracked particle is one of the adjacent particles of the interaction. Notice that for $f=g=h$ in \eqref{ternary multilinear}, one recovers expression \eqref{ternary collisional operator}.
Clearly $Q_3$ is symmetric and trilinear and there holds the identity
\begin{equation}\label{diferrence ternary}
Q_3(f,f,f)-Q_3(g,g,g)=Q_3(f-g,f,f)+Q_3(f-g,f,g)+Q_3(f-g,g,g).
\end{equation}
Assuming sufficient integrability conditions for $f$ and a test function $\phi$ for all integrals involved to make sense, we have the weak formulations:
\begin{align}
\int_{\mathbb{R}^d}Q_3^-(f,g,h)\phi\,dv&=\frac{1}{36}\sum_{\pi\in S_{0,1,2}}\int_{\mathbb{S}_1^{2d-1}\times\mathbb{R}^{3d}}B_3(\bm{u},\bm{\omega}) f_{\pi_0}g_{\pi_1}h_{\pi_2}\left(\phi+\phi_1+\phi_2\right)\,d\bm{\omega}\,dv_{1,2}\,dv,\label{ternary weak form loss}\\
\int_{\mathbb{R}^d}Q_3^+(f,g,h)\phi\,dv&=\frac{1}{36}\sum_{\pi\in S_{0,1,2}}\int_{\mathbb{S}_1^{2d-1}\times\mathbb{R}^{3d}}B_3(\bm{u},\bm{\omega}) f_{\pi_0}g_{\pi_1}h_{\pi_2}\left(\phi^*+\phi_1^*+\phi_2^*\right)\,d\bm{\omega}\,dv_{1,2}\,dv,\label{ternary weak form gain}
\end{align}
which imply
\begin{equation}\label{ternary weak form}
\begin{aligned}
\int_{\mathbb{R}^d}Q_3(f,g,h)\phi\,dv&=\frac{1}{36}\sum_{\pi\in S_{0,1,2}}\int_{\mathbb{S}_1^{2d-1}\times\mathbb{R}^{3d}}B_3(\bm{u},\bm{\omega}) f_{\pi_0}g_{\pi_1}h_{\pi_2}\left(\phi^*+\phi_1^*+\phi_2^*-\phi-\phi_1-\phi_2\right)\,d\bm{\omega}\,dv_{1,2}\,dv
\end{aligned}
\end{equation}
as well as 
\begin{align}
\int_{\mathbb{R}^d}Q_3(f,g,h)\phi\,dv&=\frac{1}{72}\sum_{\pi\in S_{0,1,2}}\int_{\mathbb{S}_1^{2d-1}\times\mathbb{R}^{3d}}B_3(\bm{u},\bm{\omega})\big( f_{\pi_0}^*g_{\pi_1}^*h_{\pi_2}^*-f_{\pi_0}g_{\pi_1}h_{\pi_2}\big)\nonumber\\
&\hspace{4cm}\left(\phi+\phi_1+\phi_2-\phi^*-\phi_1^*-\phi_2^*\right)\,d\bm{\omega}\,dv_{1,2}\,dv\label{ternary weak form entropy}
\end{align}
Although we will not use \eqref{ternary weak form entropy}, it is worth mentioning it, since it implies entropy dissipation for the ternary collisional operator, see \cite{ampa20} for more details.

\noindent In order to show \eqref{ternary weak form loss}-\eqref{ternary weak form entropy}, for $i\in\{0,1,2\}$, we write
\begin{align*}
A_i&=\sum_{\pi\in S_{0,1,2}}\int_{\mathbb{S}_1^{2d-1}\times\mathbb{R}^{3d}}B_3(\bm{u},\bm{\omega})f_{\pi_0}g_{\pi_1}h_{\pi_2}\phi_i\,d\bm{\omega}\,dv_{1,2}\,dv\\
A^*_i&=\sum_{\pi\in S_{0,1,2}}\int_{\mathbb{S}_1^{2d-1}\times\mathbb{R}^{3d}}B_3(\bm{u},\bm{\omega})f_{\pi_0}g_{\pi_1}h_{\pi_2}\phi_i^*\,d\bm{\omega}\,dv_{1,2}\,dv
\end{align*}
\noindent It suffices to show that
\begin{align}\label{sufficient weak ternary}
\int_{\mathbb{R}^d}Q_3^{-}(f,g,h)\phi\,dv=\frac{1}{36}\left(A_0+A_1+A_2\right),\quad\int_{\mathbb{R}^d}Q_3^{+}(f,g,h)\phi\,dv=\frac{1}{36}\left(A_0^*+A_1^*+A_2^*\right).
\end{align}
By \eqref{ternary micro-reversibility full text}, we clearly have 
$$\int_{\mathbb{R}^d}Q_3^{-(c)}(f,g,h)\phi\,dv=\frac{A_0}{36},\quad \int_{\mathbb{R}^d}Q_3^{+(c)}(f,g,h)\phi\,dv=\frac{A^*_0}{36}.$$
Notice that interchanging $v$ with $v_1$, we obtain
\begin{align*}
A_1&=\sum_{\pi\in S_{0,1,2}}\int_{\mathbb{S}_1^{2d-1}\times\mathbb{R}^{3d}}B_3(\bm{u_1},\bm{\omega})f_{\pi_1}^{}g^{}_{\pi_0}h_{\pi_2}^{}\phi\,d\bm{\omega}\,dv_{1,2}\,dv\\
&=\sum_{\pi\in S_{0,1,2}}\int_{\mathbb{S}_1^{2d-1}\times\mathbb{R}^{3d}}B_3(\bm{u_1},\bm{\omega})f_{\pi_0}^{}g^{}_{\pi_1}h_{\pi_2}^{}\phi\,d\bm{\omega}\,dv_{1,2}\,dv.
\end{align*}
Also, performing the involutionary change of variables $(v,v_1,v_2)\to (v^{1*},v_1^{1*},v_2^{1*})$ and using  the micro-reversibility condition \eqref{ternary micro-reversibility full text}, we have
\begin{align*}
A_1^*&=\sum_{\pi\in S_{0,1,2}}\int_{\mathbb{S}_1^{2d-1}\times\mathbb{R}^{3d}}B_3(\bm{u_1},\bm{\omega})f^{1*}_{\pi_0}g_{\pi_1}^{1*}h_{\pi_2}^{1*}\phi_1\,d\bm{\omega}\,dv_{1,2}\,dv.
\end{align*}

Hence
$$\int_{\mathbb{R}^d}Q_3^{-(a)}(f,g,h)\phi\,dv=\frac{A_1}{36},\quad \int_{\mathbb{R}^d}Q_3^{+(a)}(f,g,h)\phi\,dv=\frac{A_1^*}{36}.$$
Finally, by symmetry of the adjacent particles, we have $A_2^*=A_1^*$ and $A_2=A_1$, and \eqref{sufficient weak ternary} follows.

\noindent A sufficient integrability condition for \eqref{binary weak form loss}-\eqref{binary weak form entropy} to hold is $f,g,h\in L^1_{q+\gamma_3}$ and $|\phi(v)|\leq \psi(v)\l v\r^q$, where  $q\geq 0$ and $\psi\in L^\infty$. In particular, when $f,g,h\in L^1_2$, we can have $\phi\in C_c(\mathbb{R}^d)$. This is justified by the following Lemma: 
\begin{lemma}\label{lemma on regularity of collisional operator ternary}
Let $q\geq 0$.  Then for any test function $\phi$ satisfying $|\phi(v)|\leq \psi(v)\l v\r^q$, where $\psi\in L^\infty$, the weak formulations \eqref{ternary weak form loss}-\eqref{ternary weak form entropy} hold. In particular, $Q_3^-,Q_3^+, Q_3:L^1_{q+\gamma_3}\times L^1_{q+\gamma_3}\times L^1_{q+\gamma_3}\to L^1_q$, and there hold the bounds:
\begin{align}
\|&Q_3^-(f,g,h)\|_{L^1_q},\,\|Q_3^+(f,g,h)\|_{L^1_q},\,\|Q_3(f,g,h)\|_{L^1_q}\nonumber\\
&\leq C\|b_3\|\left(\|f\|_{L^1_{\gamma_3}}\|g\|_{L^1_{\gamma_3}}+\|f\|_{L^1_{\gamma_3}}\|h\|_{L^1_{\gamma_3}}+\|g\|_{L^1_{\gamma_3}}\|h\|_{L^1_{\gamma_3}}\right)(\|f\|_{L^1_{q+\gamma_3}}+\|g\|_{L^1_{q+\gamma_3}}+\|h\|_{L^1_{q+\gamma_3}}),\label{continuity estimate on ternary static}
\end{align}
for some constant $C>0$ depending on $\gamma_3$.  Additionally, if $I\subset \R$ is a time interval,  and $f,g,h\in C(\overline{I},L^1_{\gamma_3})\cap L^1(\mathring{I},L^1_{q+\gamma_3})$ then $Q_3^-(f,g,h),Q_3^+(f,g,h),Q_3(f,g,h)\in L^1(\overline{I},L^1_q)$. If $f,g,h\in C(\overline{I},L^1_{q+\gamma_3})$, then $Q_3^-(f,g,h),Q_3^+(f,g,h),Q_3(f,g,h)\in C(\overline{I},L^1_q)$. 
\end{lemma}
\noindent  The proof essentially reduces to proving estimate \eqref{continuity estimate on ternary static}. To achieve that, one needs to use \eqref{cross-section 3} and \eqref{ternary potential upper bound}, \eqref{binary cut-off} for the loss term, while for the gain term one needs to use \eqref{ternary potential upper bound} and Lemma \ref{ternary povzner} instead.

\subsection{Binomial and trinomial estimates}

\begin{lemma}\label{binom-max}The following polynomial estimates hold:
\begin{enumerate}
\item[(a)] If $p>1$, then for all $x,y >0$, we have
\begin{align*}
 (x+y)^p - x^p - y^p  \le  C_{2,p} (x^{p-1} y + x y^{p-1}), 
 \end{align*}
where 
\begin{align}\label{binomial constant}
C_{2,p} = p\, \max\{ 1, 2^{p-3}\}.
 \end{align}
 
\item[(b)] If $p>2$, then for all $x,y,z >0$, we have
 \begin{align*}
  & (x+y+z)^p - x^p - y^p - z^p  \le C_{3,p}  (x^{p-1} y + x y^{p-1} + x^{p-1} z + x z^{p-1} + y^{p-1} z + y z^{p-1})
  \end{align*}
  where
  \begin{align}\label{trinomial constant}
C_{3,p} = p \max\{ 1, 2^{p-3}\} + \frac{p (p-1)}{2} \max\{1, 2^{p-4} \}.
 \end{align}
  \item[(c)] If $p\ge0$, then for all  $x,y,z \ge 0$, we have
  \begin{align}
  	& (x+y)^p \le \max\{1, 2^{p-1} \} (x^p + y^p) \label{binom est} \\
	& (x+y+z)^p \le \max\{1, 3^{p-1} \} (x^p + y^p+z^p) \label{trinom est}
  \end{align}
  \end{enumerate}
\end{lemma}
\begin{proof} ${}$
\begin{enumerate}
\item[(a)] Let $p>1$ and $x,y >0$. Since  $p-2 > -1$, we have
\begin{align*}
(x+y)^p - x^p - y^p 
	= \int_0^x \int_0^y p\, (p-1) (r+s)^{p-2} dr ds.
\end{align*}
Then using that $(r+s)^{p-2} \le A_p (r^{p-2} + s^{p-2})$, where $A_p = \max\{1, 2^{p-3} \}$, we have
\begin{align*}
(x+y)^p - x^p - y^p & \le A_p  \int_0^x \int_0^y  p \,(p-1) (r^{p-2} + s^{p-2}) dr ds
\end{align*}
After integration, we have
\begin{align*}
(x+y)^p - x^p - y^p \le   A_p (p \,x \,y^{p-1} + p \,y\, x^{p-1}) = C_{2,p}  (x^{p-1} \,y + x\, y^{p-1}).
\end{align*}

\item[(b)] Let $p>2$ and $x,y,z >0$. Since $p-3 >-1$, we have
\begin{align*}
  (x+&y+z)^p  - x^p - y^p - z^p  \\
	 & = \int_0^x \int_0^y \int_0^z p (p-1) (p-2) (r+s+t)^{p-3} dr ds dt
	 + \big( (x+y)^p - x^p -y^p\big) \\
	& \qquad 	+  \big( (x+z)^p - x^p -z^p\big) 
	 +   \big( (y+z)^p - y^p -z^p\big).
\end{align*}
Using the estimate from part (a) and that $(r+s+t)^{p-3} \le B_p (r^{p-3} + s^{p-3} + t^{p-3})$, where $B_p = \max\{1, 2^{p-4} \}$, we have
\begin{align*}
  (x+&y+z)^p  - x^p - y^p - z^p \\
  	& \le B_p  \int_0^x \int_0^y \int_0^z p (p-1) (p-2) (r^{p-3} + s^{p-3} + t^{p-3}) dr ds dt \\
	& \qquad +  C_{2,p}  (x^{p-1} y + x y^{p-1} + x^{p-1} x + x z^{p-1} + y^{p-1} z + y z^{p-1}).
\end{align*}
After integration, we  get
\begin{align}\label{last}
  (x+y+z)^p  - x^p - y^p - z^p & = p\, (p-1) \, B_p\, (x\, y \, z^{p-2} + x\, y^{p-2} z + x^{p-2} \,y\, z)  \nonumber\\
	& \quad +  C_{2,p}  (x^{p-1} y + x y^{p-1} + x^{p-1} x + x z^{p-1} + y^{p-1} z + y z^{p-1}).
\end{align}
Since  $p>2$, for any $a,b>0$ we have that $(a-b) (a^{p-2} - b^{p-2}) \ge 0$, which implies 
$a^{p-2} b + a b^{p-2} \le a^{p-1} + b^{p-1}$. Therefore,
\begin{align*}
2 (x\, y \, z^{p-2} & + x\, y^{p-2} z + x^{p-2} \,y\, z) \\
	& = (x\, y \, z^{p-2} + x\, y^{p-2} z ) + (x\, y \, z^{p-2} + x^{p-2} \,y\, z) +(x\, y^{p-2} z + x^{p-2} \,y\, z)\\
	& \le (x \,y^{p-1} + x \,z^{p-1}) + (x^{p-1} \,y + y \, z^{p-1}) + (x^{p-1} \, z + y^{p-1} \, z)\\
	& =x^{p-1} y + x y^{p-1} + x^{p-1} x + x z^{p-1} + y^{p-1} z + y z^{p-1}.
\end{align*}
Combining the last inequality with \eqref{last}, completes the proof of part (b) of the Lemma.
\end{enumerate}
\end{proof}

\begin{lemma}\label{dilation of convex function}
Consider $f: [0,+\infty) \to \mathbb{R}$ with $f(0)=0$. The following hold:
\begin{enumerate}
\item If $\mu\in [0,1]$ and $f$ is convex, then $f(\mu x)\leq\mu f(x)$ for all $x\in[0,+\infty).$
\item If $\mu\ge 1$  and $f$ is concave, then $f(\mu x)\geq\mu f(x)$ for all $x\in[0,+\infty).$
\end{enumerate}
\end{lemma}
\begin{proof}
If $x=0$ or $\mu=0,1$, the inequality is trivially satisfied.  So, ssume $x>0$ and $0<\mu<1$. By the definition of convexity for all $x_1<x_2\leq x_3< x_4\in [0,\infty)$, we have
$$\frac{f(x_2)-f(x_1)}{x_2-x_1}\leq \frac{f(x_4)-f(x_3)}{x_4-x_3}$$
Since $x>0$ and $0<\mu<1$, take $x_1=0, x_2=x_3=\mu x, x_4=x$. Then
$$\frac{f(\mu x)-f(0)}{\mu x}\leq \frac{f(x)-f(\mu x)}{(1-\mu)x},$$
so, since $f(0)=0$ and $x>0$, $0<\mu<1$
$$(1-\mu)f(\mu x)\leq \mu (f(x)-f(\mu x))\Rightarrow f(\mu x)\leq \mu f(x)$$ 
\end{proof}

\bigskip
\begin{lemma}\label{approximation lemma}[Approximation of convex functions] Let $\psi:[0,\infty)\to\mathbb{R}$ be a differentiable, increasing and convex  function. Then there exists a sequence of  functions $(\psi_n)_n$ such that:
\begin{enumerate}
\item $\psi_n$ is differentiable, increasing and convex for all $n$. 
\item $\psi_{n+1}-\psi_n$ is convex.
\item $\psi_n\nearrow \psi$ as $n\to\infty$.
\item $\psi_n-p_n$ has compact support for some appropriate sequence of first degree polynomials $(p_n)_n$.
\item If, in addition, $\psi(x) = x \phi(x)$, where $\phi$ is concave, increasing and differentiable, then for all $n$, we can write $\psi_n(x) = x\phi_n(x)$, where $\phi_n$ is concave, increasing, differentiable and $\phi_n(0)=0$.
\end{enumerate}
\end{lemma}

\begin{proof}
We define the first degree polynomial 
\begin{align}\label{p_n}
		p_n(x)=\psi'(n)x+\psi(n)-n\psi'(n),
\end{align}
and
\begin{equation} \label{psi_n}
\psi_n(x)=\begin{cases}
\psi(x),\quad x\leq n\\
p_n(x),\quad x>n
\end{cases}
\end{equation}
\begin{enumerate}
\item For any $n$, $\psi_n$ is differentiable, since $\psi$, $p_n$ are so and 
$$\lim_{x\to n^-}\frac{\psi_n(x)-\psi_n(n)}{x-n}=\lim_{x\to n^-}\frac{\psi(x)-\psi(n)}{x-n}=\psi'(n),$$
$$\lim_{x\to n^+}\frac{\psi_n(x)-\psi_n(n)}{x-n}=\lim_{x\to n^+}\frac{\psi'(n)x+\psi(n)-n\psi'(n)-\psi(n)}{x-n}=\lim_{x\to n^+}\frac{(x-n)\psi'(n)}{x-n}=\psi'(n).$$
The derivative of $\psi_n$ is given by:
$$\psi_n'(x)=\begin{cases}
\psi'(x),\quad x<n\\
\psi'(n),\quad x\geq n
\end{cases}
$$
In particular, since $\psi$ is increasing, we have $\psi'\geq 0$, thus $\psi_n'\geq 0$, therefore $\psi_n$ is increasing for all $n$. Moreover, since $\psi$ is convex, we have that $\psi'$ is increasing, thus $\psi_n'$ is increasing, therefore $\psi_n$ is convex for all $n$. 
\item For fixed $n$, we compute

$$(\psi_{n+1}'-\psi_n')(x)=\begin{cases}
\psi'(x),\quad x<n,\\
\psi'(x)-\psi'(n),\quad n\leq x<n+1\\
\psi'(n+1),\quad x\geq n+1
\end{cases},
$$
Since $\psi'$ is increasing,  we have that $\psi_{n+1}'-\psi_n'\geq 0$, so $\psi_{n+1}-\psi_n$ is convex.
\item Fix $x\in[0,\infty)$. It is immediate that $\psi_n(x)\to \psi(x)$ as $n\to\infty$. To show that $(\psi_n)_n$ is increasing, fix $n$. We have the following cases:
\begin{itemize}
\item $x\leq n$: We have $\psi_{n+1}(x)-\psi_n(x)=\psi(x)-\psi(x)=0$.
\item $n<x\leq n+1$: We have
\begin{align*}
\psi_{n+1}(x)-\psi_n(x)&=\psi(x)-p_n(x)\\
&=\psi(x)-\psi(n)-\psi'(n)(x-n)\\
&=(\psi'(\xi)-\psi'(n))(x-n),
\end{align*}
for some $\xi\in (n,x)$. Since $\psi'$ is increasing, we obtain that  $\psi_{n+1}(x)\geq \psi_n(x)$.
\item $x\geq n+1$: Since $\psi$  and $\psi'$ are increasing, we take 
\begin{align*}
\psi_{n+1}(x)-\psi_n(x)&=p_{n+1}(x)-p_n(x)\\
&=\psi'(n+1)(x-n-1)-\psi'(n)(x-n)+\psi(n+1)-\psi(n)\\
&\geq (\psi'(n+1)-\psi'(n))(x-n-1)+\psi(n+1)-\psi(n)\\
&\geq 0
\end{align*}
\end{itemize}
In any case, we have $\psi_n(x)\leq \psi_{n+1}(x)$, so $(\psi_n)_n$ is increasing.
\item It is clear that $\psi_n-p_n$ is supported in $[0,n]$.

\item Define
 \begin{equation*}
	\phi_n(x) =
		\begin{cases}
		0,& x=0\\
		\phi(x),& 0<x\leq n,\\
		\psi'(n) + \frac{\psi(n) - n \psi'(n)}{x},& x>n.
\end{cases}
\end{equation*}
Clearly, $\psi_n(x)=x\phi_n(x)$. Moreover, since $\psi = x \phi$, we have $\psi' = \phi + x \phi'$, and so $x\psi' = \psi + x^2 \phi'$. Thus, 
$ \psi(n) - n \psi'(n) = - n^2 \phi'(n) \le 0$ since $\phi$ is increasing. Therefore,
\begin{equation*}
	\phi'_n(x) =
	\begin{cases}
		\phi'(x),\quad x < n,\\
		\frac{n^2 \phi'(n)}{x^2},\quad x>n.
	\end{cases}
\end{equation*}
Also,
\begin{align*}
	& \lim_{x\to n^-}\frac{\phi_n(x)-\phi_n(n)}{x-n}
		=\lim_{x\to n^-}\frac{\phi(x)-\phi(n)}{x-n}
		=\phi'(n),\\
	& \lim_{x\to n^+}\frac{\phi_n(x)-\phi_n(n)}{x-n}
		=\lim_{x\to n^+} \frac{1}{x -n} (\psi(n)-n\psi'(n)) \left( \frac{1}{x} -\frac{1}{n}\right)
		=\lim_{x\to n^+}\frac{\psi(n)-n\psi'(n) }{-xn}=\phi'(n).
\end{align*}
Therefore, $\phi'_n \ge 0$,  and so $\phi_n$ is an increasing function for all $n$. Moreover, since $\phi$ is concave, we have that $\phi'$ is decreasing, thus $\phi_n'$ is decreasing, therefore $\phi_n$ is concave for all $n$.

\end{enumerate}
\end{proof}

\subsection{Auxiliary moment estimates}
Here we provide some auxiliary moments estimates. First we present the basic interpolation estimate which is used extensively throughout the manuscript. For the proof see e.g. \cite{alga22}.
\begin{lemma}\label{interpolation lemma} Let $0\leq s_1\leq s_2$ and $s=\tau s_1+(1-\tau)s_2$, $\tau\in[0,1]$. Then, for $f\geq 0$, we have
\begin{equation}\label{general interpolation}
m_s[f]\leq  m_{s_1}[f]^\tau m_{s_2}[f]^{1-\tau}. 
\end{equation} 
In particular,  
 when $s_2>0$, the following estimate holds
\begin{equation}\label{special interpolation}
m_{s}[f]^{\frac{1}{s}}\leq m_0[f]^{\frac{1}{s}-\frac{1}{s_2}} \, m_{s_2}[f]^{\frac{1}{s_2}}.
\end{equation}
\end{lemma}
We also use the following product of moments estimate
\begin{lemma}[\cite{taalgapa18}, Lemma A.1] \label{products decreasing lemma}
Let $\ell>0$ and $0\leq i,j,k,l\leq \ell$ such that $i+j=k+l=p$. Assume that
 $\min\{k,l\}\leq\min\{i,j\}.$ Then, given $f\geq 0$, we have
$$m_i[f]m_j[f]\leq m_k[f]m_l[f].$$
\end{lemma}

\subsection{Estimates on binary and ternary potentials}
\noindent Here we establish upper and lower bounds on  $|u|^{\gamma_2}$ and $|\bm{\tild{u}}|^{\gamma_3-\theta_3}|\bm{u}|^{\theta_3}$ which we will rely on throughout the paper: \\
\begin{itemize}

\item {\it Upper bound on  $|u|^{\gamma_2}$:} Since $\gamma_2 \geq 0$, we have
\begin{equation}\label{binary potential upper bound}
|u|^{\gamma_2}
	\leq (|v|+|v_1|)^{\gamma_2}
	\le C_{\gamma_2} (\l v\r^{\gamma_2}+\l v_1\r^{\gamma_2}),
	\quad C_{\gamma_2} = \max\{1, 2^{\gamma_2-1}\}.
\end{equation}

\item {\it Lower bound on  $|u|^{\gamma_2}$:} Since $|v| = |u - v_1|  \le  |u| + |v_1|$, for any permutation $(\pi_0,\pi_1)$ of $\{0,1\}$, we have
	$\frac{1}{2}\l v_{\pi_0}\r^2 \leq \frac{1}{2}+|u|^2+|v_{\pi_1}|^2 \leq |u|^2+\l v_{\pi_1}\r^2.$
Raising both sides of the inequality to the power $\gamma_2/2$ and using the fact that $\gamma_2\in[0,2]$, we obtain the following lower bound
 \begin{align}\label{binary potential lower bound}
 	| u|^{\gamma_2}\geq 2^{-\frac{\gamma_2}{2}} \l v_{\pi_0}\r^{\gamma_2}-\l v_{\pi_1}\r^{\gamma_2}.
 \end{align}

\item {\it Upper bound on  $|\bm{\tild{u}}|^{\gamma_3-\theta_3}|\bm{u}|^{\theta_3} $:}  Since $|\bm{u}|\leq|\bm{\tild{u}}| \le |v-v_1| + |v-v_2| + |v_1 -v_2| \le 2 (|v| + |v_1| + |v_2|)$, and $\gamma_3,\theta_3\geq 0$, we  have the following upper bound
\begin{equation}\label{ternary potential upper bound}
|\bm{\tild{u}}|^{\gamma_3}
\leq C_{\gamma_3}  (\l v\r^{\gamma_3}+\l v_1\r^{\gamma_3}+\l v_2\r^{\gamma_3}), 
\quad C_{\gamma_3} = 2^{\gamma_3} \max\{ 1, 3^{\gamma_3 -1}\}.
\end{equation}

\item{\it Lower bound on  $|\bm{\tild{u}}|^{\gamma_3-\theta_3}|\bm{u}|^{\theta_3}$:}  Note that for any permutation $(\pi_0,\pi_1,\pi_2)$ of $\{0,1,2\}$, we have
$$\sqrt{2}|v_{\pi_0}|
=\left|\binom{v_{\pi_0}}{v_{\pi_0}}\right|
\leq \left| \binom{v_{\pi_0}-v_{\pi_1}}{v_{\pi_0}-v_{\pi_2}} \right| + \left| \binom{v_{\pi_1}}{v_{\pi_2}} \right|
\leq |\bm{\tild{u}}|  +|v_{\pi_1}|+|v_{\pi_2}|,$$
and thus
	$\frac{2}{3}\l v_{\pi_0}\r^2\leq \frac{2}{3}+ |\bm{\widetilde{u}}|^2+|v_{\pi_1}|^2+|v_{\pi_2}|^2\leq|\bm{\widetilde{u}}|^2+\l v_{\pi_1}\r^2+\l v_{\pi_2}\r^2.$
Raising this inequality to the power $\gamma_3/2$ and using the fact that $\gamma_3\in[0,2]$ and $|\bm{u}|\geq\frac{1}{\sqrt{3}}|\bm{\tild{u}}|$, we obtain
\begin{align}\label{ternary potential lower bound}
	|\bm{\tild{u}}|^{\gamma_3-\theta_3}|\bm{u}|^{\theta_3}
		\ge 3^{-\frac{\theta_3}{2}}|\bm{\widetilde{u}}|^{\gamma_3}
		\geq  3^{-\frac{\theta_3}{2}}\left(
		\left(\frac{2}{3}\right)^{\frac{\gamma_3}{2}}\l v_{\pi_0}\r^{\gamma_3}-\l v_{\pi_1}\r^{\gamma_3}-\l v_{\pi_2}\r^{\gamma_3}\right).
\end{align}
\end{itemize}

\subsection{ODE theory in Banach spaces} We now present a general well-posedness theorem for ODEs in Banach spaces:

\begin{theorem}\label{thm - banach} Let $X=(X,\|\cdot\|)$ be a Banach space and $S$  a bounded, convex and closed subset of $X$. Consider a mapping $\mathcal{Q}:S\to X$ that satisfies the following properties:
\begin{enumerate}
\item H\"older continuity: There exists $\alpha\in (0,1)$ and $C_H>0$ such that for all $u,v\in S$ we have
\begin{equation}\label{holder condition}
\|\mathcal{Q}[u]-\mathcal{Q}[v]\|\leq C_H\|u-v\|^{\alpha}.
\end{equation}
\item One-sided Lipschitz condition: There exists $C_L>0$ such that for all $u,v\in S$ we have
\begin{equation}\label{lipschitz condition}
\left[ \mathcal{Q}[u]-\mathcal{Q}[v],u-v\right]\le C_L\|u-v\|,
\end{equation}
where\footnote{this limit always exists since it is the side derivative of the convex map $x\to \|x\|$.}
$$[f,g]:=\lim_{h\to 0^-}\frac{\|g+hf\|-\|g\|}{h}.$$
\item Sub-tangent condition: For any $u\in S$, we have
\begin{equation}\label{sub-tangent condition}
\lim_{h\to 0^+}\frac{\dist(u+h\mathcal{Q}[u],S)}{h}=0,
\end{equation}
where for $x\in X$ we denote $\dist(x,S)=\inf\{\|x-y\|:y\in S\}$.
\end{enumerate}
Let $T>0$. Then for any $u_0\in S$, the abstract Cauchy problem:
\begin{equation}
\begin{cases}
\partial_t u=\mathcal{Q}[u],\quad t\in (0,T)\\
u|_{t=0}=u_0,
\end{cases}
\end{equation}
has a unique solution $u\in C([0,T]),S)\cap C^1((0,T),X)$.

\end{theorem}

\noindent The above statement of the theorem and the corresponding proof can be found in \cite{alga22}, which in turn are inspired  by \cite{br05}. We note that an extension of this theorem and its proof can be found in \cite{algatr16}.

\subsection{A functional inequality}
Here, we prove a lower bound on the convolution of a function $f$ and the potential function $|\cdot|^\gamma$. Estimates of this kind have been known and used in kinetic theory for a long time, see e.g. \cite{ alga22, algath18, ar83, bo97, gapavi09, mimo09}. 

The estimate provided in the lemma below  is a general functional inequality. Unlike previous results, this estimate does not require finite entropy or zero momentum, and no assumptions are made on moments of order higher than two. However, the constant in the lower bound depends on the end-time $T$.

\begin{lemma}
Let $T>0$ and  $\gamma \in (0, 2]$. Suppose a nonnegative function $f \in C([0,T],L^1_2)$  satisfies
$$0<c_0 \le \int_{\R^d} f(t,v) dv \le \int_{\R^d} f(t,v) \l v \r^2 dv \le C_2, \quad \mbox{for all } t \in [0,T], $$
where $c_0, C_2$  are constants. Then there exists a constant $C_T = C_T(T, c_0, C_2, \gamma)>0$ so that
\begin{align}\label{lower bound}
 \int_{\R^d} f(t,v_1) |v-v_1|^\gamma \, dv_1 \ge C_T \l v\r^\gamma, \quad \mbox{for all } t\in [0,T], \, v\in \R^d.
\end{align}
\end{lemma}

\begin{proof}
Estimate \eqref{binary potential lower bound} implies that  for all $ t \in [0,T]$ and all $v \in \R^d$ we have
\begin{align} \label{all v}
\int_{\R^d} f(t,v_1) |v-v_1|^\gamma \, dv_1 
	&\ge  \int_{\R^d} f(t,v_1) \left(2^{-\frac{\gamma}{2}} \l v \r^\gamma - \l v_1\r^\gamma \right)\, dv_1 
	\ge  2^{-\frac{\gamma}{2}} c_0 \,\l v \r^\gamma - C_2.
\end{align}
Therefore, if $v$ is sufficiently large so that $\frac{1}{2} 2^{-\frac{\gamma}{2}} c_0\, \l v \r^\gamma \ge C_2$, that is
\begin{align}
	|v| \ge \left(\left(\frac{2^{1+\frac{\gamma}{2}} C_2}{c_0}\right)^{2/\gamma}-  1\right)^{1/2} =:C_*,
\end{align}
then 
\begin{align}\label{large v}
	 \int_{\R^d} f(t,v_1) |v-v_1|^\gamma \, dv_1  
	 	\ge   2^{-1-\frac{\gamma}{2}} c_0 \, \l v \r^\gamma, \quad \mbox{for }  |v| \ge C_*.
\end{align}
It remains to prove the bound \eqref{lower bound}  for $|v|\le C_*$. To achieve that, we first prove the following claim: for fixed $T>0$, $\gamma \in (0, 2]$, $C_*$ and $c_0$,
\begin{align}\label{claim}
	\exists R >0 \, \mbox{ so that } \,\forall t \in[0,T] \, \forall v \in \overline{B(0, C_*)} \,\mbox{ we have }\, \int_{\{v_1: |v-v_1| \le R\}} f(t, v_1) dv_1 \le \frac{c_0}{2}.
\end{align}
To prove the claim \eqref{claim}, suppose it is not true. Then we would have
\begin{align}
	\forall n\in \N,\, \, \exists t_n \in[0,T]\, \, \exists \bar{v}_n \in \overline{B(0, C_*)} \,\mbox{ so that }\, \int_{B(\bar{v}_n, \frac{1}{n})} f(t_n, v_1) dv_1 > \frac{c_0}{2}.
\end{align}
Up to subsequences, there exits $\bar{t} \in[0,T]$ and $ \bar{v} \in \overline{B(0, C_*)}$ so that $t_n \to \bar{t}$ and $\bar{v}_n  \to \bar{v}$ as $n\to \infty$. Using triangle inequality, we get 
\begin{align*}
\frac{c_0}{2} 
	& \le \int_{\R^d} \l v_1\r^2 \Big| f(t_n, v_1) - f(\bar{t}, v_1)  \Big|dv_1 
		+ \int_{\R^d} \chi_{B(\bar{v}_n, \frac{1}{n})(v_1) } f(\bar{t}, v_1) dv_1.
\end{align*}
Since  $t_n \to \bar{t}$ and $f \in C([0,T],L^1_2)$,  $\exists N\in \N$ so that $\forall n\ge N$, $ \int_{\R^d} \l v_1\r^2 \Big| f(t_n, v_1) - f(\bar{t}, v_1)  \Big|dv_1 < \frac{c_0}{4}$. Therefore,
\begin{align*}
 	\forall  n\ge N \,  \mbox{ we have that }  \int_{\R^d} \chi_{B(\bar{v}_n, \frac{1}{n})(v_1) } f(\bar{t}, v_1) dv_1 > \frac{c_0}{4}.
\end{align*}
Letting $n\to \infty$ and using the dominated convergence theorem, we get $0 \ge \frac{c_0}{4}$, which contradicts the fact that $c_0>0$. Thus, the proof of the claim \eqref{claim} is completed.

Now suppose $|v|\le C_*$, then by the claim \ref{claim}, there exists a constant $R = R(T,  \gamma, c_0, C_2)$ so that 
\begin{align*}
	 \int_{\{v_1: |v-v_1| \le R\}} f(t, v_1) dv_1 \le \frac{c_0}{2},
	 	\quad \forall t \in[0,T], \, \forall v \in \overline{B(0, C_*)}.
\end{align*}
Then
\begin{align}\label{large v constant}
	 \int_{\R^d} f(t,v_1) |v-v_1|^\gamma \, dv_1  
	 	\ge R^\gamma  \int_{\{v_1: |v-v_1| >R\}} f(t,v_1)  dv_1  
		\ge R^\gamma  \frac{c_0}{2}, 
		\quad \forall t \in[0,T], \, \forall v \in \overline{B(0, C_*)}.
\end{align}
Combining \eqref{all v} and \eqref{large v constant}, we get that for any $\varepsilon >0$ , $t \in[0,T]$ and $v \in \overline{B(0, C_*)}$, we have
\begin{align*}
	\int_{\R^d} f(t,v_1) |v-v_1|^\gamma \, dv_1   
		\ge \varepsilon \left(  2^{-\frac{\gamma}{2}} c_0\, \l v \r^\gamma - C_2 \right) 
			+ (1-\varepsilon) R^\gamma \,  \frac{c_0}{2}.
\end{align*}
By choosing $\varepsilon$ small enough so that $(1-\varepsilon) R^\gamma \, \frac{c_0}{2} - \varepsilon C_2 \ge 0$, i.e.
$
 \varepsilon := \frac{R^\gamma \frac{c_0}{2}}{C_2 + R^\gamma \,\frac{c_0}{2}},
$
we get that
\begin{align}\label{small v}
	\int_{\R^d} f(t,v_1) |v-v_1|^\gamma \, dv_1 \ge   2^{-\frac{\gamma}{2}} c_0 \, \frac{R^\gamma \frac{c_0}{2}}{C_2 + R^\gamma \frac{c_0}{2}} \l v \r^\gamma, \qquad \forall t \in[0,T], \, \forall v \in \overline{B(0, C_*)}.
\end{align}
The lemma follows from estimates \eqref{large v} and \eqref{small v}, with 
$C_T= \max\{ 2^{-1-\frac{\gamma}{2}}c_0, \, c_0 2^{-\frac{\gamma}{2}} \frac{R^\gamma \frac{c_0}{2}}{C_2 + R^\gamma \frac{c_0}{2}}\}$.
\end{proof}


\begin{thebibliography}{99}

 \bibitem{alcagamo13} R. Alonso, J. Ca{\~n}izo, I. M. Gamba, C. Mouhot,
 {\em A new approach to the creation and propagation of exponential moments in the Boltzmann equation.}
 Comm. Partial Differential Equations 38 (2013), no. 1, 155–169.
 
 \bibitem{alga07} R. Alonso, I. M. Gamba, 
 {\em Propagation of $L^1$ and $L^\infty$ Maxwellian weighted bounds for derivatives of solutions to solutions to the homogeneous elastic Boltzmann equation.}
  J. Math. Pures. Appl., 9 (2007), pp. 575–595.
 
  \bibitem{alga22} R. Alonso, I. M. Gamba, 
 {\em The Boltzmann equation for hard potentials with integrable angular transition: coerciveness,  exponential tails rates, and Lebesgue integrability.}
Preprint, 2022.

  \bibitem{algath18} R. J. Alonso, I. M. Gamba, S. H. Tharkabhushanam,
  {\em Convergence and error estimates for the {L}agrangian-based conservative spectral method for {B}oltzmann equations.}
SIAM J. Numer. Anal. 56 (2018), no. 6, 3534–3579. 
 
 \bibitem{algatr16} R. Alonso, I. M. Gamba, M.-B. Tran,
 {\em The Cauchy problem and BEC stability for the quantum Boltzmann-Condensation system for bosons at very low temperature.}
 arXiv:1609.07467.
 
\bibitem{allo10} R. J. Alonso, B. Lods, 
{\em Bertrand Free cooling and high-energy tails of granular gases with variable restitution coefficient. }
SIAM J. Math. Anal. 42 (2010), no. 6, 2499–2538. 

 \bibitem{am20} I. Ampatzoglou, 
{\em Higher order extensions of the Boltzmann equation, Ph.D. dissertation.}
Dept. of Mathematics, UT Austin (2020).

\bibitem{ampa20} I. Ampatzoglou, N. Pavlovi\'c, 
{\em Rigorous derivation of a binary-ternary Boltzmann equation for a dense gas of hard spheres.}
arXiv:2007.00446 (2020).

\bibitem{ampa21} I. Ampatzoglou, N. Pavlovi\'c, 
{\em Rigorous derivation of a ternary Boltzmann equation for a classical system of particles.}
Comm. Math. Phys. 387 (2021), no. 2, 793–863.

\bibitem{amgapata22} I. Ampatzoglou, I. M. Gamba, N. Pavlovi\'c, M. Taskovi\'c,
{\em Global well-posedness of a binary–ternary Boltzmann equation.}
 Ann. Inst. H. Poincaré C Anal. Non Linéaire 39 (2022), no. 2, 327–369.
 
 \bibitem{ar72} L. Arkeryd, 
 {\em On the Boltzmann equation I: Existence.} 
 Arch. Ration. Mech. Anal., 45 (1972), pp. 1–16.
 
 \bibitem{ar72part2} L. Arkeryd,
 {\em On the Boltzmann equation. II. The full initial value problem. }
 Arch. Rational Mech. Anal. 45 (1972), 17–34.
 
  \bibitem{ar83} L. Arkeryd,
 {\em$L^\infty$ estimates for the space-homogeneous Boltzmann equation.}
 J. Statist. Phys. 31 (1983), no. 2, 347–361. 

\bibitem{bo97} A. V. Bobylev,
{\em Moment inequalities for the Boltzmann equation and applications to spatially homogeneous problems.}
J. Statist. Phys. 88 (1997), no. 5-6, 1183–1214.


\bibitem{bogace08} A.V. Bobylev, I.M. Gamba, C. Cercignani, 
{\em Generalized kinetic Maxwell type models of granular gases.}
Mathematical models of granular matter Series: Lecture Notes in Mathematics Vol.1937, Springer, G. Capriz, P. Giovine and P. M. Mariano (Eds.) (2008) ISBN: 978-3-540-78276-6.


\bibitem{bogace09} A. V. Bobylev, I. M. Gamba, C. Cercignani,
{\em On the self-similar asymptotics for generalized non-linear kinetic Maxwell models.}
Comm. Math. Phys.  291 (2009), no. 3, 599–644. 

\bibitem{bogapa04} A. V. Bobylev, I. M. Gamba, V. A.  Panferov, 
{\em Moment inequalities and high-energy tails for Boltzmann equations with inelastic interactions.}
J. Statist. Phys. 116 (2004), no. 5-6, 1651–1682. 

\bibitem{bo64} L. Boltzmann,
{\em Lectures on gas theory.} Translated by Stephen G. Brush, University of California Press, Berkeley-Los Angeles, Calif. Reprint of the 1896–1898 Edition. Reprinted by Dover Publications, 1995.

\bibitem{br05} A. Bressan, 
{\em Notes on the Boltzmann Equation.} 
Lecture notes for a summer course given at S.I.S.S.A., Trieste, 2005.


\bibitem{brei16} M. Briant, A. Einav, 
{\em On the Cauchy problem for the homogeneous Boltzmann-Nordheim equation for bosons: local existence, uniqueness and creation of moments.}
J. Stat. Phys. 163 (2016), no. 5, 1108–1156.

\bibitem{ce88} C.  Cercignani,
{\em  The Boltzmann equation and its applications.}
 Applied Mathematical Sciences, 67. Springer-Verlag, New York, 1988.

\bibitem{ceilpu94}
C.  Cercignani,  R.  Illner,  M.  Pulvirenti,
{\em The  Mathematical  Theory  of  Dilute  Gases.}
Springer  Verlag, New York  NY, 1994.

\bibitem{de93} L. Desvillettes, 
{\em Some applications of the method of moments for the homogeneous Boltzmann and Kac equations.} 
Arch. Ration. Mech. Anal., 123 (1993) pp. 387–404.

\bibitem{el83} T. Elmroth,  
{\em Global boundedness of moments of solutions of the Boltzmann equation for forces of infinite range.}
Arch. Ration. Mech. Anal., 82 (1983), pp. 1–12.

\bibitem{fo21} N. Fournier,
{\em On exponential moments of the homogeneous Boltzmann equation for hard potentials without cutoff.}
 Comm. Math. Phys. 387 (2021), no. 2, 973–994.
 


\bibitem{gapavi09}  I. M. Gamba,  V. Panferov, V.; C. Villani,
{\em Upper Maxwellian bounds for the spatially homogeneous Boltzmann equation.}
Arch. Ration. Mech. Anal. 194 (2009), no. 1, 253–282. 

\bibitem{gapa20}  I. M. Gamba, M. Pavi\'c-\v Colic, 
{\em On Existence and Uniqueness to Homogeneous Boltzmann Flows of Monatomic Gas Mixtures.}
Arch. Rational Mech. Anal. 235 (2020) 723–781.

\bibitem{gapa20preprint}  I. M. Gamba, M. Pavi\'c-\v Colic, 
{\em On On the Cauchy problem for Boltzmann equation modelling a polyatomic gas.}
arXiv:2005.01017 

\bibitem{lumo12} X. Lu, C. Mouhot,
    {\em On measure solutions of the {B}oltzmann equation, part {I}: moment production and stability estimates.}
   J. Differential Equations 252 (2012), no. 4, 3305–3363. 

\bibitem{ma867} J. Maxwell,
{\em On the dynamical theory of gases.}
Philos. Trans. Roy. Soc. London Ser. A 157 (1867), 49–88.

\bibitem{mimo09} S. Mischler, C. Mouhot,
{\em Stability, convergence to self-similarity and elastic limit for the Boltzmann equation for inelastic hard spheres. } Comm. Math. Phys. 288 (2009), no. 2, 431–502.

\bibitem{mimo06} S. Mischler, C. Mouhot, M. Rodriguez Ricard, 
{\em  Cooling process for inelastic Boltzmann equations for hard spheres. I. The Cauchy problem. }
J. Stat. Phys. 124 (2006), no. 2-4, 655–702. 
     
\bibitem{miwe99} S. Mischler, B. Wennberg,
{\em On the spatially homogeneous Boltzmann equation.}
 Ann. Inst. H. Poincar\' e C Anal. Non Linéaire 16 (1999), no. 4, 467–501.
 
 \bibitem{mowaya16} Y. Morimoto, S. Wang, T. Yang, 
 {\em Measure valued solutions to the spatially homogeneous Boltzmann equation without angular cutoff.}
  J. Stat. Phys. 165 (2016), no. 5, 866–906.
  
  \bibitem{mo06} C. Mouhot, 
{\em Rate of convergence to equilibrium for the spatially homogeneous Boltzmann equation with hard potentials.}
Comm. Math. Phys., 261 (2006), pp. 629–672.

  
 \bibitem{pata18}  M. Pavi\'c- \v Coli\'c, M. Taskovi\'c, 
 {\em Propagation of stretched exponential moments for the Kac equation and Boltzmann equation with Maxwell molecules.} Kinet. Relat. Models 11 (2018), no. 3, 597–613. 
  
  \bibitem{po62} A. J. Povzner, 
  {\em On the Boltzmann equation in the kinetic theory of gases.} 
  Mat. Sb. (N.S.), 58 (1962), pp. 65–86.

\bibitem{ruvo02} C. Russ, H. H. von Gr\"unberg,
{\em Three-body forces between charged colloidal particles.}
PHYSICAL REVIEW E 011402 (2002).

\bibitem{styu14} R. M. Strain, S.-B. Yun, 
{\em Spatially homogeneous Boltzmann equation for relativistic particles.}
 SIAM J. Math. Anal. 46 (2014), no. 1, 917–938. 
 
 \bibitem{taalgapa18} M. Taskovi\'c, R. J. Alonso, I. M. Gamba, N. Pavlovi\'c, 
 {\em On Mittag-Leffler moments for the Boltzmann equation for hard potentials without cutoff.} 
 SIAM J. Math. Anal. 50 (2018), no. 1, 834–869
 
 \bibitem{we97} B. Wennberg, 
 {\em Entropy dissipation and moment production for the Boltzmann equation.}
  J. Stat. Phys., 86 (1997), pp. 1053–1066.

\end{thebibliography}
\end{document}